\documentclass[11pt,reqno]{amsart}

\usepackage{amsmath,amsfonts,amsthm,amscd,amssymb,graphicx,mathrsfs}
\usepackage[utf8]{inputenc}
\usepackage{amsbsy}
\usepackage{graphicx}
\usepackage{subcaption}
\usepackage{hyperref}
\usepackage{tikz}
\usetikzlibrary{arrows}
\usepackage[text={6.5in,9in},centering,includefoot,foot=0.6in]{geometry}

\usepackage[all]{xy}

\usepackage{amsmath,amsthm,amsfonts,amssymb,amscd}
\usepackage[all]{xy}
\usepackage{sseq}

\usepackage{todonotes}

\definecolor{skyblue}{rgb}{0.85,0.85,1}

\newtheorem{theorem}{Theorem}[section]
\newtheorem{prop}[theorem]{Proposition}
\newtheorem{cor}[theorem]{Corollary}
\newtheorem{lemma}[theorem]{Lemma}
\newtheorem{rem}[theorem]{Remark}
\newtheorem{define}[theorem]{Definition}

\def\Re{\mathop\mathrm{Re}\nolimits}    % real part
\newcommand{\bbC}{\mathbb{C}}             % complex numbers
\newcommand{\bbN}{\mathbb{N}}             % natural numbers
\newcommand{\bbQ}{\mathbb{Q}}
\newcommand{\bbR}{\mathbb{R}}             % reals
\newcommand{\bbZ}{\mathbb{Z}}             % integers

\newcommand{\p}{\partial}				% partial derivative

\DeclareMathOperator{\Mas}{Mas}
\DeclareMathOperator{\Ind}{Ind}
\DeclareMathOperator{\Wind}{Wind}

\DeclareMathOperator{\sgn}{sgn}

\DeclareMathOperator{\tr}{tr}

\DeclareMathOperator{\spn}{span}
\DeclareMathOperator{\diag}{diag}

\DeclareMathOperator{\Hom}{Hom}
\DeclareMathOperator{\Pf}{Pf}

\newcommand{\cF}{{\mathcal{F}}}

\newcommand{\cH}{{\mathcal{H}}}

\newcommand{\cL}{{\mathcal{L}}}
\newcommand{\cM}{{\mathcal{M}}}
\newcommand{\cU}{{\mathcal{U}}}
\newcommand{\cZ}{{\mathcal{Z}}}

\newcommand{\R}{\mathbb{R}}
\newcommand{\Z}{\mathbb{Z}}
\newcommand{\C}{\mathbb{C}}

\newcommand{\Gr}{Gr}

\begin{document}

\title{Generalized Maslov indices for non-Hamiltonian systems}
\author[T. Baird]{Thomas John Baird}\email{tbaird@mun.ca}\address{Department of Mathematics and Statistics, Memorial University of Newfoundland, St. John's, NL, Canada}
\author[P. Cornwell]{Paul Cornwell}\email{paul.cornwell@jhuapl.edu}\address{The Johns Hopkins University Applied Physics Laboratory, Laurel, MD, USA}
\author[G. Cox]{Graham Cox}\email{gcox@mun.ca}\address{Department of Mathematics and Statistics, Memorial University of Newfoundland, St. John's, NL, Canada}
\author[C. Jones]{Christopher Jones}\email{cktrj@email.unc.edu}\address{Department of Mathematics, UNC-CH, Chapel Hill, NC, USA}
\author[R. Marangell]{Robert Marangell}\email{robert.marangell@sydney.edu.au}\address{School of Mathematics and Statistics, University of Sydney, NSW, Australia}

\maketitle
\begin{abstract}
We extend the definition of the Maslov index to a broad class of non-Hamiltonian dynamical systems. To do this, we introduce a family of topological spaces\,---\,which we call \emph{Maslov--Arnold spaces}\,---\,that share key topological features with the Lagrangian Grassmannian, and hence admit a similar index theory. This family contains the Lagrangian
Grassmannian, and much more. We construct a family of examples, called \emph{ hyperplane Maslov--Arnold spaces}, that are dense in the Grassmannian, and hence are much larger than the Lagrangian Grassmannian (which is a submanifold of positive codimension). The resulting index is then used to study eigenvalue problems for non-symmetric reaction--diffusion systems. A highlight of our analysis is a topological interpretation of the Turing instability: the bifurcation that occurs as one increases the ratio of diffusion coefficients corresponds to a change in the generalized Maslov index.
\end{abstract}

\tableofcontents

\section{Introduction}
The Maslov index is an powerful tool for distinguishing trajectories in a Hamiltonian system, and provides a natural setting for many well-known results in stability theory, such at the Morse index theorem and the Sturm oscillation theorem. Its definition depends crucially on the topology of the Lagrangian Grassmannian, and the fact that it is invariant under the flow of a Hamiltonian system.

An interesting and difficult question is whether this restriction to systems with an underlying Hamiltonian structure can be weakened in order to open up a greater range of applications. The idea we pursue in this paper is to look for subsets of the Grassmannian that have the needed topological properties. We first clarify what these topological properties are.

We take $\mathbb{R}^{2n}$ as the phase space, and denote the Grassmannian of $n$-dimensional subspaces by $\Gr_n(\bbR^{2n})$. Such an $n$-dimensional subspace is said to be Lagrangian, with respect to a given symplectic form $\omega$, if $\omega$ vanishes on it. The set of Lagrangian subspaces, known as the Lagrangian Grassmannian, is then denoted $\Lambda(n)$. The Maslov index is an integral homotopy invariant defined for continuous paths in the Lagrangian Grassmannian $\Lambda(n)$. This is well defined because $H^1(\Lambda(n);\bbZ) = \bbZ$. Moreover, the generator of $H^1(\Lambda(n);\bbZ)$ can be explicitly identified in such a way that the index of a curve is interpreted as a signed count of its intersections with a fixed Lagrangian subspace. 

Both the index being an integer and its interpretation as an intersection number are important in applications to dynamical systems. We thus seek generalizations with both of these properties. We will define a class of spaces that captures these features and call them {\em Maslov--Arnold Spaces}. First, we recall that the \emph{train} of an $n$-dimensional subspace $P \in Gr_n(\bbR^{2n})$ is the set
\begin{equation}\label{def:train1}
	\cZ_{_P} := \big\{ W \in Gr_n(\bbR^{2n}) : W \cap P \neq \{0\} \big\}
\end{equation}
of all $n$-dimensional subspaces intersecting $P$ non-trivially, and the subset
\[
    \cZ_{_P}^{1} := \big\{ W \in Gr_n(\bbR^{2n}) : \dim(W \cap P)=1 \big\} \subseteq  \cZ_{_P}
\]
is a smoothly embedded hypersurface.

As described above, the Maslov index for curves in the  Lagrangian Grassmannian has two features that make it useful in applications:
\begin{enumerate}
	\item[(1)] There exists a cohomology class $\alpha_0 \in H^1(\Lambda(n);\bbZ)$ such that the Maslov index of any continuous loop $\gamma\colon S^1 \rightarrow \Lambda(n)$ equals the canonical pairing $ \langle \alpha_0, [\gamma] \rangle \in \Z$, where $[\gamma] \in H_1(\Lambda(n); \Z)$ is the homology class represented by $\gamma$.
		\item[(2)] If $\gamma\colon S^1 \rightarrow \Lambda(n)$ is a sufficiently generic loop, then its Maslov index is equal\footnote{In Section \ref{sec:def} we will clarify this statement by giving a precise definition of a ``sufficiently generic loop" and its geometric intersection number with the train.} to its geometric intersection number with the train of a fixed Lagrangian subspace $P$.
\end{enumerate}

Motivated by these properties, we give the following definition.

\begin{define}\label{def:MA1}
A \emph{rank $n$ Maslov--Arnold (MA) space} $(\cM,P, \alpha)$ consists of
\begin{itemize}
\item  a subset $\cM \subseteq \Gr_n(\bbR^{2n})$,
\item  an $n$-dimensional subspace $P \in \Gr_n(\bbR^{2n})$, and
\item a cohomology class $\alpha \in H^1(\cM;\bbZ)$ of infinite order,
\end{itemize}
where  $\cZ_{_P}^1 \cap \cM$ has a co-orientation such that for any sufficiently generic loop $\gamma\colon S^1 \rightarrow \cM$, the geometric intersection number with $\cZ_{_P}$ equals the pairing $\langle \alpha, [\gamma] \rangle$.
\end{define}

Note that $\cM$ is not required to be a manifold, and the distinguished subspace $P$ is not required to be an element of $\cM$. We next define the generalized Maslov index.

\begin{define}
\label{def:genMas}
For any continuous loop $\gamma \colon S^1 \to \cM$, we define the \emph{generalized Maslov index of $\gamma$ with respect to $P$} by
\begin{equation}
\label{def:MI1}
	\Mas(\gamma; P) = \langle \alpha, [\gamma] \rangle.
\end{equation}
\end{define}

The definition of a Maslov--Arnold space guarantees that the generalized Maslov index of any sufficiently generic loop is equal to its geometric intersection number with the train $\cZ_{_P}$.  However, it is important to note that \eqref{def:MI1} defines the generalized Maslov index for \emph{any} continuous loop $\gamma$ in $\cM$, with no genericity assumptions needed.

In this terminology, the main result of Arnold's seminal paper \cite{A67} is that $(\Lambda(n),P,\alpha_0)$ is a Maslov--Arnold space for any $P \in \Lambda(n)$, where $\alpha_0 \in H^1(\Lambda(n);\Z)\cong \Z$ is one of the two generators. The symplectic form defining $\Lambda(n)$ determines a canonical choice of $\alpha_0$, called the \emph{Maslov class}; with this choice of generator we call $(\Lambda(n),P,\alpha_0)$ a \emph{classical Maslov--Arnold space}.
On the other hand, $\Gr_n(\bbR^{2n})$ cannot be an MA space if $n\geq 2$, because $H^1(\Gr_n(\bbR^{2n});\bbZ) \cong \bbZ_2$ contains no cohomology classes of infinite order. 
 In the case  $n=1$, it is easy to see that $\Lambda(1)= Gr_1(\bbR^{2}) =\bbR P^1\cong S^1$. This is the home of classical Sturm--Liouville theory, which is often approached through studying the angle of a path in $S^1$.

To apply our generalized index theory to dynamical systems, we must find non-trivial examples of MA spaces where the index can be computed.
A natural approach would be to ``fatten up" the Lagrangian Grassmannian to obtain a strictly larger Maslov--Arnold space. It is possible to construct such spaces (which we do in Theorem \ref{thm:fatLG}), but we show in Theorem \ref{thmcont} that no MA space exists that properly contains the Lagrangian Grassmannian and is also a smooth submanifold of $\Gr_n(\bbR^{2n})$. Therefore, in constructing MA spaces, we must make a choice: we can have a space that extends the Lagrangian Grassmannian, or is a smooth submanifold of the Grassmannian, but not both.

The main focus of this paper is a resolution to this conundrum through the construction of a large class of Maslov--Arnold spaces,
which we call \emph{Hyperplane Maslov--Arnold Spaces}, that are open, dense subsets (and hence smooth submanifolds) of the Grassmannian, but do not contain the entire Lagrangian Grassmannian. For these spaces the index has a simple geometric interpretation as a winding number in $\bbR P^1$. This gives us a practical method for computing the index, and also allows us to define it for continuous paths with distinct endpoints.

Given a nonzero $n$-form $\omega$ on $\bbR^{2n}$, consider the subset
\begin{equation}
\label{GPlucker1}
	\cH_\omega := \big\{ \spn \{v_1, \ldots ,v_n\} \in Gr_n(\bbR^{2n}) : \omega(v_1, \ldots, v_n) = 0 \big\}
\end{equation}
of the Grassmannian $Gr_n(\bbR^{2n})$. We call this the hyperplane associated to $\omega$. A special case is when $\ker \omega$ is $n$ dimensional. Then it can be shown that the associated hyperplane is the train of the subspace $\ker \omega$, i.e.
\begin{equation}\label{hyperplanetrain1}
    \cH_\omega = \big\{ W \in Gr_n(\bbR^{2n}) : W \cap \ker \omega \neq \{0\} \big\}.
\end{equation}
Another special case is when $n=2$ and $\omega$ is a non-degenerate two-form, i.e. a symplectic form, in which case $\cH_\omega \subset Gr_2(\bbR^4)$ is the corresponding Lagrangian Grassmannian.

\begin{define}
A hyperplane Maslov--Arnold space is a set
\begin{equation}
    \cM = Gr_n(\bbR^{2n}) \setminus \big(\cH_{\omega_1} \cap \cH_{\omega_2}\big),
\end{equation}
where $\omega_1$ and $\omega_2$ are linearly independent $n$-forms and $\ker \omega_1$ is $n$-dimensional.
\end{define}

That is, $\cM$ is obtained from the Grassmannian by removing the intersection of two hyperplanes, at least one of which is the train of an $n$-dimensional subspace. It will be shown below that such an $\cM$ is indeed a Maslov--Arnold space in the sense of Definition \ref{def:MA1}, where $P = \ker \omega_1$ is the distinguished subspace and the cohomology class $\alpha$ is determined by $\cH_{\omega_1}$; see Theorem~\ref{MAspacethm} for a precise statement.

For the hyperplane Maslov--Arnold spaces, our generalized Maslov index has a simple geometric interpretation. Defining a map $\phi \colon \cM \rightarrow \R P^1$ by
\begin{equation}
\label{phidef1}
    \phi\big(\spn \{v_1, \ldots ,v_n\}\big) = \big[ \omega_1(v_1,\ldots,v_n) : \omega_2(v_1,\ldots,v_n) \big],
\end{equation}
we will show that the index of a loop $\gamma \colon S^1 \to \cM$ is equal to the winding number of $\phi \circ \gamma$ in $\bbR P^1$.

From the definition of $\phi$ we see that a subspace $W \in \cM$ is contained in the train  $\cH_{\omega_1}$ if and only if $\phi(W) = [0:1]$. We thus extend our definition of the index to arbitrary paths (possibly having distinct endpoints) by simply defining it to be the winding number through the point $[0:1] \in \bbR P^1$, with a suitable convention chosen for the endpoints. This gives a well-defined index for paths in the hyperplane MA space $\cM$ that detects intersections with the train of $\ker\omega_1$, much like the classical Maslov index does for paths in the Lagrangian Grassmannian.

However, there is an important difference between our hyperplane index and the classical Maslov index. To illustrate this, suppose $\gamma(t)$ is a continuous path in $\cM$ with $\gamma(t_*) \in \cH_{\omega_1}$, i.e. $\gamma(t_*) \cap \ker \omega_1 \neq \{0\}$, for some $t_*$. This means the path $\phi \circ \gamma$ in $\bbR P^1$ has $\phi(\gamma(t_*)) = [0:1]$. Depending on the direction in which this curve passes through the point $[0:1]$, the local contribution to the winding number will be either $0$, $1$ or $-1$. That is, the largest absolute change to the index at each intersection is $1$, regardless of the dimension of the intersection $\gamma(t_*) \cap \ker \omega_1$.

In other words, while our index is sensitive to the direction of crossing, it does not measure the dimension, unlike the classical Maslov index. This point will reoccur throughout the paper; see for instance Remark \ref{rem:card}.

Using our hyperplane index theory, we prove a generalized Morse index theorem (Theorem \ref{thm:cpMorse}) relating conjugate points and positive eigenvalues for a non-selfadjoint boundary value problem. A special case is the Dirichlet eigenvalue problem on a bounded interval $(0,L)$,
\begin{equation}
\label{ex:Dir}
    u_{xx} + V(x)u = \lambda u, \quad u(0) = u(L) = 0,
\end{equation}
where $u \in \bbR^n$ and $V(x)$ is a real (but not necessarily symmetric) $n\times n$ matrix-valued potential. Assuming a certain invariance condition \eqref{xlinv} holds (which is generically the case), we find that
\begin{equation}\label{unstable:conjugate}
	\# \big\{\text{positive real eigenvalues of \eqref{ex:Dir}}  \big\} \geq \# \big\{\text{conjugate points} \big\} - \mathfrak m
\end{equation}
where $\mathfrak m \in \bbZ$ is a topological correction that can be explicitly computed in many cases of interest.

\begin{rem}
\label{rem:card}
The notation $\#$ denotes the cardinality of a set, so the left-hand side of \eqref{unstable:conjugate} is the number of distinct positive real eigenvalues, and similarly for the right-hand side. This is very different from the classical Maslov index, which counts eigenvalues and conjugate points with multiplicity. This means our hyperplane index is much simpler to work with, and is a valuable tool for detecting instability even though it does not capture information about geometric or algebraic multiplicity.
\end{rem}

The inequality \eqref{unstable:conjugate} differs from the classical Morse index theorem, and hence from Sturm--Liouville theory, in three significant ways:
\begin{enumerate}
    \item[(1)] positive eigenvalues and conjugate points are counted without multiplicity;
    \item[(2)] we obtain a lower bound, rather than an exact formula, for the number of positive eigenvalues (though in some cases this can be improved to an equality, see Lemma \ref{lem:lambdamono});
    \item[(3)] there is a topological correction $\mathfrak m$ that is not present in the usual Morse index theorem.
\end{enumerate}

The term $\mathfrak m$, which has no analogue in the classical theory, encodes non-trivial dynamical information about the system. For instance, when $\mathfrak m > 0$ it is possible to have conjugate points but no positive eigenvalues. In general, the presence of a conjugate point will only imply the existence of a positive eigenvalue if $\mathfrak m \leq 0$.

The significance of the index $\mathfrak m$ is illustrated in Section \ref{sec:Turing}. There we study the well-known Turing instability phenomenon, whereby a stable, homogeneous equilibrium of a chemical reaction is counter-intuitively destabilized in the presence of diffusion. We find that such an instability occurs if and only if the index $\mathfrak m$ is nonzero. Moreover, we find a topological mechanism for the onset of the Turing instability (the so-called \emph{Turing bifurcation}) in terms of the topology of the underlying Maslov--Arnold space.

\subsection*{Outline of paper}
In Section \ref{sec:background} we provide further motivation for our construction by describing its application to systems of reaction--diffusion equations, and contrast it with the classical Maslov index. In Section \ref{sec:eLG}, we establish relevant topological properties of Maslov--Arnold spaces, in particular the so-called \emph{hyperplane MA spaces}. In Section \ref{sec:counting} we begin to apply our theory of hyperplane MA spaces by proving a general Morse index theorem (Theorem~\ref{thm:cpMorse}), that relates conjugate points and unstable eigenvalues for non-selfadjoint operators. Next, in Section \ref{sec:app}, we look at two concrete applications of this general result: systems with large diffusion, and systems admitting homogeneous equilibria. We characterize the Turing bifurcation using our generalized index theory, and also describe possible numerical applications of our theory. Finally, in Section \ref{furex} we describe additional examples (and non-examples) of Maslov--Arnold spaces, going beyond the hyperplane spaces that are emphasized in the rest of the paper.

\section{Background and motivation}
\label{sec:background}

Many interesting physical phenomena are described by systems of reaction--diffusion equations. These have the form
\begin{equation}\label{RD}
	u_t = D u_{xx} + F(u),
\end{equation}
where $u(x,t) \in \bbR^n$, $D = \diag(d_1, \ldots, d_n)$ with all $d_i > 0$, and $F \colon \bbR^n \to \bbR^n$. Given a steady state $\bar u(x)$, i.e. a solution to $D \bar u_{xx} + F(\bar u) = 0$, it is natural to ask whether or not it is stable to small perturbations.

The linear stability of $\bar u$ is determined by the spectrum of the linearized operator
\begin{equation}\label{Ldef}
	\cL = D \frac{d^2}{dx^2} + \nabla F(\bar u).
\end{equation}
For the study of traveling waves it is natural to take the real line as the spatial domain, in which case $\cL$ would be a closed, unbounded operator on $L^2(\bbR)$, with domain $H^2(\bbR)$. The resulting eigenvalue problem and index theory on the line are more involved than for a bounded interval, but the difficulties are analytic, rather than topological, in nature. In the Hamiltonian case these issues have been satisfactorily addressed in many places, for instance \cite{BCJLMS18,CH07,HLS2,HLS}, and we expect that similar methods will work here. Therefore, in order to emphasize the relevant topology of the Maslov--Arnold spaces, which is the main purpose of this paper, we restrict our attention to problems on a bounded interval $(0,L)$, and hence will view $\cL$ as an operator on $L^2(0,L)$, with domain in $H^2(0,L)$ depending on the choice of boundary conditions.

The eigenvalue equation $\cL v = \lambda v$ can be written as a $2n \times 2n$ system
\begin{equation}\label{Lsystem}
	\frac{d}{dx} \begin{pmatrix} v \\ w \end{pmatrix} = \begin{pmatrix} 0 & D^{-1} \\ \lambda I - \nabla F(\bar u) & 0 \end{pmatrix} \begin{pmatrix} v \\ w \end{pmatrix}.
\end{equation}
If $F = \nabla G$ for some function $G\colon \bbR^n \to \bbR$, then $\nabla F = \nabla^2 G$ is symmetric, hence $\cL$ is self-adjoint, and the system \eqref{Lsystem} is Hamiltonian. In this case the state $\bar u$ has a well defined Maslov index, which can be shown to equal the number of positive eigenvalues of $\cL$.

The need for the eigenvalue equation $\cL v = \lambda v$ to be expressible as a Hamiltonian system imposes certain restrictions on the PDE \eqref{RD} under consideration. For instance, the system \eqref{Lsystem} is Hamiltonian if and only if $\nabla F(\bar u)$ is symmetric, in which case the linearized operator \eqref{Ldef} is self-adjoint. In the context of reaction--diffusion systems, this means that the nonlinearity must be of gradient type, which rules out many physically relevant models. Reaction--diffusion equations are primarily studied for their propensity to support patterns and other permanent structures, and it was was shown by Turing \cite{T52} (see Section \ref{sec:Turing} for an in-depth
discussion) that a fundamental mechanism for generating such patterns requires that $\nabla F$ has competing terms, thus ensuring that $F$ is not a gradient. In the literature, equations of the form \eqref{RD} for which a stable equilibrium can be destabilized in the presence of diffusion are called \emph{activator--inhibitor} systems.

Of course \eqref{Lsystem} is not the only way of writing the eigenvalue equation as a first order system, and there is no reason to only consider the standard symplectic structure on the phase space $\bbR^{2n}$. Yanagida \cite{Yan02a,Yan02b} initiated the study of a broad class of activator--inhibitor systems called \emph{skew-gradient}, for which $F=Q\nabla G$ with $Q=\text{diag}\{q_i\}$, $q_i=\pm 1$. Such problems can be put into Hamiltonian form by a suitable change of variables. Chen and Hu subsequently showed how to define the Maslov index of a standing wave and how to use it as a tool in stability analysis for the skew-gradient case \cite{CH07,CH14}. Cornwell and
Jones extended these ideas to traveling waves in \cite{corn,CJ18}. In both cases, the parity of the Maslov index is shown to determine the sign of the derivative of the Evans function \cite{AlGaJo90} at $\lambda =0$; cf. \cite{ChardardBridges15, CJ2017}. The results in the aforementioned works hinged on the fact that the eigenvalue equation for $\mathcal{L}$ preserves the manifold
of Lagrangian planes for a non-standard symplectic form. In contrast to the Hamiltonian case, the index might be non-monotone in its parameters, and $\mathcal{L}$ might possess complex eigenvalues. Nonetheless, a nonzero Maslov index can still be used to prove instability; cf. \eqref{unstable:conjugate}. (Jones used the same idea to prove an instability criterion for standing waves in nonlinear
Schr\"{o}dinger-type equations \cite{J88}.) The index can also be used to prove stability in a particular case if the above concerns are addressed. For example, the Maslov index was used to prove stability of both standing and traveling waves in a doubly-diffusive FitzHugh--Nagumo equation \cite{CH14,CJ18}.

In addition to skew-gradient systems, the Maslov index has also been successfully applied to other PDEs that are conservative, such as the Nonlinear Schr\"odinger equation \cite{JMS14,JMS10,J88}, and various water wave problems \cite{ChardardBridges15}. Therefore, there are many cases where the Maslov index is relevant even though the linearized operator $\cL$ is not self-adjoint. In all of these cases, however, it can be shown that there is some hidden Hamiltonian  structure in the linearized problem, see for instance \cite{corn}.

The main achievement of this paper is the definition and subsequent application of a Maslov-like index for very general systems of equations, only requiring a mild invariance condition \eqref{xlinv} to be satisfied. As already described in the introduction, we do this by introducing Maslov--Arnold spaces, and in particular the family of hyperplane Maslov--Arnold spaces. Using these spaces and the resulting indices, we prove generalized Morse index theorems for the non-selfadjoint operator $\cL$ in \eqref{Ldef}, which we use to give sufficient conditions for the instability of the steady state $\bar u$ of \eqref{RD}.

\section{Maslov--Arnold spaces}\label{sec:eLG}
The definition of a Maslov--Arnold space already appeared in the introduction; in this section we clarify some of their topological properties, in particular for the hyperplane MA spaces. In Section~\ref{sec:def} we precisely define the geometric intersection number of a ``sufficiently generic" loop, which appears in the definition of an MA space. In Section~\ref{sec:hyper} we construct the \emph{hyperplane Maslov--Arnold spaces}, which are open, dense subsets of the Grassmannian; these will be our main tool when we study reaction--diffusion systems in Sections \ref{sec:counting} and \ref{sec:app}. In Section \ref{sec:hypindex} we elaborate on the generalized Maslov index for a hyperplane MA space, which we call the \emph{hyperplane index}, and explain how to define it for paths with different endpoints. Finally, in Section \ref{sec:2D} we describe the two-dimensional case in detail.

While the hyperplane spaces suffice for the applications in this paper, we will revisit general Maslov--Arnold spaces in Section \ref{furex}, where we settle some natural theoretical questions by providing further examples (and non-examples) of MA spaces.

\subsection{The intersection number of a sufficiently generic loop}\label{sec:def}
The subset 
\[
    \cZ_{_P}^{1} := \big\{ W \in Gr_n(\bbR^{2n}) : \dim(W \cap P)=1 \big\} \subseteq  \cZ_{_P}
\]
is a smooth submanifold of $Gr_n(\bbR^{2n})$ with one-dimensional normal bundle $\nu$.  We say a map $\gamma \colon S^1 \rightarrow Gr_n(\bbR^{2n}) $ is \emph{sufficiently generic} if \begin{itemize}
    \item  it is smooth  (i.e. $C^{\infty}$), 
    \item all intersections between the image $\gamma(S^1)$ and the train $\cZ_{_P}$ are contained in  $\cZ_{_P}^{1}$, and
    \item all of these intersections are transverse, meaning that if $\gamma(t) \in \cZ_{_P}^{1} $ for some $t$, then the velocity vector $\dot{\gamma}(t) \in T_{\gamma(t)} Gr_n(\R^{2n})$ is not tangent to  $\cZ_{_P}^{1}$. \end{itemize}

Given a subset $\cM \subseteq Gr_n(\bbR^{2n})$ and an $n$-plane $P \in Gr_n(\bbR^{2n})$, we call $\cZ_{_P} \cap \cM$ the train of $P$ in $\cM$. A \emph{co-orientation}\footnote{The existence of a co-orientation is equivalent to the restricted line bundle being trivializable, meaning that there exists an isomorphism of topological line bundles $\nu|_{\cZ_{_P}^{1} \cap \cM} \cong \R \times \left(\cZ_{_P}^{1} \cap \cM\right)$.} of the train is an orientation of the restricted line bundle $\nu|_{\cZ_{_P}^{1} \cap \cM}$, where $\nu$ is the normal bundle of $\cZ_{_P}^{1}$.  Given a sufficiently generic curve $\gamma\colon S^1 \rightarrow \cM \subseteq Gr_n(\bbR^{2n})$ and a co-orientation of
$\cM \cap \cZ_{_P}^{1}$, the \emph{geometric intersection number} of $\gamma$ with the train $ \cZ_{_P} \cap \cM$ is defined to be the finite sum 
\begin{equation}\label{geomintnum}
    \sum_{\substack{ t \in S^1 \\ \gamma(t) \in \cZ_{_P}}} \sgn(t)
\end{equation}
where $\sgn(t) = 1$ (resp. $-1$) if the induced linear isomorphism $ T_t S^1 \rightarrow \nu_{\gamma(t)}$ is orientation preserving (resp. reversing).

\subsection{Hyperplane Maslov--Arnold spaces}
\label{sec:hyper}

Let $V \cong \R^{2n}$ and denote by $\bigwedge^n(V)$ the $n$th degree exterior product of $V$, which is a vector space of dimension $\binom{2n}{n}$. The projective space $P\left(\bigwedge^n(V)\right)$ is the set of the one-dimensional subspaces of $\bigwedge^n(V)$. Given a non-zero $n$-vector $ \xi \in \bigwedge^n(V)$, we denote by $[\xi] \in P(\bigwedge^n(V)) $ the span of
$\xi$. The Pl\"ucker embedding maps $Gr_n(V)$ into $P(\bigwedge^n(V))$, sending $\spn \{v_1, \ldots, v_n\}$ to $[v_1\wedge \cdots \wedge v_n]$. We will sometimes abuse notation and simply identify $Gr_n(V)$ with its image $G \subseteq P(\bigwedge^n(V))$. Observe that $G$ equals the subset of those $[\xi] \in P(\bigwedge^n(V))$ for which $\xi$ is decomposable as a product of vectors in
$V$.

Let $V^*:= \Hom(V, \R)$ denote the dual vector space of $V$. For $k \geq 1$, each $\omega \in \bigwedge^k(V^*)$ corresponds to a skew-symmetric multilinear map $\omega \colon V^k= V \times \cdots \times V \rightarrow \R$. There is a canonical isomorphism $\bigwedge^k(V^*) \cong \bigwedge^k(V)^*$, so elements $\omega \in \bigwedge^k(V^*)$ are in one-to-one correspondence with
linear maps
$  \omega \colon \bigwedge^k(V) \rightarrow \R.$  Both interpretations of $ \bigwedge^k(V^*)$ will be important in what follows.

Each non-zero $n$-form $\omega \in \bigwedge^n(V^*) $ determines a hyperplane 
\begin{equation}
\label{Homega}
	H_\omega := \left\{ [\xi] \in P\Big(\bigwedge\nolimits^{\! n}(V) \Big) : \omega(\xi) = 0 \right\}.
\end{equation}
Conversely, a hyperplane $H \subseteq  P \left(\bigwedge^{n}(V) \right)$ determines, up to multiplication by a non-zero scalar, an $n$-form $\omega$ such that $H = H_{\omega}$. If the hyperplane $H_\omega$ is intersected with $G$, we get 
\begin{equation}\label{GH1}
	G \cap H_\omega  =   \left\{ [v_1 \wedge \cdots  \wedge v_n ] \in P\Big(\bigwedge\nolimits^{\! n}(V) \Big) : \omega(v_1, \ldots ,v_n) = 0 \right\},
\end{equation}
and by the Pl\"ucker embedding this corresponds to
\begin{equation}
\label{GPlucker}
	G \cap H_\omega \cong \big\{ \spn \{v_1, \ldots ,v_n\} \in Gr_n(V) : \omega(v_1, \ldots, v_n) = 0 \big\}.
\end{equation}
For instance, if $n=2$ and $\omega$ is a non-degenerate two-form (i.e. a symplectic form), then $G \cap H_\omega$ is the Lagrangian Grassmannian $\Lambda_\omega$.

\begin{rem}
The preimage of $G \cap H_\omega$ under the Pl\"ucker embedding, which appears on the right-hand side of \eqref{GPlucker}, is precisely the set $\cH_\omega$ defined in \eqref{GPlucker1}. For the rest of the paper we will identify these spaces whenever convenient, and use the notation $G \setminus (H_1 \cap H_2)$ and 
$Gr_n(\bbR^{2n}) \setminus \big(\cH_{\omega_1} \cap \cH_{\omega_2}\big)$ interchangably.
\end{rem}

Another important type of hyperplane, particularly relevant to our theory of Maslov--Arnold spaces, is that corresponding to the train of a fixed subspace, as defined in \eqref{def:train1}. Given 
%an $k$-form $\omega \in \bigwedge^k(V^*)$ for $k \geq 1$, and
a vector $v \in V$, the \emph{contraction map} $\iota_v \colon \bigwedge^k(V^*) \rightarrow \bigwedge^{k-1}(V^*)$ is defined for each $k \geq 1$ by $(\iota_v \omega)(w_1, \ldots ,w_{k-1}) := \omega(v, w_1, \ldots ,w_{k-1})$. Define the kernel of $\omega$ by $\ker\omega := \{ v \in V : \iota_v \omega = 0\}$.

\begin{lemma}\label{lemtransverse}
Let $\omega \in \bigwedge^n(V^*)$. If $\ker\omega \subseteq V$ has dimension $n$, then $G \cap H_{\omega}$ is the train of the subspace $\ker\omega$, i.e.
\begin{equation}\label{hyperplanetrain}
	G \cap H_\omega \cong \cZ_{\ker\omega} = \big\{ W \in Gr_n(V) : W \cap \ker \omega \neq \{0\} \big\}.
\end{equation}

Moreover, letting $Z^1 \subseteq G$ denote the image of $\cZ_{\ker\omega}^1$ under the Pl\"ucker embedding, the normal bundle of $Z^1$ in $G$ is naturally isomorphic to the restriction to $Z^1$ of the normal bundle of $H_{\omega}$ in $P(\bigwedge^n(V))$.
\end{lemma}

\begin{proof}
Let $v_1, \ldots ,v_n$ be a basis of $\ker\omega$, and extend to a basis $v_1,\ldots , v_{2n} $ of $V$, with dual basis $v_1^*, \ldots , v_{2n} ^*\in V^*$. Expressing $\omega$ in terms of this dual basis, and imposing the condition that $\iota_{v_i} \omega = 0$ for $1 \leq i \leq n$, we deduce that 
\begin{equation}\label{hyperplanetrain2}
\omega = c v^*_{n+1} \wedge \cdots \wedge v^*_{2n}
\end{equation}
for some nonzero $c \in \R$. 
It follows from \eqref{GH1} that $[w_1 \wedge \cdots \wedge w_n] \in G \cap H_{\omega}$ if and only if $\spn\{ w_1, \ldots ,w_n\}$ intersects
$\ker\omega$ non-trivially, proving (\ref{hyperplanetrain}).

Since $Z^1$ is a smooth submanifold of codimension one in $G$ and $H_{\omega}$ is a smooth submanifold of codimension one in $P(\bigwedge^n(V))$, to prove the isomorphism of normal bundles, it suffices to show that $G$ and $H_\omega$ intersect transversely along $Z^1$ in $P(\bigwedge^n(V))$. 

A point in $Z^1$ represents a $n$-dimensional subspace $W$ for which $W \cap \ker\omega$ is one dimensional. We can choose a basis $\{v_1,...,v_{2n}\}$ for $V$ so that 

\begin{itemize}
    \item $\spn\{v_1,...,v_n\} =\ker\omega$,
    \item $\spn\{v_1, v_{n+2},...,v_{2n} \} = W$,
    \item $\spn\{v_1\} = W \cap \ker\omega $.
\end{itemize}

In terms of the dual basis, $\omega$ is of the form (\ref{hyperplanetrain2}). Consider the smooth path in $\bigwedge^n(V)$ defined by $\gamma(t) :=   (v_1 + t v_{n+1})\wedge v_{n+2} \wedge \cdots \wedge v_{2n}$. Notice that $[\gamma(t)]$ defines a path in $G$, that $\gamma(0) = W$, and that $\omega(\gamma(t)) = t$, so $d \omega(\gamma(t))/dt = 1 \neq 0$. It follows that the velocity vector of $[\gamma(t)]$ at $W$ is tangent to $G$, but not to $H_\omega$, so we conclude that the intersection between $G$ and $H_\omega$ is transverse at that point.
\end{proof}

In view of Lemma \ref{lemtransverse}, one can try to construct an MA space using the following strategy. Given an $n$-dimensional subspace $P \subseteq V$, choose an $n$-form $\omega$ for which $\ker \omega = P$, then look for an open subset $U \subseteq P(\bigwedge^n(V))$ such that the normal bundle of $U \cap H_{\omega}$ is orientable. This then determines a cohomology class in $H^1(U;\Z)$ dual to $U \cap H_{\omega}$ (by pulling back the Thom class of the normal bundle, as described in \cite[\S 6]{bott1982differential}). If the restriction of this class to $U \cap G$ has infinite order, then $U \cap G$ is an MA space. Our hyperplane construction can be described in these terms, but it can also be explained in elementary geometric terms that avoids the machinery of Thom classes and allows us to interpret the Maslov index of a loop as the winding number of that loop around a circle.

Given an integer $m \geq 1$, consider the projection map $\pi\colon \R^{m+1} \rightarrow \R^2$ defined by $\pi(x_1,\ldots,x_{m+1}) = (x_m,x_{m+1})$. This is surjective, with kernel naturally isomorphic to $\R^{m-1}$. More generally, for any point $p \in \R^2$, the preimage $\pi^{-1}(p)$ is an affine space modelled on $\R^{m-1}$. 

Consider the open subset $$U:= \big\{ [x_1:\cdots:x_{m+1}] \in \R P^m : (x_m, x_{m+1}) \neq (0,0) \big\}, $$ which can be identified with the complement of a copy of $\R P^{m-2}$ in $\R P^m$.  Then $\pi$ descends to a $C^{\infty}$ map $[\pi]\colon U \rightarrow \R P^1$ defined by the rule 
\begin{equation}\label{pimap}
[\pi]\big([x_1:\dots:x_{m+1}]\big) := [\pi(x_1,\dots,x_{m+1})].
\end{equation}
Notice that $U$ is precisely the subset of $\R P^m$ where (\ref{pimap}) is well-defined.

\begin{lemma}\label{mobstrip}
The map $[\pi] \colon U \to \bbR P^1$ is a smooth fibre bundle, with fibres diffeomorphic to $\R^{m-1}$. In particular, $[\pi]$ is a homotopy equivalence between $U$ and $\R P^1$, and the preimage $[\pi]^{-1}([0:1])$ is an embedded submanifold of codimension one in $U$ with a trivializable normal bundle.
\end{lemma}
\begin{proof}
That $[\pi]$ is a submersion follows immediately from the fact that $\pi$ is a submersion. If $x \in \R^2 \setminus \{(0,0)\}$ then it is easy to see that the natural forgetful map from $\pi^{-1}(x)$ to $[\pi]^{-1}([x])$ is a diffeomorphism. By the implicit function theorem, the fibres $[\pi]$ are embedded submanifolds with trivializable normal bundles.
\end{proof}

To understand Lemma \ref{mobstrip} it may be helpful to consider the example when $m=2$, so $\R P^m = \R P^2$ is a real projective plane and $\R P^{m-2} = \R P^0$ is a single point. Then $U$ is $\R P^2$ with a single point removed, which is diffeomorphic to the M\"obius strip, and the $\R^1$-fibre bundle $[\pi]\colon U \rightarrow \R P^1$ is simply the projection of the M\"obius strip onto the base circle. The general case is much like this, but with fibre $\R^{m-1}$.

We now have the ingredients necessary to construct an MA--space. Let $V \cong \R^{2n}$ and let $P \subseteq V $ be an $n$-dimensional subspace. Suppose that $\omega_1, \omega_2 \in \Lambda^n(V^*)$ are linearly independent $n$-forms and that $P = \ker \omega_1$.
Define $\cM \subseteq G$ by
\begin{eqnarray}
    \cM &:=&  G \setminus ( H_{\omega_1} \cap H_{\omega_2})\\
    &= & \big\{ [u_1 \wedge \cdots \wedge u_n] \in G ~:~ \omega_1(u_1,\ldots,u_n) \neq 0 ~\mathrm{ or }~ \omega_2(u_1,\ldots,u_n) \neq 0 \big\}.
 \end{eqnarray}

We therefore have a well-defined continuous map $\phi \colon \cM \rightarrow \R P^1$ given by
\begin{equation}
\label{phidef}
    %\phi \colon \cM \rightarrow \R P^1,~~ 
    \phi\big([u_1 \wedge \cdots \wedge u_n]\big) = [ \omega_1(u_1,\ldots,u_n) : \omega_2(u_1,\ldots,u_n) ].
\end{equation}
A choice of orientation on $\R P^1$ determines a generator $\mu \in H^1(\R P^1;\Z) \cong \Z$ and we set 
\begin{equation}\label{maslovclass}
\alpha = \phi^*(\mu) \in H^1(\cM;\Z).
\end{equation}

\begin{theorem}\label{cor:int}\label{MAspacethm}
The triple $(\cM, P, \alpha)$ defined above is an MA space. Moreover, the index of a continuous loop $\gamma\colon S^1 \rightarrow \cM$ is simply the winding number of the composite map $\phi \circ \gamma \colon S^1 \rightarrow \R P^1$. \end{theorem}

\begin{proof}
Combining Lemmas \ref{lemtransverse} and \ref{mobstrip}, we see that $\phi^{-1}([0:1]) = \cM \cap H_{\omega_1} = \cM \cap \cZ_{P}$ and all points lying in $\cM \cap \cZ_P^1$ are regular values for $\phi$. Therefore the normal bundle of $\cM \cap \cZ_P^1$ in $\cM$ is simply the pull back of the tangent space $T_{[0:1]} \R P^1$. The geometric intersection number of a sufficiently generic loop $\gamma \colon S^1 \rightarrow \cM$ will thus agree with the usual intersection number of the composite $\phi \circ \gamma\colon S^1 \rightarrow \R P^1$ with the point $[0:1]$, which in turn agrees with the winding number of $\phi \circ \gamma$.

It only remains to prove that $\alpha$ has infinite order.  It is enough to produce a loop $\gamma\colon [0,1] \rightarrow \cM$ with non-zero index. 

By iterative application of Lemma \ref{linalg} we can find vectors $v_1, \ldots ,v_{n-1}$ so that the contractions $\iota_{v_1\wedge \cdots \wedge v_{n-1}} \omega_1$ and $\iota_{v_1\wedge \cdots \wedge v_{n-1}} \omega_2$ are linearly independent. Therefore, there exist $u_1,u_2 \in V$ such that $\omega_i( u_j, v_1, \ldots ,v_{n-1}) = \delta_{ij}$. Consequently, the loop
\[
	\gamma(t)  = \spn\{  \cos(\pi t) u_1 + \sin(\pi t) u_2, v_1, \ldots ,v_{n-1} \}
\]
has index one.  
\end{proof}

\begin{lemma}\label{linalg}
Let $V$ be a vector space and $ k\geq 2$.  If $\omega_1, \omega_2 \in \bigwedge^k(V^*)$ are linearly independent, then there exists a vector $v \in V$ such that the contractions $ \iota_v \omega_1, \iota_v \omega_2 \in \bigwedge^{k-1}(V^*)$ are linearly independent. 
\end{lemma}

\begin{proof}
Choose a basis $e_1, \ldots , e_n \in V$, with dual basis $e^*_1, \ldots, e^*_n \in V^*$, and expand $ \omega_1 = \sum_I a_I e^*_I$ and $\omega_2 = \sum_I b_I e^*_I$, where $I = \{ i_1 < \cdots < i_k\}$ are multi-indices and $e^*_I := e_{i_1}^* \wedge ... \wedge e_{i_k}^* $. Since $\omega_1$ and $\omega_2$ are linearly independent, there is a pair of multi-indices $I,J$ such that
the minor
\begin{equation}\label{minor2}
\det \begin{pmatrix}
a_I & b_I \\
a_J & b_J
\end{pmatrix} \neq 0.
\end{equation}
If there exists a pair of multi-indices $I,J$ satisfying (\ref{minor2}) and an index $i \in I \cap J$, then $ \iota_{e_i} \omega_1$ and $\iota_{e_i} \omega_2$ are linearly independent and we are done. 

Suppose instead that every pair of multi-indices satisfying (\ref{minor2}) has $I \cap J = \varnothing$. For a particular such pair, select $i \in I$ and $j \in J$ and define $I' := I \cup \{j\} \setminus \{i\}$ and $J' := J \cup \{i\} \setminus \{j\}$. Since $k \geq 2$ it follows that each of $I \cap I'$, $I' \cap J$, $I\cap J'$, and $J \cap J'$ is non-empty. Considering the
corresponding minors, we deduce that $a_{I'} = a_{J'} = b_{I'} = b_{J'} =0$ and consequently that $ \iota_{e_i + e_j} \omega_1$ and $\iota_{e_i+ e_j} \omega_2$ are linearly independent.
\end{proof}

\subsection{The hyperplane index}
\label{sec:hypindex}
Since the generalized Maslov index for loops in a hyperplane space $\cM$ can be interpreted as a winding number, we can easily extend its definition to non-closed paths. This amounts to choosing a convention for the endpoints.

We start by defining the winding number through $[0:1]$ for a continuous path in $\R P^1$. We do this by first mapping to $S^1$ and then looking at the winding through the point $e^{i\pi}$. Viewing $S^1$ as a subset of the complex plane, we define a map $\tau \colon \bbR P^1 \to S^1$ by
\begin{equation}
\label{taudef}
    \tau\big([x:y]\big) = 
    \left( \frac{x-iy}{|x-iy|}\right)^2.% = \frac{x^2 - y^2 - 2ixy}{x^2 + y^2}.
\end{equation}

\begin{define}\label{def:winding}
Let $\eta \colon [a,b] \to \R P^1$ be a continuous path. If $\eta(t_0) = [0:1]$ for some $t_0 \in [a,b]$, then there is a unique lift $\theta \colon [a,b] \to \bbR$ such that $\theta(t_0) = \pi$ and $e^{i \theta(t)} = (\tau \circ \eta)(t)$ for $t \in [a,b]$, and we define
\begin{equation}\label{Wdef}
	\Wind(\eta) = \left\lfloor \frac{\theta(b) - \pi}{2\pi} \right\rfloor - \left\lfloor \frac{\theta(a) - \pi}{2\pi} \right\rfloor.
\end{equation}
If no such $t_0$ exists we set $\Wind(\eta) = 0$.
\end{define}

It is not hard to see that this is well defined (independent of the choice of $t_0$). It is clearly additive under concatenation of paths, and if $\eta(a) = \eta(b)$ it reduces to the usual winding number of a loop, $(\theta(b) - \theta(a))/2\pi$. Some consequences of this definition can be seen in Figure \ref{fig:winding}, where we show the composite path $\tau \circ \eta$ in $S^1$. These illustrate four possible cases: a positive or negative curve, passing through $e^{i\pi}$ at either $t_0=a$ or $t_0=b$. All four examples are parameterized so that $t_0=\pi$.
\begin{itemize}
    \item $e^{i t}$, $\pi/2 \leq t \leq \pi$ has $\theta(t) = t$, so the winding number is $\lfloor 0 \rfloor - \lfloor -1/4 \rfloor = 1$
    \item $e^{i t}$, $\pi \leq t \leq 3\pi/2$ has $\theta(t) = t$, so the winding number is $\lfloor 1/4 \rfloor - \lfloor 0 \rfloor = 0$
    \item $e^{-it}$, $\pi/2 \leq t \leq \pi$ has $\theta(t) = 2\pi - t$, so the winding number is $\lfloor 0 \rfloor - \lfloor 1/4 \rfloor = 0$
    \item $e^{-it}$, $\pi \leq t \leq 3\pi/2$ has $\theta(t) = 2\pi - t$, so the winding number is $\lfloor -1/4 \rfloor - \lfloor 0 \rfloor = -1$
\end{itemize}
These four cases are shown from left to right. Therefore, Definition \ref{def:winding} says that for a positively oriented curve we count crossings at $t_0=b$ but not $t_0=a$, and vice versa for a negatively oriented curve.

%version3, hopefully correct this time...
\begin{figure}
\begin{tikzpicture}
	\draw[->,>=stealth',thick] (0,1.5) arc[radius=1.5, start angle=90, end angle=180];
	\draw[->,thick] (0,-0.5) to (0,2);
	\draw[thick,dashed] (-2,0) to (0,0);
	\draw[->,thick] (0,0) to (0.5,0);
	\node at (-0.5, -1) {$\Wind=1$};
\end{tikzpicture}
\hfill
\begin{tikzpicture}
	\draw[<-,>=stealth',thick] (0,-1.5) arc[radius=1.5, start angle=-90, end angle=-180];
	\draw[->,thick] (0,-2) to (0,0.5);
	\draw[thick,dashed] (-2,0) to (0,0);
	\draw[->,thick] (0,0) to (0.5,0);
	\node at (-0.5, -2.5) {$\Wind=0$};
\end{tikzpicture}
\hfill
\begin{tikzpicture}
	\draw[->,>=stealth',thick] (0,-1.5) arc[radius=1.5, start angle=-90, end angle=-180];
	\draw[->,thick] (0,-2) to (0,0.5);
	\draw[thick,dashed] (-2,0) to (0,0);
	\draw[->,thick] (0,0) to (0.5,0);
	\node at (-0.5, -2.5) {$\Wind=0$};
\end{tikzpicture}
\hfill
\begin{tikzpicture}
	\draw[<-,>=stealth',thick] (0,1.5) arc[radius=1.5, start angle=90, end angle=180];
	\draw[->,thick] (0,-0.5) to (0,2);
	\draw[thick,dashed] (-2,0) to (0,0);
	\draw[->,thick] (0,0) to (0.5,0);
	\node at (-0.5, -1) {$\Wind=-1$};
\end{tikzpicture}
\caption{Illustrating the winding number, with respect to the point $e^{i\pi}$, or equivalently the negative real axis, for non-closed curves in $S^1$ with crossings at their endpoints. Our convention is to count negative crossings at the beginning of a curve and positive crossings at the end.}
\label{fig:winding}
\end{figure}
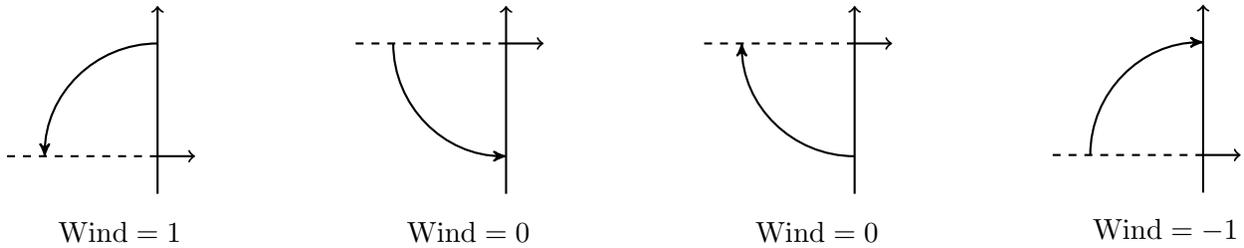

This leads to the following results for monotone paths. If $\eta \colon [a,b] \to \bbR P^1$ is continuously differentiable and has the property that $\theta'(t_*) > 0$ whenever $\eta(t_*) =[0:1]$, then the set $\{t_* \in [a,b] : \eta(t_*) = [0:1] \}$ is finite, and
\begin{equation}\label{eq:pos}
	\Wind(\eta) = \# \big\{t_* \in (a,b] : \eta(t_*) = [0:1] \big\}.
\end{equation}
Similarly, if $\theta'(t_*) < 0$ whenever $\eta(t_*) = [0:1]$, then
\begin{equation}\label{eq:neg}
	\Wind(\eta) = - \# \big\{t_* \in [a,b) : \eta(t_*) = [0:1] \big\}.
\end{equation}
In general, for any continuous path $\eta$ we have
\begin{equation}\label{eq:gen}
	\big|\Wind(\eta)\big| \leq \# \big\{t_* \in [a,b] : \eta(t_*) = [0:1] \big\}.
\end{equation}
To make use of these formulas, we need to find the lift $\theta(t)$ when the path $\eta(t)$ is given in terms of homogeneous coordinates. Suppose $\eta(t) = [x(t) : y(t)]$ has $x(t_0) = 0$, and hence $y(t_0) \neq 0$. 
It follows from \eqref{taudef} that
\[
    \tau \circ \eta = \frac{(x^2 - y^2) - 2ixy}{x^2+y^2}.
\]
In the notation of Definition \ref{def:winding}, this means the lift of $\tau \circ\eta$ is 
$\theta(t) = \pi - \arctan \frac{2x(t)y(t)}{x(t)^2 - y(t)^2}$, and hence
\begin{equation}
\label{thetap}
    \theta'(t_0) = 2 \frac{x'(t_0)}{y(t_0)}.
\end{equation}

That is, the monotonicity of the path is determined by the sign of the ratio $x'(t_0)/y(t_0)$. This simple observation will be used repeatedly in Section \ref{sec:counting}.

Having described the winding number for a path in $\R P^1$ with different endpoints, we are finally ready to extend the definition of the generalized Maslov index from loops to arbitrary paths in a hyperplane MA space.

\begin{define}
\label{def:Hindex}
Let $\gamma \colon [a,b] \to \cM$ be a continuous path in the hyperplane MA space $\cM = G \setminus (H_1 \cap H_2)$. We define its \emph{hyperplane index} to be
\begin{equation}
    \Ind(\gamma;P) := \Wind(\phi \circ \gamma),
\end{equation}
where $\phi \colon \cM \to \R P^1$ is defined in \eqref{phidef}.
\end{define}

For future reference we summarize some important properties of this index, which follow easily from the definition.

\begin{prop}
\label{ind:properties}
The hyperplane index has the following properties:
\begin{enumerate}
    \item (extension) If $\gamma \colon S^1 \to \cM$ is a loop, then $\Ind(\gamma;P)$ is equal to the generalized Maslov index $\Mas(\gamma;P)$ from Definition \ref{def:genMas}.
    \item (nullity) If $\gamma \colon[a,b] \to \cM$ is a path with $\gamma(t) \notin H_1$ for all $t \in [a,b]$, then $\Ind(\gamma;P) = 0$.
    \item (additivity) If $\gamma_1 \colon [a,b] \to \cM$ and $\gamma_2 \colon [b,c] \to \cM$ are paths with $\gamma_1(b) = \gamma_2(b)$, and $\gamma_2 \ast \gamma_1$ denotes their concatenation, then
    \[
        \Ind(\gamma_2 \ast \gamma_1;P) = 
        \Ind(\gamma_1;P) + Ind(\gamma_2;P).
    \]
    \item (homotopy invariance) If $\gamma_1, \gamma_2 \colon [a,b] \to \cM$ are homotopic in $\cM$, with fixed endpoints, then
    \[
        \Ind(\gamma_1;P) = \Ind(\gamma_2;P).
    \]
\end{enumerate}
\end{prop}

\subsection{The two-dimensional case}
\label{sec:2D}

We consider in detail the $n=2$ case, where $\cM$ can be described explicitly.  The hyperplanes now come in two types. If $\omega$ is a non-degenerate 2-form, i.e. a sympectic form, then $H_{\omega} \cap G$ is the corresponding Lagrangian Grassmanian. If $\omega$ is degenerate, then $\ker\omega \subseteq V$ is two-dimensional, and $ H_{\omega} \cap G$ is the train of $\ker\omega$.

Given linearly independent forms $\omega_1, \omega_2 \in \bigwedge^2(V)$, they span a pencil of bilinear forms $x \omega_1 + y \omega_2$, $(x,y) \neq (0,0)$. Choose a basis for $V$, so that the $\omega_i$ are represented by skew-symmetric $4 \times 4$-matrices. Consider the homogeneous quadratic polynomial $q(x,y):= \Pf( x \omega_1 + y \omega_2)$, where $\Pf$ denotes the Pfaffian\footnote{Recall that for a skew symmetric $4 \times 4$ matrix $A$, $\Pf(A) := A_{1,2}A_{3,4}-A_{1,3}A_{2,4} +A_{1,4}A_{2,3}$ satisfies $\Pf(A)^2 = \det(A)$.}. The roots of $q$ correspond to the degenerate two-forms in the pencil. There can be zero, one, two, or
infinitely many roots.

\begin{prop}
Up to a change of basis transformation of $V$, there are four possible isomorphism types for $\cM$. They are classified by the number of real roots of $q(x,y):= \Pf(x\omega_1 +y\omega_2)$.
\end{prop}

\begin{proof}
The Pl\"ucker embedding identifies $G \subseteq P(\bigwedge^2(V)) \cong \R P^5$ as a quadric, the so-called Klein quadric, defined by the non-degenerate, split signature symmetric bilinear form 
\begin{align*}
	B \colon  \bigwedge\nolimits^{\!2}(V) \otimes \bigwedge\nolimits^{\!2}(V)  \rightarrow \bigwedge\nolimits^{\!4}(V) \cong \R, && B(\eta, \xi) = \eta \wedge \xi.
\end{align*}
We call a linear transformation $ A \in GL( \bigwedge^2(V))$ \emph{orthogonal} if it leaves $B$ invariant and \emph{anti-orthogonal} if it sends $B$ to $-B$. Observe that both orthogonal and anti-orthogonal transformations preserve $G$.

Let $W \subseteq \bigwedge^2(V)$ be the four-dimensional subspace for which $P(W) = H_{\omega_1} \cap H_{\omega_2}$. Since $B$ is non-degenerate, the $B$-complement of $W$, $W^\perp: = \{ u \in \bigwedge^2(V) : B(u,w)=0, \text{ for all } w \in W\}$, is two dimensional. Consider the restricted bilinear form $B':= B|_{W^{\perp}}$. The associated quadratic form $q'(v) := B(v,v)$ on
$W^{\perp}$ can be identified via duality with $q(x,y)$. By Sylvester's law of inertia, there are six possible isomorphism classes for $B'$ modulo change of basis, and four isomorphism classes modulo multiplication by $\pm 1$. These are classified by the number of roots of $q(x,y)$.

If $W_1, W_2 \subseteq \bigwedge^2(V)$ are four-dimensional subspaces such that $B|_{W^{\perp}_1}$ is isomorphic to $B|_{W_2^{\perp}}$, then by Witt's Theorem (see \cite[Thm 1.2]{M}) there exists an orthogonal transformation of $\bigwedge^2(V)$ sending $W_1$ to $W_2$. Similarly, if $B|_{W^{\perp}_1}$ is isomorphic to $-B|_{W_2^{\perp}}$ then there exists an anti-orthogonal
transformation sending $W_1$ to $W_2$. It follows in either case that $ G \setminus P(W_1)$ is isomorphic to $G \setminus P(W_2)$.

Finally we must show that the orthogonal transformation of $\bigwedge^2(V)$ used above can be induced by a linear transformation of $V$ (the anti-orthogonal case is an easy consequence). Denote by $O(B)$ the group of orthogonal transformations of $(\bigwedge^2(V),B)$. The natural homomorphism $SL(V) \rightarrow O(B)$ has kernel $ \pm I_4$, so since both groups are 15
dimensional, it is a surjection onto the identity component of $O(B)$. It remains to show that for each two-dimensional $ U \subseteq \bigwedge^2(V)$, there exists $A \in O(B)$ in each path component of $O(B)$ such that $A(U) = U$.

Choose a basis $e_1, \ldots ,e_6 \in \bigwedge^2(V)$ so that $ B(e_i, e_j) = (-1)^i\delta_{ij}$, where $\delta_{ij}$ is the Kronecker delta. According to \cite[Cor 1.1]{M}, representatives for the four path components of $O(B)$ are given by the transformations that fix $e_1,e_2,e_3,e_4$ and send $e_5 \mapsto \pm e_5$ and $e_6 \mapsto \pm e_6$. Since every different isomorphism
class of $B|_U$ can be realized by a two-dimensional $ U \subseteq \spn \{e_1, \ldots ,e_4\}$, this completes the proof.
\end{proof}

Up to a change of basis for $V$, the pencil of bilinear forms above is isomorphic to one of four possibilities 
$$
  \begin{pmatrix}
0 & x & y & 0\\
-x & 0 &  0& 0\\
-y & 0&0 & 0 \\
0 &  0 & 0 & 0\\
\end{pmatrix}, \ 
\begin{pmatrix}
0 & x & 0 & 0\\
-x & 0 &  0& 0 \\
 0 & 0 &0 & y\\
0 &  0 & -y & 0\\
\end{pmatrix}, \ 
 \begin{pmatrix}
0 & 0 & x & y\\
0 & 0 &  -y& x \\
-x & y &0 & 0\\
-y &  -x & 0 & 0\\
\end{pmatrix}, \
 \begin{pmatrix}
0 & x & y & 0\\
-x & 0 &  0 & y \\
-y & 0 &  0& 0\\
0 & -y & 0 & 0\\
\end{pmatrix}, $$
which have respective Pfaffians (up to sign)  $q(x,y)=0$, $ xy$, $ x^2+y^2$, and $y^2$. 

\begin{rem}\label{fourcases}

If $n=2$ then $X := G \cap H_1 \cap H_2$ is homeomorphic to one of the following four respective types.
\begin{itemize}
\item[(i)] If $q =0$, then every linear combination  $x \omega_1 + y \omega_2$ is degenerate. In this case $X$ is the intersection of trains for  $\ker\omega_1$ and $\ker\omega_2$, which intersect non-trivially. It follows that $X$ is a union of  $\R P^1 \times \R P^1$ with $\R P^2$ along a wedge sum $\R P^1 \vee \R P^1$.

\item[(ii)] If $q$ has two distinct real roots, then $X$ is the intersection of trains for a pair of two-dimensional subspaces $P_1, P_2 \subseteq V$ which intersect trivially. In this case $X = P(P_1) \times P(P_2) \cong \R P^1 \times \R P^1$ is a torus.

\item[(iii)] If $q$ has one root with multiplicity two, then $X$ can be identified with the intersection of the Lagrangian Grassmannian and the train of a Lagrangian subspace, for some symplectic form $\omega$. Therefore, $X$ is isomorphic to the Maslov cycle described by Arnol'd \cite[\S 3]{A85}; it is homeomorphic to the one point compactification of $S^1 \times \R$. 

\item[(iv)] If $q$ has no real roots, then there exists a quaternionic structure $I,J,K$ on $V$ in which the pencil is spanned by symplectic forms $\omega_I$ and $\omega_J$, and $X$ can be identified with the intersection of their respective Lagrangian Grassmanians, $\Lambda_I \cap \Lambda_J$. Equivalently, $X \cong S^2$ is identified with the complex projective line $P( \C^2_K)$
with respect to the third complex structure $K$. In this case $G \setminus X$ is \emph{not} an MA space, because $H_1$ is not a train.

We note that $G \setminus X$ can also be identified with the homogeneous space $GL_2(\C)/GL_2(\R)$. To see this, consider the action of $GL_2(\C)$ on $G$ determined by a choice of complex basis $v_1,v_2 \in (V,K) \cong \C^2$. This action has two orbits: the orbit $X$ consisting of complex one-dimensional subspaces of $(V,K)$, and its complement $G \setminus X$ consisting of real two-dimensional subspaces that are not invariant under $K$. The stabilizer of the real span of $\{v_1,v_2\}$ is identified with $GL_2(\R)$, whence $G \setminus X \cong GL_2(\C)/GL_2(\R)$ by the orbit-stablizer theorem. This is analogous to the homogeneous space construction of the classical Lagrangian Grassmannian as $U(n)/O(n)$.
\end{itemize} 

\end{rem}

In Section \ref{sec:counting} we construct a hyperplane Maslov--Arnold space for the study of $n\times n$ systems of reaction--diffusion equations. When $n=2$ it is of the type (iii) described above.

\begin{prop}
If $\cM$ is one of the four cases above, then $H^1(\cM;\Z) \cong \Z$ and is generated by the class $\alpha$ defined in (\ref{maslovclass}). 
\end{prop}

\begin{proof}
By Poincar\'e duality $H^1(\cM;\Z) \cong H_3(G, X;\Z)$. Consider the long exact sequence of the pair
\[ \cdots \rightarrow H_3(G;\bbZ) \rightarrow H_3(G,X;\bbZ) \rightarrow H_2(X;\bbZ) \rightarrow H_2(G;\bbZ) \rightarrow \cdots \]
The homology groups of real Grassmannian have been calculated in \cite[Table IV]{J79}, giving  $H_3(G;\bbZ) = 0$ and $H_2(G;\bbZ) \cong \Z/2$. Since $X$ is isomorphic to a two-dimensional cell complex, $H_2(X;\Z)$ is torsion free. Exactness therefore implies that $H_3(G,X;\Z)$ is isomorphic to $H_2(X;\Z)$. In all four cases above it is straightforward to check $H_2(X;\Z) \cong \Z$, so it follows that $H^1(\cM;\Z) \cong \Z$. In Theorem~\ref{MAspacethm} we constructed a
loop in $\cM$ whose geometric intersection number with $H_1 \cap \cM$ is one, so it must generate $H_3(G,X;\Z) \cong H^1(\cM;\Z)$.
\end{proof}

\section{Counting unstable eigenvalues with the hyperplane index}\label{sec:counting}
We now explain how our theory of Maslov--Arnold spaces applies to the eigenvalue problem $\cL v = \lambda v$ for the operator $\cL$ defined in \eqref{Ldef}, with suitable boundary conditions. In this section we construct a hyperplane MA space that has desirable monotonicity properties for reaction--diffusion systems and hence allows us to relate real unstable eigenvalues to conjugate points, leading to the general result in Theorem \ref{thm:cpMorse}. Specific applications of this theorem will be explored in Section \ref{sec:app}.

We consider a coupled system of eigenvalue equations on a bounded interval $(0,L)$, with separated boundary conditions given by subspaces $P_0, P_1 \in Gr_n(\bbR^{2n})$. That is, we seek solutions to the first-order system 
\begin{equation}\label{Lsystem2}
	\frac{d}{dx} \begin{pmatrix} v \\ w \end{pmatrix} = \begin{pmatrix} 0 & D^{-1} \\ \lambda I - \nabla F(\bar u) & 0 \end{pmatrix} \begin{pmatrix} v \\ w \end{pmatrix}
\end{equation}
satisfying the boundary conditions
\begin{equation}
	 \begin{pmatrix} v(0) \\ w(0) \end{pmatrix} \in P_0, \qquad 
	  \begin{pmatrix} v(L) \\ w(L) \end{pmatrix} \in P_1.
\end{equation}
For instance, Dirichlet and Neumann boundary conditions correspond to the subspaces $P_D = \{(0,p) : p \in \bbR^n\}$ and $P_N = \{(q,0) : q \in \bbR^n\}$, respectively. The Robin boundary condition $D u_x = \Theta u$, where $\Theta$ is a real $n \times n$ matrix, corresponds to $P_R = \{ (q, \Theta q) : q \in \bbR^n\}$. Note that $P_R$ is Lagrangian if and only if $\Theta$ is
symmetric, and the special case $\Theta=0$ yields Neumann boundary conditions.

For each $x \in [0,L]$ and $\lambda \geq 0$ we define the subspace
\begin{equation}\label{mudef}
	W(x,\lambda) = \left\{ \begin{pmatrix} v(x) \\ w(x) \end{pmatrix} : \begin{pmatrix} v \\ w \end{pmatrix} 
	\text{ satisfies \eqref{Lsystem2} and } \begin{pmatrix} v(0) \\ w(0) \end{pmatrix} \in P_0 \right\},
\end{equation}
so that $\lambda$ is an eigenvalue of $\cL$ if and only $W(L,\lambda) \cap P_1 \neq \{0\}$.

Using our theory of hyperplane Maslov--Arnold spaces, we obtain a generalized Morse index theorem that relates unstable eigenvalues of $\cL$ to conjugate points, where $x_*$ is said to be a conjugate point if $W(x_*,0) \cap P_1 \neq \{0\}$.

\begin{theorem}\label{thm:cpMorse}
Assume that $P_1 = P_D$ and $P_0$ is either $P_D$ or $P_R$ for some $\Theta \in M_n(\bbR)$.
\begin{enumerate}
    \item For the path $W(x,\lambda)$ defined by \eqref{mudef}, there exists $\lambda_\infty > 0$ such that $W(x,\lambda) \cap P_1 = \{0\}$ for all $0 < x \leq L$ and $\lambda \geq \lambda_\infty$.
    
    \item If $H_1$ is the hyperplane corresponding to $P_1$, and $H_2 \neq H_1$ is a hyperplane such that
    \begin{equation}
    \label{xlinv}
    W(x,0) \in G \setminus (H_1 \cap H_2) \quad \text{ and } \quad W(L,\lambda) \in G \setminus (H_1 \cap H_2)
\end{equation}
for all $x \in (0,L]$ and all $\lambda \in [0,\lambda_\infty]$, then
\begin{align}
\begin{split}\label{eq:MorseMaslov1}
	\# \big\{\text{nonnegative eigenvalues of } \cL \big\} 
	&\geq 
	%- \Mas\left(W(L,\lambda)\big|_{\lambda \in [0,\lambda_\infty]} ; P_1 \right)  \\
	%& = 
	\Ind\left(W(x,0)\big|_{x \in [\delta,L]} ; P_1 \right) 
	- \mathfrak{m}
\end{split}
\end{align}
for $0 < \delta \ll 1$, where $\mathfrak{m}$ denotes the generalized Maslov index of the image (under $W$) of the boundary of $[\delta,L] \times [0,\lambda_\infty]$, oriented counterclockwise.

\item There is a hyperplane $H_2$ with the property that if \eqref{xlinv} holds, then
\begin{equation}
	\Ind\left(W(x,0)\big|_{x \in [\delta,L]} ; P_1 \right)  \\
	= \# \big\{\text{conjugate points in $(0,L]$}\big\}
\end{equation}
for $0 < \delta \ll 1$,
hence
\begin{equation}\label{eq:MorseMaslov2}
    \# \big\{\text{nonnegative eigenvalues of } \cL \big\} 
    \geq \# \big\{\text{conjugate points in $(0,L]$}\big\}
    - \mathfrak{m}
\end{equation}
and
\begin{equation}
\label{cpbound}
	\# \big\{\text{positive eigenvalues of } \cL \big\} 
	\geq \# \big\{\text{conjugate points in $(0,L)$} \big\}
	- \mathfrak{m}.
\end{equation}
\end{enumerate}
\end{theorem}

We emphasize, as in Remark \ref{rem:card}, that the hyperplane index counts eigenvalues and conjugate points \emph{without multiplicity}, unlike the classical Maslov index.

The choice of $\lambda_\infty$ and the condition \eqref{xlinv} guarantee that the image of the boundary of $[\delta,L] \times [0,\lambda_\infty]$ remains in the MA space, provided $\delta>0$ is sufficiently small, and so the index $\mathfrak{m}$ is defined. The hypothesis \eqref{xlinv} is significantly weaker than the assumption that $W$ maps the entire rectangle $[\delta,L] \times [0,\lambda_\infty]$ into the MA space. However, if this stronger invariance property holds, then the boundary of the rectangle is null homotopic and hence has zero index.

\begin{cor}
\label{cor:m0}
If, in addition to the hypotheses of Theorem \ref{thm:cpMorse}, $W(x,\lambda) \in G \setminus (H_1 \cap H_2)$ for all $(x,\lambda) \in (0,L] \times [0,\lambda_\infty]$, then $\mathfrak m = 0$, and so
\begin{equation}
	\# \big\{\text{positive eigenvalues of } \cL \big\} 
	\geq \# \big\{\text{conjugate points in $(0,L)$} \big\}.
\end{equation}
\end{cor}

The hyperplane Maslov index only detects real eigenvalues, whereas $\cL$ can have complex eigenvalues, since it is not assumed to be selfadjoint. However, since the number of unstable eigenvalues (i.e. those with positive real part) is bounded below by the number of positive eigenvalues, the existence of an interior conjugate point is a sufficient condition for instability, as long as $\mathfrak m \leq 0$. (In Section \ref{sec:Turing} we will see an example with $\mathfrak m > 0$, where there are interior conjugate points but no unstable eigenvalues.)

\begin{figure}
\begin{tikzpicture}
%	\draw[- >,>=stealth',thick] (0,-1.5) arc[radius=1.5, start angle=-90, end angle=0];
	\draw[->,thick] (0,-0.5) to (0,4);
	\draw[->,thick] (-0.5,0) to (5,0);
	\node at (5.5,0 ) {$x$};
	\node at (0,4.5 ) {$\lambda$};
	\draw[very thick] (0.25,0) -- (0.25,3) -- (4.5,3) -- (4.5,0) -- (0.25,0);
	\node at (-0.8,1.6) {$W(\delta,\lambda)$};
	\node at (5.5,1.6) {$W(L,\lambda)$};
	\node at (2.3,3.4) {$W(x,\lambda_\infty)$};
	\node at (2.3,0.4) {$W(x,0)$};
	\node at (2,2) {$\times$};
	%\draw[dashed,->] (2,2) circle (0.5);
	\draw[thick] (2.4,2) arc (0:180:0.4);
	\draw[->,thick] (1.6,2) arc (180:360:0.4);
	\node at (2.7,1.9) {$\gamma$};
\end{tikzpicture}
\caption{The homotopy argument in Theorem \ref{thm:cpMorse}. If the MA space contains the image of $[\delta,L] \times [0,\lambda_\infty]$ minus a single interior point, then $\mathfrak m$ equals the index of the loop $\gamma$, which can be shrunk to an arbitrarily small neighbourhood of $\times$.}
\label{fig:homotopy2flip}
\end{figure}
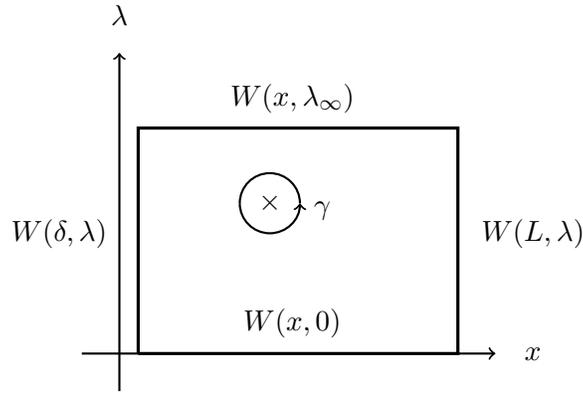

The main restrictions in Theorem \ref{thm:cpMorse} and Corollary \ref{cor:m0} are the invariance conditions on $W(x,\lambda)$. In the Hamiltonian case, this is guaranteed by the invariance of the Lagrangian Grassmannian under the associated flow. For the hyperplane MA spaces we do not know a corresponding family of dynamical systems for which such an invariance result necessarily holds, and in general these seem difficult to characterize. However, this can be checked on a case-by-case basis, as we demonstrate for several classes of examples in Section \ref{sec:app}.

Of particular interest are cases when the hypotheses of Theorem \ref{thm:cpMorse} are satisfied but those of Corollary \ref{cor:m0} are not, meaning the boundary of the rectangle $[\delta,L] \times [0,\lambda_\infty]$ is mapped into the MA space, but some points in its interior are not. Since the rectangle is two-dimensional and the set $H_1 \cap H_2 \subseteq G$ has codimension two, their intersection will generically consist of a finite set of points. If this is the case, the index $\mathfrak m$ can be computed using an arbitrarily small loop around each of these points, as illustrated in Figure \ref{fig:homotopy2flip}. This makes it possible to determine $\mathfrak m$ using purely local information; this is described in Lemma \ref{lem:mgamma}, which will be used repeatedly in Section \ref{sec:app}.

In Section \ref{sec:Turing} we will see that $\mathfrak m$ gives a topological characterization of the Turing bifurcation: as the ratio of diffusion coefficients increases through the critical value, the image of $[\delta,L] \times [0,\lambda_\infty]$ leaves the MA space, after which its boundary has nonzero index.

Most of this section is devoted to the proof of Theorem \ref{thm:cpMorse}. In Section \ref{sec:prelim} we give some preliminary calculations that will be of use here, and also in the applications in Section \ref{sec:app}. In Section \ref{sec:conH} we construct the promised hyperplane space, and in Section \ref{sec:cpproof} we complete the proof by computing the hyperplane indices. Finally, in Section \ref{sec:m}, we explain how the index $\mathfrak m$ can be computed using local information about $W(x,\lambda)$.

\subsection{Preliminary calculations}\label{sec:prelim}
We start by considering the more general system
\begin{equation}\label{Asystem}
	\frac{du}{dx} = A(x) u,
\end{equation}
where $u \in \bbR^{2n}$ and $A(\cdot)$ is a continuous family of real $2n\times2n$ matrices.

We first recall from~\eqref{Homega} that each nonzero $n$-form $\omega$ corresponds to a hyperplane $H_\omega$, whose image under the Pl\"ucker embedding intersects the Grassmannian in the set
\begin{equation}
\label{GPlucker2}
    	\big\{ \spn \{v_1, \ldots ,v_n\} \in Gr_n(V) : \omega(v_1, \ldots, v_n) = 0 \big\}.
\end{equation}
For an oriented $n$-plane $\widetilde W \subseteq \bbR^{2n}$ we then define
\begin{equation}\label{psidef}
	\psi_\omega(\widetilde W) := \frac{\omega(f_1, \ldots, f_n)}{| f_1 \wedge \cdots \wedge f_n|},
\end{equation}
where $(f_1,\ldots, f_n)$ is any positively oriented basis. It follows that $\psi_\omega(\widetilde W) = 0$ if and only if the unoriented subspace $W$ is contained in the set defined in \eqref{GPlucker2}. The denominator of \eqref{psidef} can be computed as $\sqrt{\det G}$, where $G$ denotes the \emph{Gram matrix}, with entries $G_{ij} = \left<f_i, f_j\right>$. For a positive orthonormal basis we have $G_{ij} = \delta_{ij}$ and hence $\psi_\omega(\widetilde W) = \omega(f_1, \ldots, f_n)$.

Since $\omega$ is skew symmetric, we have $\psi_\omega(-\widetilde W) = -\psi_\omega(\widetilde W)$, where $-\widetilde W$ is the oppositely oriented version of $\widetilde W$. For an unoriented subspace $W$, $\psi_\omega(W)$ is therefore only defined up to a sign, but the product and quotient $\psi_1(W) \psi_2(W)$ and $\psi_1(W) / \psi_2(W)$ are both well defined, where $\omega_1$ and $\omega_2$ are any two $n$-forms and we have abbreviated $\psi_i =\psi_{\omega_i}$.

\begin{lemma}\label{lem:psideriv}
Let $W(x)$ be an integral curve of \eqref{Asystem}. If $(f_1,\ldots,f_n)$ is a positive orthonormal basis for $W(x_0)$, then 
\begin{equation}\label{psideriv}
	\frac{d \psi_\omega(\widetilde W) }{dx} \bigg|_{x = x_0} = \sum_{j=1}^n \omega(f_1, \ldots, A(x_0) f_j, \ldots, f_n) - \psi_\omega(\widetilde W) \sum_{j=1}^n \left<A(x_0) f_j, f_j \right>.
\end{equation}
\end{lemma}

\begin{proof}
Write $ \psi_\omega(\widetilde W) = n/d$ where $n$ and $d$ are the numerator and denominator of the expression \eqref{psidef}. Then
 $$  \frac{d\psi_\omega(\widetilde W)} {dx} \bigg|_{x = x_0} = \frac{d n' -n d'} {d^2} \bigg|_{x=x_0} = n'(x_0)-\psi_\omega(\widetilde W) d'(x_0),$$
where we have substituted $d(x_0) =1$ and $n(x_0) = \psi_\omega(\widetilde W)$. Using the fact that $W(x)$ is an integral curve, one easily calculates
\[
	n'(x_0) =  \sum_{j=1}^n \omega(f_1, \ldots, A(x_0) f_j, \ldots, f_n).
\]
Moreover, since $d(x) = \sqrt{\det G(x)}$ and $G(x_0)$ is the identity matrix, Jacobi's formula for the derivative of the determinant yields
\[
	d'(x_0) = \frac12 \tr\left(\frac{dG}{dx}\Big|_{x=x_0} \right) = \sum_{j=1}^n \left<A(x_0) f_j, f_j \right>,
\]
which completes the proof.
\end{proof}

\subsection{Choosing hyperplanes for a reaction--diffusion system}
\label{sec:conH}
We now restrict our attention to the eigenvalue problem \eqref{Lsystem}, letting
\begin{equation}\label{Adef}
	A(x,\lambda) = \begin{pmatrix} 0 & D^{-1} \\ B(x,\lambda) & 0 \end{pmatrix}
\end{equation}
in \eqref{Asystem}, where $D = \diag(d_1,\ldots,d_n)$.

For Dirichlet boundary conditions it is natural to let $H_1 \cap G$ be the train of the Dirichlet subspace. We thus choose $H_1$ to be the hyperplane corresponding to the degenerate $n$-form
\begin{equation}\label{omega1def}
	\omega_1 = e_1^* \wedge \cdots \wedge e_n^*,
\end{equation}
where $e_1, \ldots, e_{2n}$ denotes the standard orthonormal basis for $\bbR^{2n}$.
Since the resulting index equals the geometric intersection number with $H_1 \cap G$, it will count solutions to the Dirichlet problem, which are (by definition) conjugate points. When $n=2$, the two-form $\omega_1$ corresponds to the matrix
\[
	\Omega_1 = \begin{pmatrix} 0 & 1 & 0 & 0 \\ -1 & 0 & 0 & 0 \\
	0& 0 & 0 & 0 \\ 0 & 0 & 0 & 0 \end{pmatrix},
\]
in the sense that $\omega_1(v,w) = v^T \Omega_1 w$ for any $v, w \in \bbR^4$.

The choice of $H_2$ is less obvious. Motivated by the calculation to follow in Lemma \ref{psiRD}, we let
\begin{equation}\label{omega2def}
	\omega_2 = \sum_{j =1}^n \frac{1}{d_j} e_1^* \wedge \cdots \wedge \widehat{ e_j^*} \wedge e_{j+n}^* \wedge \cdots \wedge e_n^*,
\end{equation}
i.e. the $j$th summand is proportional to $\omega_1$ but with $e_j^*$ replaced by $e_{j+n}^*$. When $n=2$ this is 
\[
	\omega_2 = \frac{1}{d_1} e_3^* \wedge e_2^* + \frac{1}{d_2} e_1^* \wedge e_4^*,
\]
corresponding to the matrix
\begin{equation}\label{Omegadd}
	\Omega_2 = \begin{pmatrix} 0 & 0 & 0 & 1/d_2 \\
	0 & 0 & -1/d_1 & 0 \\
	0 & 1/d_1 & 0 & 0 \\
	-1/d_2 & 0 & 0 & 0
	\end{pmatrix}.
\end{equation}
This choice yields a monotonicity result (Lemma \ref{lemma:H2mono}) that is key to the third part of Theorem~\ref{thm:cpMorse}. Moreover, it will play a prominent role in Section \ref{sec:app}, where we prove a long-time invariance result for reaction--diffusion systems with large diffusivities.

We now apply Lemma \ref{lem:psideriv} to the symplectic forms $\omega_1$ and $\omega_2$. To state the result, we additionally define
\begin{equation}\label{omega3def}
	\omega_3 = \sum_{\substack{j,k=1 \\ j < k}}^n \frac{2}{d_j d_k} e_1^* \wedge \cdots \wedge e_{j+n}^* \wedge \cdots \wedge e_{k+n}^* \wedge \cdots \wedge e_{n}^*.
\end{equation}
That is, the $j$-$k$ summand is obtained from $\omega_1$ by replacing $e_j^*$ and $e_k^*$ by $e_{j+n}^*$ and $e_{k+n}^*$, respectively. For $n=2$ we have
\[
	\omega_3 = \frac{2}{d_1d_2} e_3^* \wedge e_4^*,
\]
which is a degenerate two-form whose corresponding hyperplane is the train of the Neumann subspace.

\begin{lemma}\label{psiRD}
Let $W(x,\lambda)$ be an integral curve of $\frac{du}{dx} = A(x,\lambda) u$, with $A(x,\lambda)$ given by \eqref{Adef}, and define $\omega_1$, $\omega_2$ and $\omega_3$ by \eqref{omega1def}, \eqref{omega2def} and \eqref{omega3def}, respectively. Then
\begin{equation}\label{psi1deriv}
	\frac{d\psi_1}{dx} = \psi_2 - \gamma \psi_1
\end{equation}
and
\begin{equation}\label{psi2deriv}
	\frac{d\psi_2}{dx} = \left(  \frac{b_{11}}{d_1} + \cdots +  \frac{b_{nn}}{d_n}\right)\psi_1 + \psi_3 - \gamma \psi_2,
\end{equation}
where $\gamma = \sum_{j=1}^n \left<A f_j, f_j \right>$ and $b_{ij}$ is the $i$-$j$ component of the matrix $B$. Moreover, if $W(x_0,\lambda) = P_D$ is the Dirichlet subspace, then
\begin{equation}
	\psi_1(x_0) = \psi_1'(x_0) = \cdots = \psi_1^{(n-1)}(x_0) = 0 %\quad \text{ and } \quad \psi_1^{(n)}(x_0) = \frac{n!}{d_1 \cdots d_n} \neq 0.
\end{equation}
and
\begin{equation}
	\psi_1^{(n)}(x_0) = \frac{n!}{d_1 \cdots d_n} \neq 0.
\end{equation}
\end{lemma}

\begin{proof}
From Lemma \ref{lem:psideriv} we have
\[
	\frac{d\psi_i}{dx} = \sum_{j=1}^n \omega_i(f_1, \ldots, A f_j, \ldots, f_n)  - \gamma \psi_i
\]
for $i = 1,2$. For $\omega_1$ we observe that
\[
	\omega_1(f_1, \ldots, A f_j, \ldots, f_n) = \big(e_1^* \wedge \cdots \wedge e_j^*A \wedge \cdots \wedge e_n^*\big)(f_1, \ldots, f_n).
\]
The composition $e_j^* A \colon V \to \bbR$ is given by $e_j^* A = \sum_{k=1}^{2n} A_{jk} e_k^*$, hence
\[
	e_j^* A = \frac{1}{d_j} e_{j+n}^*, \qquad e_{j+n}^*A = \sum_{k=1}^n b_{jk} e_k^*,
\]
for any $1 \leq j \leq n$. It follows that
\[
	e_1^* \wedge \cdots \wedge e_j^*A \wedge \cdots \wedge e_n^* 
	= \frac{1}{d_j} e_1^* \wedge \cdots  \wedge e_{j+n}^* \wedge \cdots \wedge e_n^*,
\]
which is precisely the $j$th summand in the definition of $\omega_2$. This implies
\begin{equation}\label{omega2comp}
	\sum_{j=1}^n \omega_1(f_1, \ldots, A f_j, \ldots, f_n) = \omega_2(f_1, \ldots, f_n),
\end{equation}
and completes the proof of \eqref{psi1deriv}.

For \eqref{psi2deriv} we need to compute
\[
	\sum_{j=1}^n \omega_2(f_1, \ldots, A f_j, \ldots, f_n) = \sum_{j,k=1}^n \omega_2^k(f_1, \ldots, A f_j, \ldots, f_n),
\]
where $\omega_2^k := d_k^{-1} e_1^* \wedge \cdots \wedge e_{k+n}^* \wedge \cdots \wedge e_n^*$
denotes the $k$th summand in the definition of $\omega_2$. For summands with $j=k$ we have
\begin{align*}
	\frac{1}{d_j} e_1^* \wedge \cdots \wedge e_{j+n}^*A \wedge \cdots \wedge e_n^* &= \frac{1}{d_j} e_1^* \wedge \cdots \wedge \left( \sum_{l=1}^n b_{jl} e_l^* \right) \wedge \cdots \wedge e_n^* \\
	&= \frac{b_{jj}}{d_j} \omega_1.
\end{align*}
For summands with $j \neq k$ we have
\[
	\frac{1}{d_k} e_1^* \wedge \cdots \wedge e_j^* A \wedge \cdots \wedge e_{k+n}^* \wedge \cdots \wedge e_n^* = \frac{1}{d_k} e_1^* \wedge \cdots \wedge \left( \frac{1}{d_j} e_{j+n}^*\right) \wedge \cdots \wedge e_{k+n}^* \wedge \cdots \wedge e_n^*,
\]
which is precisely the $j$,$k$ term in the definition of $\omega_3$, so the proof of \eqref{psi2deriv} is complete.

To prove the final statement, we recall that $P_D = \spn\{e_{n+1}, \ldots, e_{2n}\}$, so an $n$-form $\omega = e_{j_1}^* \wedge \cdots \wedge e_{j_n}^*$ will vanish on $P_D$ unless $\{j_1, \ldots, j_n\} = \{n+1, \ldots, 2n\}$. In general, suppose $m$ of the indices $j_1, \ldots, j_n$ are contained in $\{n+1,\ldots, 2n\}$, with the remaining $n-m$ in $\{1, \ldots, n\}$. Then, as in
the calculations above, the derivative of $\psi_\omega$ will have terms with $m-1$, $m$ and $m+1$ indices in $\{n+1, \ldots, 2n\}$. To find the first nonvanishing derivative of $\psi_\omega$ on $P_D$, we therefore only need to keep track of the $m+1$ term. We thus compute
\begin{align*}
	\frac{d\psi_1}{dx} &= \psi_2 + \cdots, \\
	\frac{d\psi_2}{dx} &= \psi_3 + \cdots, \\
	\frac{d\psi_3}{dx} &= \psi_4 + \cdots, \qquad \omega_4 := \sum_{\substack{j,k,l=1 \\ j < k < l}}^n \frac{3!}{d_j d_k} e_1^* \wedge \cdots \wedge e_{j+n}^* \wedge \cdots \wedge e_{k+n}^* \wedge \cdots \wedge e_{l+n}^* \wedge \cdots \wedge e_{n}^* \\
	&\ \,  \vdots \\
	\frac{d\psi_n}{dx} &= \psi_{n+1} + \cdots, \qquad \omega_{n+1} := \frac{n!}{d_1 \cdots d_n} e_{n+1}^* \wedge \cdots \wedge e_{2n}^*,
\end{align*}
and the result follows.
\end{proof}

\begin{rem}
The form $\omega_2$ was chosen to make the equality \eqref{omega2comp} hold. For a general matrix $A(x,\lambda)$ there is no guarantee that this can be done with a constant form $\omega_2$. The fact that it is possible here is a consequence of the block structure of $A$ in \eqref{Adef}, and the fact that $\omega_1$ only depends on the first $n$ coordinates.
\end{rem}

\subsection{Positive eigenvalues and conjugate points}
\label{sec:cpproof}
We are now ready to begin the proof of Theorem \ref{thm:cpMorse}. We start with the existence of $\lambda_\infty$.

\begin{lemma}\label{lem:linfinity}
Assuming the hypotheses of Theorem \ref{thm:cpMorse}, there exists $\lambda_\infty > 0$ such that $W(x,\lambda) \cap P_1 = \{0\}$ for all $0 < x \leq L$ and $\lambda \geq \lambda_\infty$. Moreover, every eigenvalue $\lambda \in \sigma(\cL)$ has $\Re \lambda \leq \lambda_\infty$.
\end{lemma}

Note that the property $W(x,\lambda) \cap P_1 = \{0\}$ is only guaranteed for $0 < x \leq L$. It is possible for $W(0,\lambda)$ to intersect $P_1$ nontrivially, for instance if $P_0 = P_1$.

\begin{proof}
Suppose there is a (possibly complex-valued) solution $v$ to
\[
	D v_{xx} + \nabla F(\bar u) v = \lambda v
\]
on $[0,x_*]$, satisfying the boundary conditions
\[
	 \begin{pmatrix} v \\ D v_x \end{pmatrix}\Bigg|_{x=0} \in P_0, \qquad 
	  \begin{pmatrix} v \\ D v_x \end{pmatrix}\Bigg|_{x=x_*} \in P_1.
\]
Since $P_1 = P_D$, this means $v(x_*) = 0$. Similarly, at $x=0$ we have either $v(0) = 0$ or $D v_x(0) = \Theta v(0)$, depending on the choice of $P_0$.

Multiplying the eigenvalue equation by the conjugate of $v$ and integrating by parts, using $v(x_*) = 0$, we find that
\begin{equation}\label{eq:IBP}
	\lambda \int_0^{x_*} |v|^2 \,dx =  - \left<D v_x(0), v(0) \right> + \int_0^{x_*} \big(\left<\nabla F(\bar u) v, v\right> - \left<D v_x, v_x \right> \big) \,dx,
\end{equation}
where $\left<\cdot,\cdot\right>$ denotes the $\bbC^n$ inner product. Defining constants
\[
	d = \min\{d_1, \ldots, d_n\}, \qquad K = \sup_{x \in [0,L]} \big\| \nabla F(\bar u(x)) \big\|,
\]
we obtain
\begin{equation}\label{eq:intbound}
	\Re \int_0^{x_*} \big(\left<\nabla F(\bar u) v, v\right> - \left<D v_x, v_x \right> \big) \,dx \leq K \int_0^{x_*} |v|^2\, dx
	- d \int_0^{x_*} |v_x|^2\, dx.
\end{equation}

To deal with the boundary term in \eqref{eq:IBP}, we treat the Dirichlet and Robin cases separately. If $P_0 = P_D$, then the boundary term vanishes, so we get
\[
	\Re \lambda \leq K
\]
and it suffices to choose any $\lambda_\infty > K$. On the other hand, if $P_0 = P_R$, the boundary term satisfies $\left| \left<D v_x(0), v(0) \right> \right| = \left| \left<\Theta v(0), v(0) \right> \right| \leq C |v(0)|^2$ for some positive constant $C$. Moreover, since $v(x_*) = 0$, we have
\begin{align*}
	|v(0)|^2 =& \left| \int_0^{x_*} \frac{d}{dx} |v(x)|^2 dx \right| \\
	& \leq 2 \int_0^{x_*} |v|  |v_x| \,dx \\
	& \leq \epsilon^{-1} \int_0^{x_*} |v|^2\,dx + \epsilon \int_0^{x_*} |v_x|^2\,dx
\end{align*}
for any $\epsilon > 0$. Choosing $\epsilon = d/C$, and combining the above inequality with \eqref{eq:IBP} and \eqref{eq:intbound}, we obtain
\[
	\Re \lambda \leq  K + \frac{C^2}{d},
\]
which completes the proof.
\end{proof}

This proves the first assertion in Theorem~\ref{thm:cpMorse}. Moving onto the second part, we consider the Maslov--Arnold space $\cM = G \setminus (H_1 \cap H_2)$, with $H_1$ as in Section \ref{sec:conH} and any $H_2 \neq H_1$, and consider the path $W(x,\lambda)$ in $Gr_n(\bbR^{2n})$ defined by \eqref{mudef}.

We first show that the image of the boundary of $[\delta,L] \times [0,\lambda_\infty]$ is contained in $\cM$, and hence its hyperplane index $\mathfrak m$ is well defined.

Referring to Figure \ref{fig:homotopy2flip}, the hypothesis \eqref{xlinv} guarantees that the bottom and right side of the rectangle are mapped into $\cM$ for any $\delta > 0$. Lemma \ref{lem:linfinity} guarantees that the top of the rectangle is mapped into $\cM$, and in fact
\begin{equation}
    \Ind\left(W(x,\lambda_\infty)\big|_{x \in [\delta,L]} ; P_1 \right) = 0, \label{eq:infinity}
\end{equation}
for any $\delta > 0$, by Proposition \ref{ind:properties}(ii). The following lemma guarantees that the left side of the rectangle is also mapped into $\cM$, with
\begin{equation}
    \Ind\left(W(\delta,\lambda)\big|_{\lambda \in [0,\lambda_\infty]} ; P_1 \right) = 0, \label{eq:delta0} 
\end{equation}
provided $\delta$ is sufficiently small.

\begin{lemma}
\label{lem:delta}
There exists $\delta_0>0$ such that $W(x,\lambda) \cap P_1 = \{0\}$ for all $\lambda \in [0,\lambda_\infty]$ and $x \in (0,\delta_0]$.
\end{lemma}

\begin{proof}
Recalling that $W(0,\lambda) = P_0$, there are two cases to consider. If $P_0 = P_R$, then $W(0,\lambda) \cap P_1 = P_R \cap P_D = \{0\}$ for all $\lambda$. Since $W(x,\lambda)$ in continuous in $x$ and $\lambda$, and $[0,\lambda_\infty]$ is compact, there exists $\delta > 0$ such that $W(x,\lambda) \cap P_1 = \{0\}$ for all $\lambda \in [0,\lambda_\infty]$ and $x \in
[0,\delta)$. (Note that $x=0$ is allowed in this case.)

The other case is when $P_0 = P_D$, so $W(0,\lambda) \cap P_1 = P_D \neq \{0\}$. Defining $\eta(x,\lambda) = \psi_1(x,\lambda)^2$, we have
\[
	\eta(0,\lambda) = \cdots = \eta^{(2n-1)}(0,\lambda) = 0 \qquad \text{and} \qquad \eta^{(2n)}(0,\lambda) > 0
\]
from Lemma \ref{psiRD}. Therefore, for fixed $\lambda$ we have $\eta(x,\lambda) > 0$ for sufficiently small $x>0$, and so by compactness there exists $\delta > 0$ such that $\eta(x,\lambda) > 0$ for all $\lambda \in [0,\lambda_\infty]$ and $x \in (0,\delta)$. This completes the proof, since $\eta(x,\lambda) > 0$ implies $\psi_1(x,\lambda) \neq 0$ and hence $W(x,\lambda) \cap P_1 = \{0\}$.
\end{proof}

It follows that the hyperplane index $\mathfrak m$ of the boundary is well defined, and is given by
\begin{equation}
    \mathfrak m = \Ind\left(W(x,0)\big|_{x \in [\delta,L]} ; P_1 \right) + \Ind\left(W(L,\lambda)\big|_{\lambda \in [0,\lambda_\infty]} ; P_1 \right)
\end{equation}
as a result of \eqref{eq:infinity}, \eqref{eq:delta0} and Proposition \ref{ind:properties}(iii). To complete the proof of Theorem \ref{thm:cpMorse}(ii) we use \eqref{eq:gen} to obtain
\begin{align*}
    -\Ind\left(W(L,\lambda)\big|_{\lambda \in [0,\lambda_\infty]} ; P_1 \right) & \leq \# \big\{ \lambda \in [0,\lambda_\infty] : W(1,\lambda) \cap P_1 \neq \{0\} \big\} \\
	&= \# \big\{\text{eigenvalues of } \cL \text{ in } [0,\lambda_\infty] \big\} \\
	&= \# \big\{\text{nonnegative eigenvalues of } \cL \big\},
\end{align*}
where the last equality follows from Lemma \ref{lem:linfinity}.

The following lemma verifies \eqref{eq:MorseMaslov2}, and hence completes the proof of Theorem \ref{thm:cpMorse}. Note that up to this point $H_2$ has been an arbitrary hyperplane different from $H_1$, and did not appear explicitly in any of the preceding calculations.

\begin{lemma}
\label{lemma:H2mono}
If $H_2$ corresponds to the form $\omega_2$ defined in \eqref{omega2def}, then
\[
	\Ind\left(W(x,0)\big|_{x \in [\delta,L]} ; P_1 \right) = \# \big\{\text{conjugate points in $(0,L]$} \big\}
\]
for $0 < \delta \ll 1$.
\end{lemma}

The hyperplane index on the left-hand side is a \emph{signed} count of the $x_* \in [\delta,L]$ for which $W(x_*,0) \cap P_1 \neq \{0\}$. These are conjugate points (by definition) so to prove the lemma we just need to show that they all contribute to the Maslov index with the same sign. This is where the choice of $H_2$ becomes crucial.

\begin{proof}
For simplicity we abbreviate $\psi_i(x) = \psi_i(W(x,0))$.
If $x_* \in [\delta,L]$ is a conjugate point, then $\psi_1(x_*) = 0$, so \eqref{psi1deriv} implies $\psi_1'(x_0) = \psi_2(x_0)$. Substituting this in \eqref{thetap}, we obtain
\[
	\theta'(x_*) = \frac{d}{dx} \frac{\psi_1(v_1,\ldots,v_n)}{\psi_2(v_1,\ldots,v_n)} \bigg|_{x=x_*} = 1 > 0.
\]
Using \eqref{eq:pos}, we conclude that the Maslov index equals the number of conjugate points in $(\delta,L]$, and hence the number of conjugate points in $(0,L]$ if $\delta$ is sufficiently small.
\end{proof}

\subsection{Computing the index $\mathfrak m$}
\label{sec:m}

We now explain how to compute the index $\mathfrak m$ using information about $W(x,\lambda)$ near a point $(x_*,\lambda_*)$ where it leaves the MA space. We first describe what happens when $W$ leaves the MA space.

\begin{lemma}
\label{lem:leave}
Suppose $(x_*,\lambda_*)$ is contained in a neighbourhood $U \subset \bbR^2$ such that $\{(x,\lambda) \in U : \psi_1(x,\lambda) = 0\}$ can be written as the graph of continuously differentiable functions $\lambda_1(x), \ldots, \lambda_n(x)$ with $\lambda_i(x_*)= \lambda_*$ for all $i$. Then $W(x_*,\lambda_*) \in H_1 \cap H_2$ if and only if either
\begin{equation}
\label{dLdetX}
    \frac{\p \psi_1}{\p \lambda}(x_*,\lambda_*) = 0
\end{equation}
or $\lambda_i'(x_*) = 0$ for all $i$.
\end{lemma}

For instance, if two eigenvalue curves $\lambda_i(x)$ and $\lambda_j(x)$ intersect transversely at $(x_*,\lambda_*)$, then \eqref{dLdetX} must hold, otherwise the implicit function would be violated. Some examples are shown in Figure~\ref{fig:mgamma}.

\begin{proof}
Recall that $W(x_*,\lambda_*) \in H_1 \cap H_2$ if and only if $\psi_1(x_*,\lambda_*) = \psi_2(x_*,\lambda_*) = 0$. For each $i$ we have $\psi_1(x,\lambda_i(x)) = 0$ when $x$ is close to $x_*$, and hence
\[
    0 = \frac{\p \psi_1}{\p x}(x,\lambda_i(x)) + \frac{\p \psi_1}{\p \lambda}(x,\lambda_i(x)) \lambda_i'(x).
\]
Moreover, \eqref{psi1deriv} implies $\p \psi_1/\p x = \psi_2$ at any point where $\psi_1 = 0$, so we conclude that $W(x_*,\lambda_*) \in H_1 \cap H_2$ if and only if
\begin{equation}
    \frac{\p \psi_1}{\p \lambda}(x_*,\lambda_i(x_*)) \lambda_i'(x) = 0
\end{equation}
for each $i$, and the claim follows.
\end{proof}

We now give a simple rule for computing the index locally, assuming the zero set of $\psi_1$ can be parameterized as in Lemma \ref{lem:leave}.

\begin{lemma}
\label{lem:mgamma}
Suppose $W(x_*,\lambda_*) \in H_1 \cap H_2$, and assume there is a neighbourhood $U$ of $(x_*,\lambda_*)$ such that $\{(x,\lambda) \in U : \psi_1(x,\lambda) = 0 \}$ is given by the graphs of continuous functions $\lambda_1(x), \ldots, \lambda_n(x)$, only intersecting at the point $(x_*,\lambda_*)$, of which $i_-$ are strictly increasing for $x<x_*$ and  $i_+$ are strictly increasing for $x>x_*$, with the rest strictly decreasing. If $H_2$ is chosen as in Theorem \ref{thm:cpMorse}(iii), then
\begin{equation}
\label{ind:mgamma}
    \Ind(W\circ\gamma_*;P_1) = 2 (i_- - i_+),
\end{equation}
where $\gamma_*$ is a sufficiently small loop around $(x_*,\lambda_*)$ in the $(x,\lambda)$-plane, oriented counterclockwise.
\end{lemma}

Figure \ref{fig:mgamma} illustrates some possibles cases of the theorem, and also shows the idea of the proof.

\begin{proof}
By homotopy invariance we can assume that $\gamma_*$ is a small rectangle centered at $(x_*,\lambda_*)$. Moreover, since every curve $\lambda_i(x)$ is either strictly increasing or strictly decreasing to the left and right of $x_*$, we can assume that they only intersect the top and bottom of the rectangle, and not its sides, by making it sufficiently thin in the $\lambda$ direction.

Since these curves only intersect at $(x_*,\lambda_*)$, it follows that there are $i_- + (n - i_+)$ distinct crossings on the bottom of the rectangle, and $(n - i_-) + i_+$ on the top. Using Lemma \ref{lemma:H2mono}, we conclude that
\[
    \Ind(W \circ\gamma_*; P_1) = \big(i_- + (n - i_+)\big) - \big((n - i_-) + i_+\big)
    = 2(i_- - i_+),
\]
as claimed.
\end{proof}

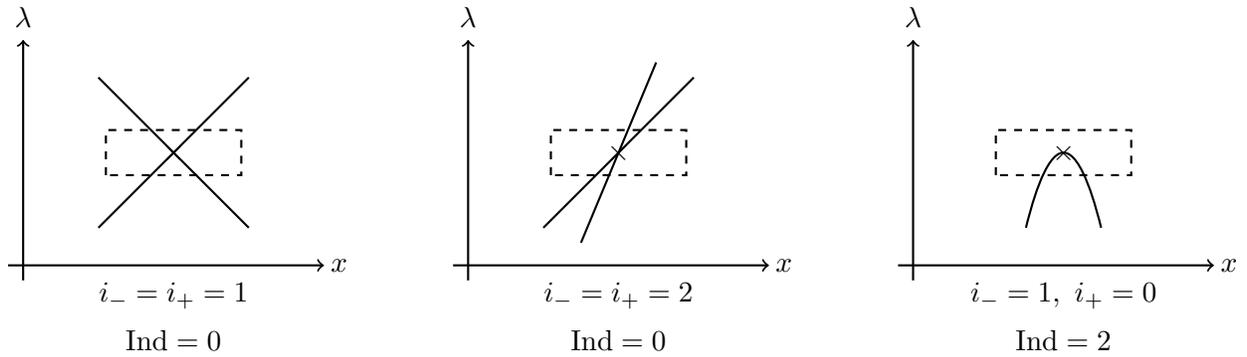
\begin{figure}
\begin{tikzpicture}
	\draw[->,thick] (0,-0.2) to (0,3);
	\draw[->,thick] (-0.2,0) to (4,0);
	\node at (4.2,0) {$x$};
	\node at (0,3.3 ) {$\lambda$};
	\draw[thick,dashed] (1.1,1.2) -- (1.1,1.8) -- (2.9,1.8) -- (2.9,1.2) -- (1.1,1.2);
	\draw[thick] (1,0.5) -- (3,2.5);
	\draw[thick] (3,0.5) -- (1,2.5);
	\node at (2,-0.4) {$i_- = i_+ = 1$};
	\node at (2,-1) {$\Ind = 0$};
	\node at (2,1.5) {$\times$};
%	\draw[thick] (2.4,2) arc (0:180:0.4);
%	\draw[->,thick] (1.6,2) arc (180:360:0.4);
%	\node at (2.7,1.9) {$\gamma$};
\end{tikzpicture}
\hfill
\begin{tikzpicture}
	\draw[->,thick] (0,-0.2) to (0,3);
	\draw[->,thick] (-0.2,0) to (4,0);
	\node at (4.2,0 ) {$x$};
	\node at (0,3.3 ) {$\lambda$};
	\draw[thick,dashed] (1.1,1.2) -- (1.1,1.8) -- (2.9,1.8) -- (2.9,1.2) -- (1.1,1.2);
	\draw[thick] (1,0.5) -- (3,2.5);
	\draw[thick] (1.5,0.3) -- (2.5,2.7);
	\node at (2,-0.4) {$i_- = i_+ = 2$};
	\node at (2,-1) {$\Ind = 0$};
	\node at (2,1.5) {$\times$};
\end{tikzpicture}
\hfill
\begin{tikzpicture}
	\draw[->,thick] (0,-0.2) to (0,3);
	\draw[->,thick] (-0.2,0) to (4,0);
	\node at (4.2,0 ) {$x$};
	\node at (0,3.3 ) {$\lambda$};
	\draw[thick,dashed] (1.1,1.2) -- (1.1,1.8) -- (2.9,1.8) -- (2.9,1.2) -- (1.1,1.2);
%	\draw[->,thick] (2,1.5) arc (0:90:0.5);
	\draw[thick] (2,1.5) parabola (2.5,0.5);
	\draw[thick] (2,1.5) parabola (1.5,0.5);
	\node at (2,-0.4) {$i_- = 1, \ i_+ = 0$};
	\node at (2,-1) {$\Ind = 2$};
	\node at (2,1.5) {$\times$};
\end{tikzpicture}
\caption{Some possible configurations in Lemma \ref{lem:mgamma} and the corresponding indices. The curves $\lambda_1(x), \ldots, \lambda_n(x)$ are denoted by solid lines, and the loop $\gamma_*$ is the dashed box. The index of the loop $W\circ\gamma_*$, given by \eqref{ind:mgamma}, equals the number of crossings on the bottom of the box minus the number of crossings on the top.}
\label{fig:mgamma}
\end{figure}

\section{Applications to reaction--diffusion systems}
\label{sec:app}

We now apply our results to various systems of reaction--diffusion equations. In each case, the main difficulty is verifying the invariance conditions of either Theorem \ref{thm:cpMorse} or Corollary \ref{cor:m0}. In Section~\ref{sec:largeD} we do this under the assumption that one of the diffusion coefficients is large, relative to the size of the domain. In Section \ref{sec:cc} we specialize to the case of homogeneous equilibria, in which case the linearization has constant coefficients and the invariance conditions can be verified by explicit computation. Finally, in Section \ref{sec:Turing} we use our constant-coefficient results to analyze the Turing instability mechanism.

\subsection{Systems with large diffusion}\label{sec:largeD}
Consider the eigenvalue problem with mixed boundary conditions
\begin{equation}\label{mixedBVP}
	D \frac{d^2 u}{dx^2} + V u = \lambda u, \quad u'(0) = u(L) = 0,
\end{equation}
recalling that $D = \diag(d_1, \ldots, d_n)$.
The corresponding boundary subspaces are
\[
	P_0 = \{(q,0) : q \in \bbR^n\}, \quad P_1 = \{(0,p) : p \in \bbR^n\},
\]
and so $x_*$ is a conjugate point if and only if there exists a nontrivial solution to the boundary value problem
\[
	D \frac{d^2 u}{dx^2} + V u = 0, \quad u'(0) = u(x_*) = 0.
\]

Our result is that Corollary \ref{cor:m0} applies to the above system as long as none of the $d_j$ are too small, and all of the products $d_j d_k$ with $j \neq k$ are sufficiently large.

\begin{theorem}\label{thm:largediff}
Fix $L$ and $d_* >0$, and suppose $V \in C[0,L]$. There exists a constant $\Delta>0$ such that if $d_j \geq d_*$ for all $j$ and $d_j d_k \geq \Delta$ for $j \neq k$, then the hypotheses of Corollary \ref{cor:m0} are satisfied, and hence
\begin{equation}\label{unstable:conjugate2}
	\# \big\{\text{positive real eigenvalues of \eqref{mixedBVP}} \big\} \geq \# \big\{\text{conjugate points in } (0,L) \big\}.
\end{equation}
\end{theorem}

The constant $\Delta$ depends on $L$, $d_*$ and $V$, and can be estimated from the proof if desired. In particular, we see that it suffices to choose $\Delta = (2/C)e^{CL}$, where $C$ is a constant depending only on $d_*$ and $V$.

\begin{proof}
From Lemma \ref{lem:linfinity} we see that $\lambda_\infty$ can be any number satisfying
\[
	\lambda_\infty > \sup_{x \in [0,L]} \big\| V(x) \big\|.
\]
In particular, it can be chosen independent of $D$.

We now use Lemma \ref{psiRD}, with $B(x,\lambda) = \lambda I - V(x)$. Define
\begin{equation}
	\rho = \frac12 \big(\psi_1^2 + \psi_2^2),
\end{equation}
so that $\rho(0,\lambda) = 1/2$. It follows that
\[
	\frac{d\rho}{dx} = - 2 \gamma \rho + \left(1 + \frac{b_{11}}{d_1} + \cdots + \frac{b_{nn}}{d_n}\right) \psi_1 \psi_2 + \psi_2 \psi_3.
\]
From the definition of $\gamma$ (in Lemma \ref{psiRD}) we obtain
\[
	|\gamma(x,\lambda) | \leq n \|A(x,\lambda) \| \leq n \left( \| B(x,\lambda) \| +  \|D^{-1} \| \right) \leq n \left( \max \|B(x,\lambda)\| + \frac{1}{d_*} \right) =: C_1
\]
where the maximum is taken over $(x,\lambda) \in [0,L] \times [0,\lambda_\infty]$.
We similarly have
\[
	\left| \left(1 + \frac{b_{11}}{d_1} + \cdots + \frac{b_{nn}}{d_n}\right) \psi_1 \psi_2 \right| \leq \underbrace{\left(1 +  \frac{\max |b_{11}(x,\lambda)|}{d_*} + \cdots + \frac{\max |b_{nn}(x,\lambda)|}{d_*} \right)}_{C_2} \rho.
\]
Moreover, using
\[
	| \psi_3 | \leq \sum_{\substack{j,k=1 \\ j < k}}^n \frac{2}{d_j d_k} \leq \frac{n(n-1)}{\Delta},
\]
we obtain
\[
	| \psi_2 \psi_3 | \leq |\psi_2| \frac{n(n-1)}{\Delta}  
	\leq \underbrace{ \frac12 \left(\frac{n(n-1)}{d_*} \right)^2}_{C_3} \rho + \frac{1}{\Delta},
\]
and hence
$
	\frac{d\rho}{dx} \geq -C \rho -\Delta^{-1},
$
where $C = 2C_1 + C_2 + C_3$ depends only on $d_*$ and $V$. This is equivalent to
\[
	\frac{d}{dx} \left( e^{Cx} \rho(x)\right) \geq - \frac{e^{Cx}}{\Delta}
\]
so we have
\[
	e^{Cx} \rho(x) - \frac12 \geq - \frac{1}{\Delta} \int_0^x e^{Ct} dt = \frac{1 - e^{Cx}}{C \Delta}.
\]
Therefore, we will have $\rho(x,\lambda) > 0$ for $\lambda \in [0,\lambda_\infty]$ provided
\[
	e^{Cx} < 1 + \frac{C \Delta}{2}.
\]
This equality will hold for all $(x,\lambda) \in [0,L] \times [0,\lambda_\infty]$ if it holds when $x=L$. Therefore, we need $e^{CL} < 1 + \frac{C \Delta}{2}$. This is satisfied for a sufficiently large choice of $\Delta$,  depending only on $L$ and $C$ (i.e. on $L$, $d_*$ and $V$).
\end{proof}

\subsection{Stability of homogeneous equilibria}
\label{sec:cc}
If the steady state $\bar u$ is homogeneous (constant in $x$) and $D=I$, then the linearized operator \eqref{Ldef} has the form 
\[
	\cL = \frac{d^2}{dx^2} + V
\]
where $V = \nabla F(\bar u) \in M_n(\bbR)$ is a constant real matrix. The case of unequal diffusivities, $D \neq I$, can be treated by similar methods but the calculations are more involved; see Section \ref{sec:Turing} for an example.

Consider the Dirichlet problem on $(0,L)$,
\begin{equation}
	\cL v = \lambda v, \quad v(0) = v(L) = 0 \in \bbR^2,
\end{equation}
with constant potential $V \in M_2(\bbR)$. It follows from a direct computation that
\begin{equation}\label{cc:Morse}
	\# \big\{\text{positive eigenvalues of } \cL \big\} = \# \big\{\text{conjugate points in $(0,L)$} \big\}.
\end{equation}

Here we reconsider this problem using the machinery developed in the previous section, to see if the same conclusion can be obtained using our hyperplane index. We first require a definition.

\begin{define}
\label{def:Vgeneric}
We say that $V \in M_2(\bbR)$ is \emph{non-generic for $(0,L)$} if either of the following conditions hold:
\begin{enumerate}
\item its eigenvalues $\nu_1$ and $\nu_2$ are positive and satisfy
\begin{equation}\label{b12ratio}
	\frac{\nu_1}{\nu_2} = \left( \frac{m}{n} \right)^2
\end{equation}
for some integers $m$ and $n$ with
\begin{equation}\label{mnrange}
	1 \leq m \leq \frac{\sqrt{\nu_1} L}{\pi}, \quad 1 \leq n \leq \frac{\sqrt{\nu_2} L}{\pi};
\end{equation}
\item $\nu_1$ and $\nu_2$ satisfy
\begin{equation}
\label{nu21equal}
    \nu_1 - \nu_2 = (m^2 - n^2)(\pi/L)^2
\end{equation}
for some integers $m$ and $n$.
\end{enumerate}
Otherwise $V$ is said to be \emph{generic for $(0,L)$}.
\end{define}

For each $L$ the set of generic matrices is clearly open and dense in  $M_2(\bbR)$. For a problem on an unbounded domain it is natural to approximate by a sequence of of bounded domains, for instance $(0,\infty)$ by $(0,N)$ for $N \to \infty$. While we do not consider such problems here, we note in passing that the set
\begin{equation}
    \big\{ V \in M_2(\bbR) : V \text{ is generic for } (0,N) \text{ for every } N \in \bbN \big\}
\end{equation}
is a countable intersection of open, dense sets, and hence is residual. Finally, we note that for any value of $L$, \eqref{nu21equal} forbids the possibility that $\nu_1 = \nu_2$.

Our result is the following.

\begin{theorem}\label{thm:ccMorseMaslov}
If $V$ is generic for $(0,L)$, in the sense of Definition \ref{def:Vgeneric},     then Theorem \ref{thm:cpMorse} applies, and we have
\begin{align*}
	\# \big\{\text{nonnegative eigenvalues of } \cL \big\} &= - \Ind\left(W(L,\cdot)\big|_{\lambda \in [0,\lambda_\infty]} ; P_1 \right) \\
	& = - \Ind\left(W(\cdot,0)\big|_{x \in [\delta,L]} ; P_1 \right) \\
	&= \# \big\{\text{conjugate points in $(0,L]$} \big\}
\end{align*}
and
\[
	\# \big\{\text{positive eigenvalues of } \cL \big\} = \# \big\{\text{conjugate points in $(0,L)$} \big\}.
\]
\end{theorem}

The proof consists of three steps. First, we show that $W(x,\lambda)$ satisfies the invariance condition \eqref{xlinv} in Theorem \ref{thm:cpMorse}, and hence
\begin{align*}
	\# \big\{\text{nonnegative eigenvalues of } \cL \big\} &\geq - \Ind\left(W(L,\cdot)\big|_{\lambda \in [0,\lambda_\infty]} ; P_1 \right) \\
	& =  \Ind\left(W(\cdot,0)\big|_{x \in [\delta,L]} ; P_1 \right) - \mathfrak m \\
	&= \# \big\{\text{conjugate points in $(0,L]$} \big\}
	- \mathfrak m.
\end{align*}
Next, we show that the path $W(L,\lambda)$ is monotone in $\lambda$, which implies
\[
	\# \big\{\text{nonnegative eigenvalues of } \cL \big\}  = - \Ind\left(W(L,\lambda)\big|_{\lambda \in [0,\lambda_\infty]} ; P_1 \right).
\]
Finally, we use Lemma \ref{lem:mgamma} to show that $\mathfrak m=0$, which completes the proof.

We write the eigenvalue problem in the general form
\begin{align}\label{eq:constant}
	\frac{d}{dx} \begin{pmatrix} v \\ w \end{pmatrix} = \begin{pmatrix} 0 & I \\ B(\lambda) & 0 \end{pmatrix} \begin{pmatrix} v \\ w \end{pmatrix},
\end{align}
where $B(\lambda) \in M_2(\bbR)$ does not depend on $x$. Later we will set $B(\lambda) = \lambda I - V$. As above, we define a family of two-dimensional subspaces
\begin{equation}\label{ccmudef}
	W(x,\lambda) = \left\{ \begin{pmatrix} v(x) \\ w(x) \end{pmatrix} : \begin{pmatrix} v \\ w \end{pmatrix} 
	\text{ satisfies \eqref{eq:constant} and } v(0) = 0 \right\} \subseteq \bbR^4
\end{equation}
for $x \geq 0$.

The system \eqref{eq:constant} is of the form considered in Section \ref{sec:counting}, with $d_1 = d_2 = 1$, so we choose $\omega_1$ and $\omega_2$ corresponding to the matrices
\begin{equation}\label{Omega1d}
	\Omega_1 = \begin{pmatrix} 0 & 1 & 0 & 0 \\ -1 & 0 & 0 & 0 \\
	0& 0 & 0 & 0 \\ 0 & 0 & 0 & 0 \end{pmatrix}, \qquad 
	\Omega_2 = \begin{pmatrix} 0 & 0 & 0 & 1 \\
	0 & 0 & -1 & 0 \\
	0 & 1 & 0 & 0 \\
	-1 & 0 & 0 & 0
	\end{pmatrix}.
\end{equation}
Let $H_1$ and $H_2$ denote the corresponding hyperplanes, and $\cM = G \setminus (H_1 \cap H_2)$ the resulting Maslov--Arnold space.

\begin{prop}\label{prop:inv}
Let $W(x,\lambda)$ be defined by \eqref{ccmudef}. For $x \neq 0$ we have $W(x,\lambda) \in H_1 \cap H_2$ if and only if the eigenvalues $\beta_1,\beta_2$ of $B(\lambda)$ are real and negative and satisfy
\[
	\sin\sqrt{-\beta_1} x = \sin\sqrt{-\beta_2} x = 0.
\]
\end{prop}

\begin{proof}%[Proof of Proposition ...]
We first compute a frame for $W(x,\lambda)$. A frame for a two-dimensional subspace $W$  is (by definition) a $4\times2$ matrix whose columns span $W$. Writing this as
\[
	\begin{pmatrix} X \\ Y \end{pmatrix} = \begin{pmatrix} x_{11} & x_{12} \\ x_{21} & x_{22} \\
	y_{11} & y_{12} \\ y_{21} & y_{22} \end{pmatrix},
\]
and denoting the columns by $u_1$ and $u_2$, we compute
\begin{align*}
    \omega_1(u_1,u_2) &= x_{11} x_{22} - x_{12} x_{21} \\
    &= \det X
\end{align*}
and
\begin{align*}
    \omega_2(u_1,u_2) &= x_{11} y_{22} - x_{21} y_{12} + y_{11} x_{22} - y_{21} x_{12} \\
    &= \det(X + Y) - \det X - \det Y.
\end{align*}
It follows that
\[
	W \in H_1 \ \Longleftrightarrow \ \det X = 0
\]
and
\[
	W \in H_2 \ \Longleftrightarrow \ \det (X+Y) = \det X + \det Y.
\]

Note that $W(x,\lambda)$ is spanned by the last two columns of the fundamental solution matrix $e^{Ax}$, where $A = \left( \begin{smallmatrix} 0 & I \\ B & 0 \end{smallmatrix} \right)$.
We thus compute
\begin{align*}
	e^{Ax}
	&= \sum_{m=0}^\infty \frac{1}{(2m)!} \begin{pmatrix} B^m x^{2m} & 0 \\ 0 & B^m x^{2m} \end{pmatrix} 
	+ \sum_{m=0}^\infty \frac{1}{(2m+1)!} \begin{pmatrix} 0 & B^{m} x^{2m+1}  \\ B^{m+1} x^{2m+1} & 0 \end{pmatrix}
\end{align*}
to conclude that a frame for $W(x,\lambda)$ is given by
\begin{align}\label{frameXY}
	\begin{pmatrix} X \\ Y \end{pmatrix} = 
	\sum_{m=0}^\infty \begin{pmatrix}
	 \frac{B^{m} x^{2m+1} }{(2m+1)!} \\
	\frac{B^m x^{2m}}{(2m)!}
	\end{pmatrix}
	= \begin{pmatrix}
	B^{-1/2} \sinh(\sqrt B x) \\
	\cosh(\sqrt{B} x)
	\end{pmatrix}.
\end{align}
The functions on the right-hand side are defined by their power series, which converge for all numbers $x$ and matrices $B$.

Letting $\beta_1$ and $\beta_2$ denote the eigenvalues of $B(\lambda)$, it follows that $W(x,\lambda) \in H_1 \cap H_2$ if and only if
\begin{align}\label{detXY}
	\det X = \frac{\sinh(\sqrt\beta_1 x) }{\sqrt\beta_1} \frac{\sinh(\sqrt\beta_2 x)}{\sqrt\beta_2} 
	= 0
\end{align}
and
\begin{align}\label{traceXY}
	\det(X+Y) - \det X - \det Y = \frac{\sinh(\sqrt\beta_1 x) }{\sqrt\beta_1} \cosh(\sqrt\beta_2 x)
	+ \frac{\sinh(\sqrt\beta_2 x)}{\sqrt\beta_2} \cosh(\sqrt\beta_1 x) = 0.
\end{align}
As in \eqref{frameXY}, the functions $\beta^{-1/2} \sinh(\sqrt\beta x)$ and $\cosh(\sqrt\beta x)$ are defined by power series which converge for all values of $x$ and $\beta$. In particular, when $\beta=0$ we obtain $\beta^{-1/2} \sinh(\sqrt\beta x) = x$, and when $x = 0$ we obtain $\beta^{-1/2} \sinh(\sqrt\beta x) = 0$ for any value of $\beta$.

Now suppose that $W(x,\lambda) \in H_1 \cap H_2$, so both \eqref{detXY} and \eqref{traceXY} are satisfied. If $\beta_1^{-1/2} \sinh(\sqrt\beta_1 x) = 0$, then $\beta_1 \neq 0$, hence $\sinh(\sqrt\beta_1 x) = 0$ and so $\cosh(\sqrt\beta_1 x) \neq 0$. Then \eqref{traceXY} implies $\beta_2^{-1/2} \sinh(\sqrt\beta_2 x) = 0$, hence $\beta_2 \neq 0$ and
$\sinh(\sqrt\beta_2 x) = 0$. Therefore, $W(x,\lambda) \in H_1 \cap H_2$ if and only if $\sinh(\sqrt\beta_1 x) = \sinh(\sqrt\beta_2 x) = 0$, which is possible if and only if $\beta_1$ and $\beta_2$ are both real and negative and satisfy $\sin(\sqrt{-\beta_1} x) = \sin(\sqrt{-\beta_2} x) = 0$.
\end{proof}

\begin{rem}\label{rem:3_cases}
The above calculations show that $x_*$ is a conjugate point if and only if at least one of the eigenvalues of $X(x_*,0)$ is zero, whereas $W(x_*,0) \in H_1 \cap H_2$ if and only if both eigenvalues are zero, so there are three possibilities:
\begin{enumerate}
	\item The eigenvalues of $X$ do not vanish for any $x_* \in (0,L]$, so the index is zero.
	\item For some $x_* \in (0,L]$ both eigenvalues of $X$ vanish, so the index is not defined.
	\item For some $x_* \in (0,L]$ exactly one eigenvalue of $X$ vanishes, so the index is nonzero.
\end{enumerate}
All three cases will arise in Section \ref{sec:Turing}, when we use our index theory to characterize the Turing instability.
\end{rem}

Proposition \ref{prop:inv} implies the following.

\begin{cor}\label{cor:ccinv_v2}
If $B(\lambda) = \lambda I - V$ and $V$ is generic, then 
\begin{equation}
    \label{xlinv2}
    W(x,0) \in G \setminus (H_1 \cap H_2) \quad \text{ and } \quad W(L,\lambda) \in G \setminus (H_1 \cap H_2)
\end{equation}
for all $x \in (0,L]$ and all $\lambda \in \bbR$.
\end{cor}

\begin{proof}
The eigenvalues of $B(\lambda)$ are given by $\beta_i(\lambda) = \lambda - \nu_i$. If $W(x,0) \in H_1 \cap H_2$ for some $0 < x \leq L$, then $\nu_1$ and $\nu_2$ are both positive and satisfy $\sin \sqrt{\nu_1}x = \sin \sqrt{\nu_2}x = 0$. This implies $\nu_1 = (m\pi/x)^2$ and $\nu_2 = (n\pi/x)^2$, where
\[
    1 \leq m =\frac{\sqrt{\nu_1} x}{\pi} \leq \frac{\sqrt{\nu_1} L}{\pi}
\]
and likewise for $n$, which is not possible if $V$ is generic. Similarly, if $W(L,\lambda) \in H_1 \cap H_2$ for some $\lambda$, we have $\nu_1 = \lambda + (m\pi/L)^2$ and $\nu_2 = \lambda + (n\pi/L)^2$, hence $\nu_1 - \nu_2 = (m^2 - n^2)(\pi/L)^2$, which is not possible if $V$ is generic.
\end{proof}

This first part of the corollary can be visualized as in Figure \ref{fig:lattice}. The condition \eqref{b12ratio} is satisfied if and only if the line through $(0,0)$ and $(\sqrt{\nu_1}, \sqrt{\nu_2})$ intersects one of the indicated lattice points.

\begin{figure}[!tbp]
	\begin{tikzpicture}[scale=0.9]
%		\draw[thick] (0,3) -- (5,3); % top of rectangle
%		\draw[thick] (5,0) -- (5,3); % right side of rectangle
%		\draw[thick] (6.5,0) arc (0:90:6.5 and 4.7);
		\draw[very thick,->] (0,0) -- (7,0); % m-axis
		\node at (7.5,0) {$m$};
		\draw[very thick] (6,-0.15) -- (6,0.15);
		\node at (6,-0.7) {$\lfloor \frac{\sqrt{\nu_1} L}{\pi} \rfloor$};
		\draw[very thick,->] (0,0) -- (0,5); % n-axis
		\node at (0,5.4) {$n$};
		\draw[very thick] (-0.15,4) -- (0.15,4);
		\node at (-1,4) {$\lfloor \frac{\sqrt{\nu_2} L}{\pi} \rfloor$};
		\draw[thick,dashed] (0,0) -- (5.5,4.6);
		\fill (3,2.5) circle[radius=3pt];
	%	\node at (3.5,2.5) {$(\sqrt{-\beta_1}, \sqrt{-\beta_2})$};
		\draw (1,1) circle[radius=2pt];
		\draw (1,2) circle[radius=2pt];
		\draw (1,3) circle[radius=2pt];
		\draw (1,4) circle[radius=2pt];
		\draw (2,1) circle[radius=2pt];
		\draw (2,2) circle[radius=2pt];
		\draw (2,3) circle[radius=2pt];
		\draw (2,4) circle[radius=2pt];
		\draw (3,1) circle[radius=2pt];
		\draw (3,2) circle[radius=2pt];
		\draw (3,3) circle[radius=2pt];
		\draw (3,4) circle[radius=2pt];
		\draw (4,1) circle[radius=2pt];
		\draw (4,2) circle[radius=2pt];
		\draw (4,3) circle[radius=2pt];
		\draw (4,4) circle[radius=2pt];
		\draw (5,1) circle[radius=2pt];
		\draw (5,2) circle[radius=2pt];
		\draw (5,3) circle[radius=2pt];
		\draw (5,4) circle[radius=2pt];
		\draw (6,1) circle[radius=2pt];
		\draw (6,2) circle[radius=2pt];
		\draw (6,3) circle[radius=2pt];
		\draw (6,4) circle[radius=2pt];
	\end{tikzpicture}
\caption{Illustrating the result of Corollary \ref{cor:ccinv_v2}: \eqref{b12ratio} is satisfied if and only if the line through $(0,0)$ and $(\sqrt{\nu_1}, \sqrt{\nu_2})$ intersects a lattice point $(m,n)$ with $m$ and $n$ as in \eqref{mnrange}.}
\label{fig:lattice}
\end{figure}
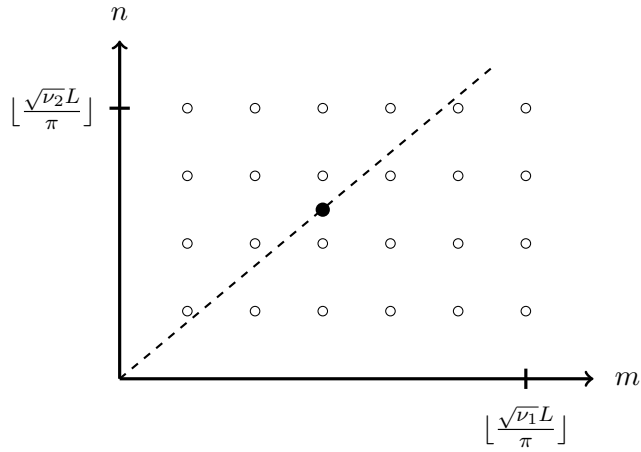

Given the invariance result of Corollary \ref{cor:ccinv_v2}, we can now apply Theorem \ref{thm:cpMorse} to obtain
\begin{align*}
	\# \big\{\text{nonnegative eigenvalues of } \cL \big\} &\geq - \Ind\left(W(L,\cdot)\big|_{\lambda \in [0,\lambda_\infty]} ; P_1 \right) \\
	&= \# \big\{\text{conjugate points in $(0,L]$} \big\}
	-\mathfrak m.
\end{align*}
We next show that the above inequality is in fact an equality.

\begin{lemma}
\label{lem:lambdamono}
Assuming the hypotheses of Theorem \ref{thm:ccMorseMaslov}, we have
\[
	\# \big\{\text{nonnegative eigenvalues of } \cL \big\} = - \Ind\left(W(L,\cdot)\big|_{\lambda \in [0,\lambda_\infty]} ; P_1 \right) 
\]
for sufficiently large $\lambda_\infty$.
\end{lemma}

\begin{proof}
It is enough to show that the curve $\lambda \mapsto W(L,\lambda)$ is negative.
%, i.e. its image in $\bbR P^1$ under the map $\phi$ defined in \eqref{phidef} always passes though the point $[0:1]$ in the negative (clockwise) direction). 
Using \eqref{eq:neg}, this will imply
\begin{align*}
	\Ind\left(W(L,\lambda)\big|_{\lambda \in [0,\lambda_\infty]} ; P_1 \right) & = -\# \big\{ \lambda \in [0,\lambda_\infty) : W(L,\lambda) \cap P_1 \neq \{0\} \big\} \\
	&= -\# \big\{\text{eigenvalues of } \cL \text{ in } [0,\lambda_\infty) \big\} \\
	&= -\# \big\{\text{nonnegative eigenvalues of } \cL \big\}
\end{align*}
and hence complete the proof.

We prove monotonicity using \eqref{thetap}. For convenience we abbreviate $\psi_i(\lambda) = \psi_i(W(L,\lambda))$. From the computations in Proposition \ref{prop:inv} we have
\[
	%\frac{\omega_1(v,w)}{\omega_2(v,w)} = 
	\frac{\psi_1(\lambda)}{\psi_2(\lambda)} = \frac{\det X}{\det(X+Y) - \det X - \det Y}
\]
and so
\[
	\frac{\psi_1'(\lambda_*)}{\psi_2(\lambda_*)} = \frac{\frac{d}{d\lambda} \det X}{\det(X+Y) - \det Y}
\]
at any point $\lambda_*$ where $\det X = 0$.

To differentiate $\det X$, as given by \eqref{detXY}, we first observe that
\[
	\frac{d}{d\lambda} \frac{\sinh(\sqrt{\lambda - \nu} x)}{\sqrt{\lambda - \nu}} = \frac{1}{2(\lambda - \nu)} \left( x \cosh(\sqrt{\lambda - \nu} x) - \frac{\sinh(\sqrt{\lambda - \nu} x)}{\sqrt{\lambda - \nu}} \right).
\]
If $\sinh(\sqrt{\lambda_* - \nu_1} L) = 0$, then
\[
	\frac{d}{d\lambda} \det X \Big|_{\lambda = \lambda_*} = \frac{L}{2(\lambda_* - \nu_1)} \cosh(\sqrt{\lambda_* - \nu_1} L)   \frac{\sinh(\sqrt{\lambda_* - \nu_2} L)}{\sqrt{\lambda_* - \nu_2}}.
\]
Similarly, using \eqref{traceXY} we obtain
\[
	\big(\det(X+Y) - \det Y\big) \Big|_{\lambda = \lambda_*} = \frac{\sinh(\sqrt{\lambda_* - \nu_2} L)}{\sqrt{\lambda_* - \nu_2}} \cosh(\sqrt{\lambda_* - \nu_1} L)
\]
and hence
\[
	\frac{\psi_1'(\lambda_*)}{\psi_2(\lambda_*)}= \frac{L}{2(\lambda_* - \nu_1)} < 0
\]
where $\lambda_* - \nu_1 < 0$ because $\sinh(\sqrt{\lambda - \nu_1} L) = 0$. The case $\sinh(\sqrt{\lambda_* - \nu_2} L) = 0$ is identical.
\end{proof}

Finally, we prove that $\mathfrak m = 0$.

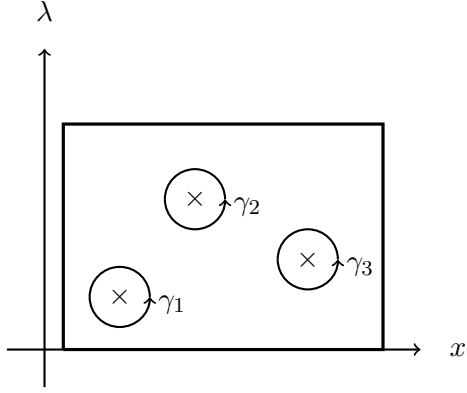
\begin{figure}
\begin{tikzpicture}
%	\draw[- >,>=stealth',thick] (0,-1.5) arc[radius=1.5, start angle=-90, end angle=0];
	\draw[->,thick] (0,-0.5) to (0,4);
	\draw[->,thick] (-0.5,0) to (5,0);
	\node at (5.5,0 ) {$x$};
	\node at (0,4.5 ) {$\lambda$};
	\draw[very thick] (0.25,0) -- (0.25,3) -- (4.5,3) -- (4.5,0) -- (0.25,0);
% 	\node at (-0.8,1.6) {$W(\delta,\lambda)$};
% 	\node at (5.5,1.6) {$W(L,\lambda)$};
% 	\node at (2.3,3.4) {$W(x,\lambda_\infty)$};
% 	\node at (2.3,0.4) {$W(x,0)$};
 	\node at (1,0.7) {$\times$};
	\draw[thick] (1.4,0.7) arc (0:180:0.4);
	\draw[->,thick] (0.6,0.7) arc (180:360:0.4);
	\node at (1.7,0.6) {$\gamma_1$};
	\node at (2,2) {$\times$};
	\draw[thick] (2.4,2) arc (0:180:0.4);
	\draw[->,thick] (1.6,2) arc (180:360:0.4);
	\node at (2.7,1.9) {$\gamma_2$};
    \node at (3.5,1.2) {$\times$};
	\draw[thick] (3.9,1.2) arc (0:180:0.4);
	\draw[->,thick] (3.1,1.2) arc (180:360:0.4);
	\node at (4.2,1.1) {$\gamma_3$};
\end{tikzpicture}
\caption{Illustrating the proof of Lemma \ref{lem:m0}: if $W(x,\lambda)$ leaves the MA space at three points, then $\mathfrak m$ is sum of the indices of the loops $\gamma_1$, $\gamma_2$ and $\gamma_3$.}
\label{fig:homotopy3}
\end{figure}

\begin{lemma}
\label{lem:m0}
Assuming the hypotheses of Theorem \ref{thm:ccMorseMaslov}, we have $\mathfrak m = 0$.
\end{lemma}

\begin{proof}
From \eqref{detXY} we see that the set $\{(x,\lambda) : \psi_1(x,\lambda)=0\}$ is the union of the curves
\begin{equation}
\label{lambdan}
    \lambda_{n,i}(x) = \nu_i - \left(\frac{n\pi}{x}\right)^2
\end{equation}
for all $n \in \bbN$ and $i=1,2$. Moreover, Proposition \ref{prop:inv} show that $W(x_*,\lambda_*) = 0$ if and only if $\lambda_{m,1}(x_*) = \lambda_{n,2}(x_*) = \lambda_*$ for some $m$ and $n$. This occurs when
\[
    x_* = \pi \sqrt{\frac{m^2-n^2}{\nu_1 - \nu_2}},
\]
where we recall that $\nu_1 \neq \nu_2$ because $V$ is generic. The set of all such $x_*$ is discrete, and so it suffices to consider a single point $(x_*,\lambda_*) \in (\delta,L) \times (0,\lambda_\infty)$ at which $W$ leaves the MA space; see Figure~\ref{fig:homotopy3}.

For a suitable neighbourhood $U$ of $(x_*,\lambda_*)$, it follows from \eqref{lambdan} that the set $\{(x,\lambda) : \psi_1(x,\lambda) = 0\} \cap U$ is the union of two strictly increasing curves, $\lambda_{m,1}(x)$ and $\lambda_{n,2}(x)$, as shown in the middle of Figure~\ref{fig:mgamma}. In this case we have $i_- = i_+ = 2$, and hence $\Ind(W \circ \gamma_*; P_1) = 0$ by Lemma \ref{lem:mgamma}.
\end{proof}

\subsection{The Turing instability}\label{sec:Turing}
In the previous section, where $D=I$, we saw that $\mathfrak m = 0$ even if the image of $[\delta,L] \times [0,\lambda_\infty]$ left the MA space at a finite set of points. In this section we show that $\mathfrak m$ can be nonzero when the diffusion coefficients $d_1$ and $d_2$ are not equal.
The setting is a two-component reaction--diffusion system (\ref{RD}) with a so-called \emph{Turing instability}. This phenomenon\,---\,first discovered by A.M. Turing \cite{T52}\,---\,refers to a stable, homogeneous equilibrium of a chemical reaction that is counter-intuitively destabilized in the presence of diffusion.

A necessary and sufficient condition for this destabilization to occur is that the ratio $d_2/d_1$ be sufficiently far from $1$. The main result of this section is that the index $\mathfrak m$ is nonzero if and only if this condition is satisfied. The non-vanishing of the hyperplane index therefore gives a topological criterion for the Turing bifurcation. We give a precise statement below, in Theorem \ref{thm:Turing}, after describing the general framework.

Assume that there exists $\bar{u}\in\bbR^2$ such that $F(\bar{u})=0$, and the eigenvalues $\nu_1,\nu_2$ of $\nabla F(\bar{u})$ have negative real part. In other words, $\bar{u}$ is a stable equilibrium of the dynamical system
\begin{equation}\label{kinetics}
    u_t=F(u).
\end{equation}
Setting
\begin{equation}
    \nabla F(\bar u) :=A = \begin{pmatrix}
    a_{11} & a_{12}\\a_{21} & a_{22}
    \end{pmatrix},
\end{equation}
we thus have
\begin{equation}\label{Acondition}
    \det A > 0 , \qquad \tr A < 0.
\end{equation}
By rescaling the independent variable, we may assume that the matrix $D$ takes the form $D = (1, d)$, where $d = d_2/d_1$. We further assume that $\bar{u}$ undergoes a Turing bifurcation, which is to say that $d$ is chosen so that \eqref{Ldef} has positive spectrum, hence $\bar u$ is unstable.
It is well known (see, for instance, \cite[\S 2.3]{Murray}) that a Turing instability exists in this setting if and only if
\begin{equation}\label{Turing condition}
    a_{22} + d a_{11}>2\sqrt{d\det A}.
\end{equation}
It is worth noting that a necessary condition for a Turing instability is that $a_{12}a_{21}<0$, so in particular $F(u)$ cannot be a gradient. Moreover, \eqref{Acondition} and \eqref{Turing condition} together imply that $d \neq 1$, so Theorem \ref{thm:ccMorseMaslov} does not apply.

The condition \eqref{Turing condition} is derived for the stability problem on $\bbR$, whereas our results are only formulated for finite intervals.
Nonetheless, the problem on $\bbR$ can shed light on what is happening on large enough intervals \cite{Sas00}. With this in mind, the statement of our theorem is natural: \emph{the hyperplane index detects the Turing instability as long as we take $L$ sufficiently large.}

\begin{theorem}
\label{thm:Turing}
Assume that $A$ satisfies \eqref{Acondition}. For $d>0$, the condition \eqref{Turing condition} holds if and only if there exists a number $L_0 > 0$ for which the index of $[\delta,L_0] \times [0,\lambda_\infty]$ is nonzero.
\end{theorem}

We have thus given a topological condition that is both necessary and sufficient for the Turing instability to occur: the condition \eqref{Turing condition} is satisfied if and only if the hyperplane index is non-zero for some $L$. 

In fact, if \eqref{Turing condition} is satisfied, then the index of $[\delta,L] \times [0,\lambda_\infty]$ will be nonzero for almost all $L \geq L_0$. Precisely formulating this requires some care, however, since larger values of $L$ make it more likely that the path $W(x,0)$ will leave the MA space for some $0 < x_* \leq L$. We can prevent this by excluding a small set of $d$ values.

\begin{cor}
\label{cor:Turing2}
Assume that $A$ satisfies \eqref{Acondition}. There is a countable set $\Delta_* \subseteq (0,\infty)$ with the following property:
\begin{enumerate}
    \item If $a_{22} + d a_{11} < 2\sqrt{d\det A}$, then the index of $[\delta,L] \times [0,\lambda_\infty]$ is equal to zero for every $L >0$.
    \item If $a_{22} + d a_{11}>2\sqrt{d\det A}$ and $d \notin \Delta_*$, then the index of $[\delta,L] \times [0,\lambda_\infty]$ is defined and nonzero for all $L$ in an open, dense subset of $[L_0,\infty)$.
\end{enumerate}
\end{cor}

We will see below that the index is not defined at the critical value $d_*$ where \eqref{Turing condition} is an equality. This is to be expected, since $\mathfrak m$ is a homotopy invariant but its value changes as $d$ passes through $d_*$. In other words, the images of the boundary of $[\delta,L] \times [0,\lambda_\infty]$ for $d<d_*$ and $d>d_*$ are not homotopic in $\cM$, though they are homotopic in the Grassmannian. We therefore have a topological characterization of the Turing bifurcation: it occurs at the value of $d$ for which the boundary of $[\delta,L] \times [0,\lambda_\infty]$ first leaves the MA space.

To prove the above results we will use Theorem~\ref{thm:cpMorse} as well as the calculations in Section \ref{sec:cc} for systems with constant coefficients. We therefore write the eigenvalue equation in the form \eqref{eq:constant}, i.e.
\begin{equation}\label{LsystemTuring}
\frac{d}{dx} \begin{pmatrix} v \\ w \end{pmatrix} = \begin{pmatrix} 0 & I \\ D^{-1}(\lambda I - A) & 0 \end{pmatrix} \begin{pmatrix} v \\ w \end{pmatrix},
\end{equation}
so that $B(\lambda) = D^{-1}(\lambda I-A)$. To use Proposition~\ref{prop:inv} we need to understand how the eigenvalues of $B(\lambda)$ depends on $\lambda$ and $d$. When we want to make this dependence explicit we will denote these by $\beta_i(\lambda;d)$. Recalling from \eqref{detXY} that the determinant of $X$ (or equivalently the function $\psi_1$) can only vanish when either $\beta_1$ or $\beta_2$ is a negative real number, we prove the following result.

\begin{prop}
\label{prop:betacurve}
Assume that $A$ satisfies \eqref{Acondition}. The eigenvalues $\beta_1(\lambda;d)$ and $\beta_2(\lambda;d)$ of $B(\lambda) = D^{-1}(\lambda I - A)$ have the following properties:
\begin{enumerate}
    \item If $a_{22} + d a_{11} < 2\sqrt{d\det A}$, then $\beta_1(\lambda;d)$ and $\beta_2(\lambda;d)$ are either  non-real complex conjugates or positive real numbers for every $\lambda \geq 0$;
    \item If $a_{22} + d a_{11} = 2\sqrt{d\det A}$, then
    $\beta_1(0;d) = \beta_2(0;d)$ is a negative real number, and $\beta_1(\lambda;d)$ and $\beta_2(\lambda;d)$ are either non-real complex conjugates or positive real numbers for every $\lambda > 0$;
    \item If $a_{22} + d a_{11} > 2\sqrt{d\det A}$, then there exists a number $\lambda_c > 0$ such that $\beta_1(\lambda;d)$ and $\beta_2(\lambda;d)$ are distinct negative real numbers for $0 \leq \lambda < \lambda_c$, are equal and negative for $\lambda = \lambda_c$, and are non-real complex conjugates or positive real numbers for $\lambda > \lambda_c$.
\end{enumerate}

For (iii) we can assume that $\beta_1(\lambda;d) < \beta_2(\lambda;d)$ for all $0 \leq \lambda < \lambda_c$, in which case
\begin{equation}
\label{beta:deriv}
    \beta_1'(\lambda;d) > 0 >  \beta_2'(\lambda;d).
\end{equation}
\end{prop}

We will see below that these correspond to the three cases in Remark \ref{rem:3_cases}. Case (iii) is the most interesting: the eigenvalues $\beta_1$ and $\beta_2$ increases and decrease, respectively, until they collide at $\lambda = \lambda_c$ and move off of the real axis in a complex conjugate pair, as shown in Figure \ref{fig:Blambda}. They later rejoin on the real axis, at which point they are both positive. It is similarly possible to describe the eigenvalues of $B(\lambda)$ for $\lambda<0$, but we omit this as it plays no role in our analysis.

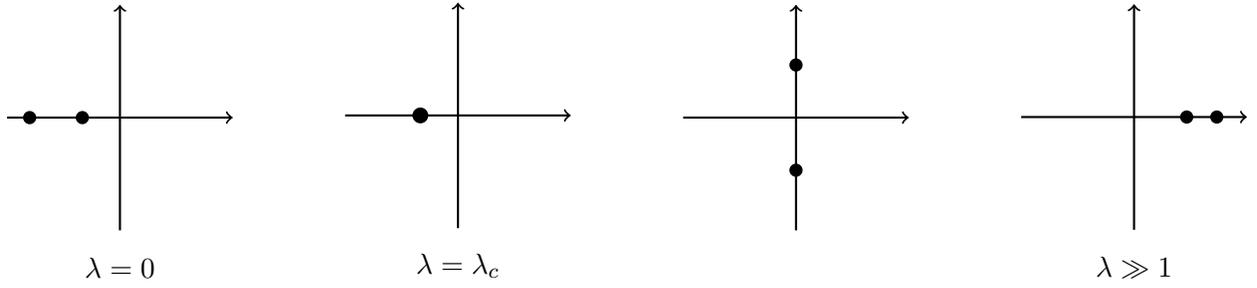
\begin{figure}
\begin{tikzpicture}
	\draw[->,thick] (0,-1.5) to (0,1.5);
	\draw[->,thick] (-1.5,0) to (1.5,0);
	\fill (-1.2,0) circle[radius=2.5pt];
	\fill (-0.5,0) circle[radius=2.5pt];
	\node at (0, -2) {$\lambda = 0$};
\end{tikzpicture}
\hfill
\begin{tikzpicture}
    \draw[->,thick] (0,-1.5) to (0,1.5);
	\draw[->,thick] (-1.5,0) to (1.5,0);
	\fill (-0.5,0) circle[radius=3pt];
	\node at (0, -2) {$\lambda = \lambda_c$};
\end{tikzpicture}
\hfill
\begin{tikzpicture}
    \draw[->,thick] (0,-1.5) to (0,1.5);
	\draw[->,thick] (-1.5,0) to (1.5,0);
	\fill (0,-0.7) circle[radius=2.5pt];
	\fill (0,0.7) circle[radius=2.5pt];
	\node at (0, -2) {$\phantom{\lambda=0}$};
\end{tikzpicture}
\hfill
\begin{tikzpicture}
	\draw[->,thick] (0,-1.5) to (0,1.5);
	\draw[->,thick] (-1.5,0) to (1.5,0);
	\fill (0.7,0) circle[radius=2.5pt];
	\fill (1.1,0) circle[radius=2.5pt];
	\node at (0, -2) {$\lambda \gg 1$};
\end{tikzpicture}
\caption{Case (iii) of Proposition \ref{prop:betacurve}: the eigenvalues of $B(\lambda)$ are negative and distinct when $\lambda < \lambda_c$. At $\lambda = \lambda_c$ they collide and leave the real axis as a complex conjugate pair, crossing the imaginary axis into the right-half plane, later rejoining on positive real axis.}
\label{fig:Blambda}
\end{figure}

\begin{proof}
We first compute
\begin{align}
\begin{split}\label{TuringDetTr}
    \det B(\lambda) & =\frac{1}{d}\big( \lambda^2 -\lambda\tr A +\det A\big)\\
    \tr B(\lambda) & = \frac{1}{d}\big( \lambda(1+d)- (a_{22} + d a_{11}) \big). %\label{TuringTr}.
\end{split}
\end{align}
It follows from \eqref{Acondition} that $\det B(\lambda) > 0$ for all $\lambda \geq 0$. Therefore, the eigenvalues of $B(\lambda)$ are either complex conjugates or real numbers of the same sign. In particular, they are never zero. An elementary calculation shows that the discriminant $\Delta B(\lambda)   = \big(\tr B(\lambda) \big)^2-4\det B(\lambda)$ satisfies
\begin{equation}
\label{Bdisc}
    d^2 \Delta B(\lambda) = (d-1)^2 \lambda^2 + 2(d-1)(a_{22} - d a_{11})\lambda + \big((a_{22} + d a_{11})^2 - 4d \det A\big).
\end{equation}

We divide the proof of (i) into two cases. If $a_{22} + d a_{11} \leq 0$, then $\tr B(\lambda) \geq 0$ for all $\lambda \geq 0$, and the claim follows. For the case $0 < a_{22} + d a_{11} < 2\sqrt{d\det A}$ we must have $d \neq 1$, so the discriminant is quadratic in $\lambda$, with $\Delta B(0) < 0$. Therefore, there exists a number $\lambda_+>0$ such that $\Delta B(\lambda) < 0$ for $0 \leq \lambda < \lambda_+$ and $\Delta B(\lambda) > 0$ for $\lambda > \lambda_+$. It follows that $B(\lambda)$ has complex conjugate eigenvalues for $0 \leq\lambda < \lambda_+$ and real eigenvalues for $\lambda \geq \lambda_+$. For $\lambda \geq \lambda_+$ we have $\tr B(\lambda) \geq 2\sqrt{\det B(\lambda)} > 0$, and so the real eigenvalues are both positive.

We now consider (ii), where $a_{22} + d a_{11} = 2\sqrt{d\det A}$. We again must have $d \neq 1$, so the discriminant is quadratic, with $\Delta B(0) = 0$. We claim that $(d-1)(a_{22} - d a_{11})$ is negative. If $a_{11} > 0$, then \eqref{Acondition} implies $a_{22} < 0$, and the fact that $a_{22} + d a_{11} > 0$ implies $d>1$, hence $a_{22} - d a_{11} < 0$ and $(d-1)(a_{22} - d a_{11})<0$. The proof when $a_{11} < 0$ is similar. Therefore $\Delta B(\lambda)$ has a zero and is decreasing at $\lambda=0$, so there exists a number $\lambda_+>0$ such that $\Delta B(\lambda) < 0$ for $0 < \lambda < \lambda_+$ and $\Delta B(\lambda) > 0$ for $\lambda > \lambda_+$. The conclusion follows as in the proof of (i).

For (iii) we similarly find that $d \neq 1$ and $\Delta B(0) > 0$, with $\Delta B(\lambda)$ decreasing at $\lambda=0$. Therefore, there exists numbers $0 < \lambda_- < \lambda_+$ such that $\Delta B(\lambda)$ is positive on $(0,\lambda-) \cup (\lambda_+,\infty)$ and negative on $(\lambda_-,\lambda_+)$. Setting $\lambda_c = \lambda_-$ completes the proof.

For the last part of the proof, we abbreviate $\beta_i = \beta_i(\lambda;d)$ and use \eqref{TuringDetTr} to compute
\[
    \beta_1' + \beta_2' = (\tr B)' > 0,
\]
which implies at least one of $\beta_1'$ and $\beta_2'$ is positive. If the other was non-negative we would have $\beta_1'\beta_2 + \beta_1\beta_2' < 0$, since $\beta_1$ and $\beta_2$ are both negative for $0 \leq \lambda < \lambda_c$. However, this contradicts the fact that
\[
    \beta_1'\beta_2 + \beta_1\beta_2' = (\det B)'(\lambda) > 0
\]
and thus completes the proof.
\end{proof}

Using this result, we easily obtain the following.

\begin{cor}
\label{cor:TuringCase1}
If $A$ satisfies \eqref{Acondition} and
\begin{equation}
    a_{22} + d a_{11} < 2\sqrt{d\det A},
\end{equation}
then $W(x,\lambda) \notin H_1$ for all $x>0$ and $\lambda \geq 0$, hence the boundary of $[\delta,L] \times [0,\lambda_\infty]$ has index $\mathfrak m = 0$ for every $L>0$.
\end{cor}

This corresponds to case (i) of Remark \ref{rem:3_cases}\,---\,the frame matrix $X(x,\lambda)$ is always invertible, so the index is well defined but necessarily zero.

\begin{proof}
Proposition \ref{prop:betacurve}(i) implies $\beta_1(\lambda;d)$ and $\beta_2(\lambda;d)$ are either non-real or positive for any $\lambda \geq 0$. It follows from \eqref{detXY} that $\det X(x,\lambda) \neq 0$, and hence $W(x,\lambda) \notin H_1$, for all $x > 0$ and $\lambda \geq 0$.
\end{proof}

This means if the index of $[\delta,L_0] \times [0,\lambda_\infty]$ is defined and nonzero for some $L_0 > 0$, then $ a_{22} + d a_{11} \geq 2\sqrt{d\det A}$. We next show that this is in fact a strict inequality; in the critical case of equality the index is either equal to zero or is not defined.

\begin{lemma}
If $d = d_*$ is chosen so that \eqref{Turing condition} is an equality, then there exists $L_*>0$ such that the index of $[\delta,L] \times [0,\lambda_\infty]$ is zero for $L < L_*$ and is not defined for $L \geq L_*$.
\end{lemma}

This corresponds to case (ii) in Remark \ref{rem:3_cases}.

\begin{proof}
From Proposition~\ref{prop:betacurve}(ii) we know that $\beta_1(\lambda;d_*)$ and $\beta_2(\lambda;d_*)$ are either positive or non-real for all $\lambda>0$, and so $\det X(x,\lambda) \neq 0$ for all $x>0$ and $\lambda>0$. On the other hand, for $\lambda=0$ we have that $\beta_1(0;d_*) = \beta_2(0;d_*)$ is negative, and so \eqref{detXY} implies that the following equivalence:
\begin{align*}
    W(x,0) \in H_1 \quad&\Longleftrightarrow\quad 
    \det X(x,0) = 0 \\
    \quad&\Longleftrightarrow\quad   
    x = \frac{m\pi}{\sqrt{-\beta_1(0;d_*)}} \
    \text{ for some } m \in \bbN \\
    \quad&\Longleftrightarrow\quad
    W(x,0) \in H_1 \cap H_2.
\end{align*}
The index is therefore defined and equal to zero for $L < L_* := \pi / \sqrt{-\beta_1(0;d_*)}$ and is not defined for $L \geq L_*$.
\end{proof}

Combined with Corollary \ref{cor:TuringCase1}, this lemma proves one direction of Theorem \ref{thm:Turing}: if the index of $[\delta,L_0] \times [0,\lambda_\infty]$ is defined and nonzero for some $L_0 > 0$, then $a_{22} + d a_{11} > 2\sqrt{d\det A}$. To complete the proof, we must show the reverse implication.

The idea of the proof is shown in Figure \ref{fig:manycrossings}. Parameterizing the set $\{(x,\lambda) : \det X(x,\lambda) = 0\}$ as a union of smooth curves $\{\lambda_n(x)\}$, we will show that $W(x_*,\lambda_*) \in H_1 \cap H_2$ if and only if $\lambda_* = \lambda_n(x_*)$ is either a maximum of a single eigenvalue curve or a transverse intersection of two distinct curves. Moreover, each maximum contributes $+2$ to the index $\mathfrak m$, while the intersections contribute nothing. Therefore, it suffices to choose $L_0$ larger than the $x$ coordinate of the first maximum but smaller than the $x$ coordinate of the first intersection. This will guarantee that the index is well defined but not equal to zero.

\begin{lemma}
\label{lem:betaindex}
Suppose that $\beta_1(\lambda_*;d)$ and $\beta_2(\lambda_*;d)$ are both negative and satisfy
\begin{equation}
\label{betamn}
    \frac{\beta_1(\lambda_*;d)}{\beta_2(\lambda_*;d)} = \left( \frac mn \right)^2
\end{equation}
for some $m,n \in \bbN$, so that $W(x_*,\lambda_*) \in H_1 \cap H_2$ at the point
\begin{equation}
\label{x*leave}
    x_* = \frac{m\pi}{\sqrt{-\beta_1(\lambda_*;d)}} = 
    \frac{n\pi}{\sqrt{-\beta_2(\lambda_*;d)}}.
\end{equation}
Letting $\gamma_*$ denote a small loop around $(x_*,\lambda_*)$, we have
\begin{equation}
    \Ind(W\circ\gamma_*;P_1) = \begin{cases} 2 & \text{if }\ m=n, \\ 0 & \text{otherwise.} \end{cases}
\end{equation}
\end{lemma}

Geometrically, points with $m=n$ correspond to maxima of the eigenvalue curves, whereas points with $m \neq n$ correspond to transverse intersections of different curves; see Figures \ref{fig:mgamma} and \ref{fig:manycrossings}.

\begin{figure}
\begin{center}
    \includegraphics[scale=0.295]{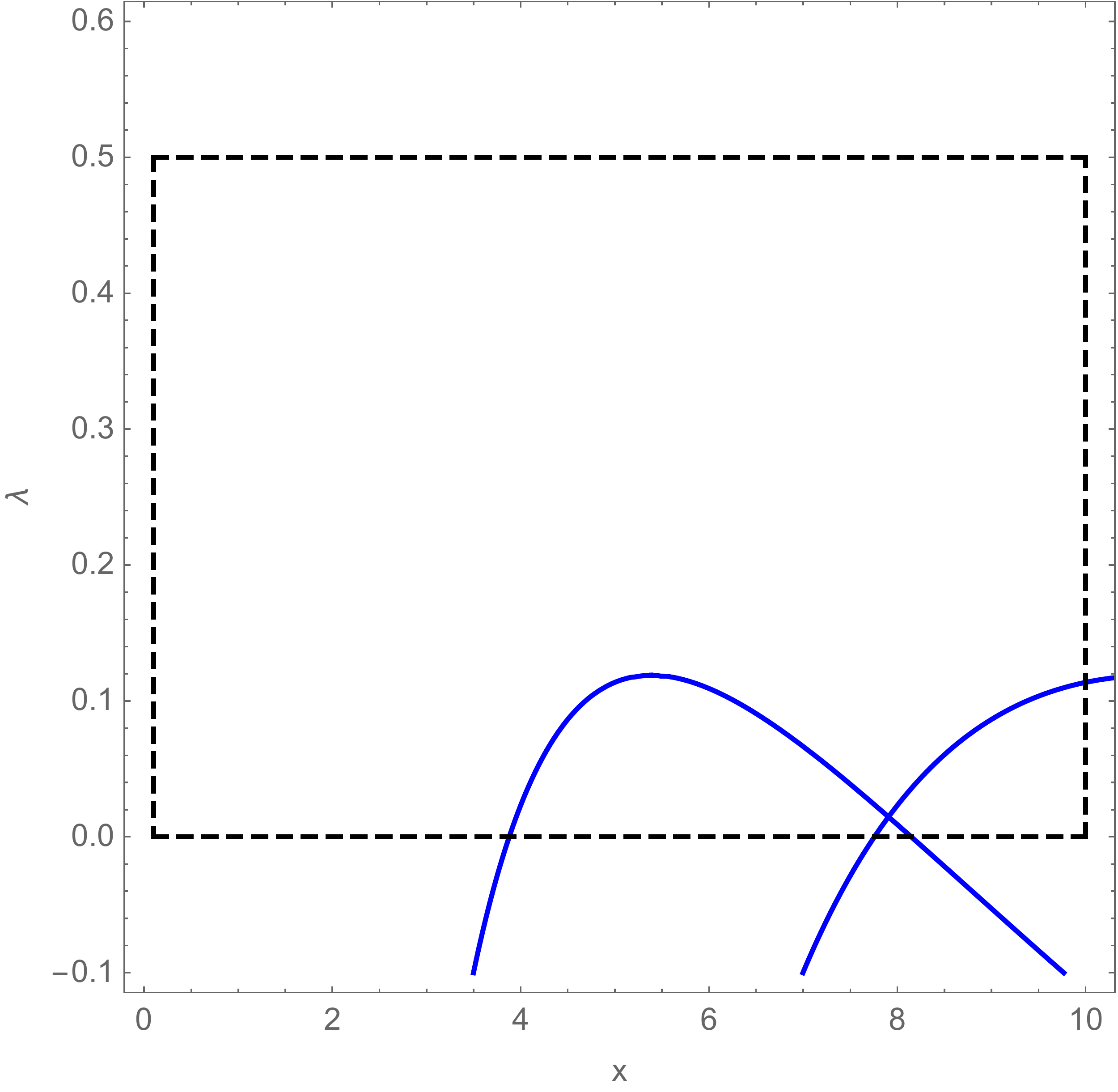} 
    \hfill
    \includegraphics[scale=0.3]{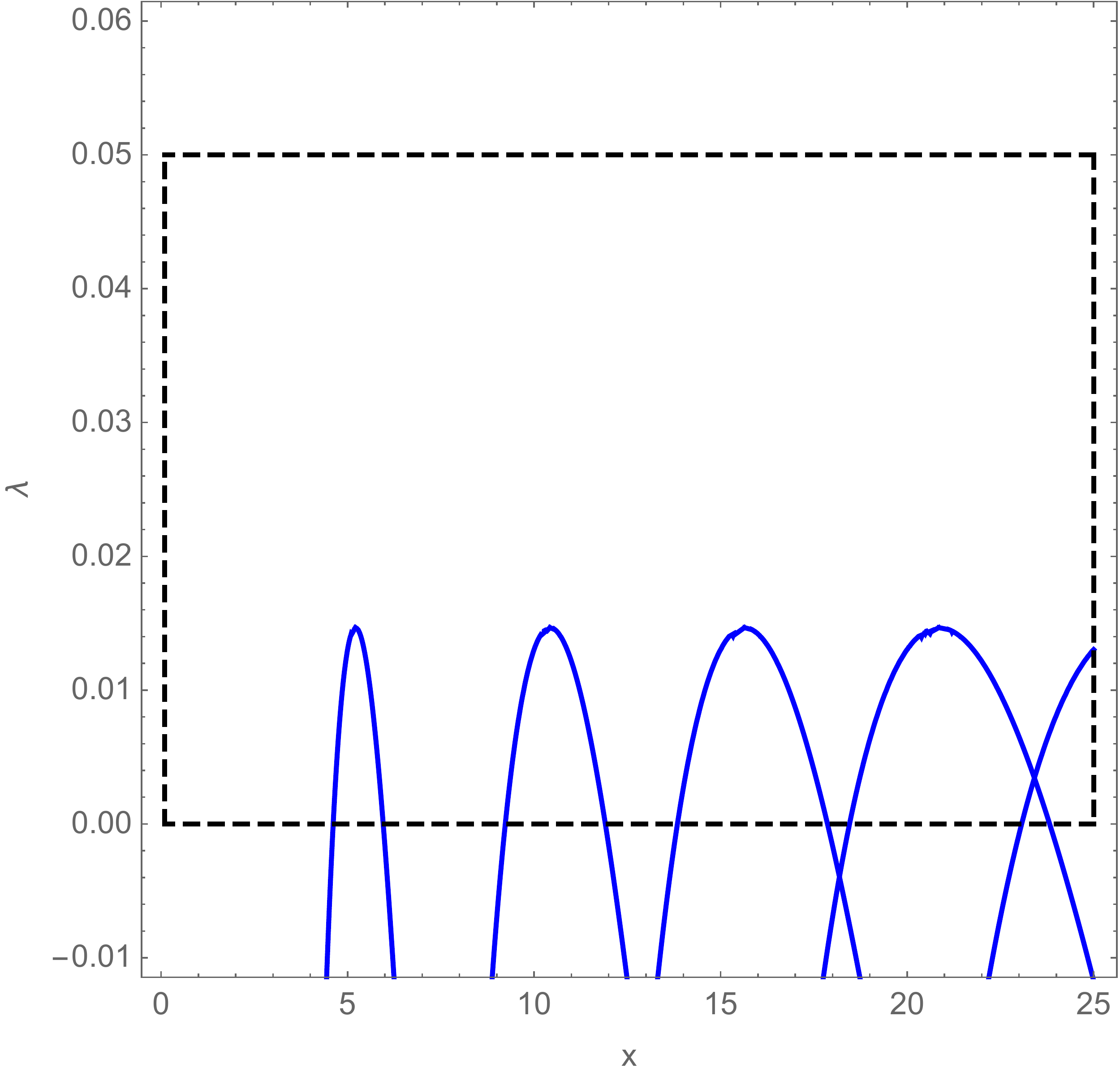}
\end{center}
\caption{A plot of the eigenvalue curves $\lambda_n(x)$ in Remark \ref{rem:curves}. The left figure shows a maximum of $\lambda_1$ (meaning \eqref{betamn} is satisfied with $m=n=1$) and an intersection between $\lambda_1$ and $\lambda_2$, at which point \eqref{betamn} is satisfied with $m=2$ and $n=1$. Similarly, the right figures shows maxima of the first four eigenvalue curves, as well as intersections with $m/n = 4/3$ (where $\lambda_*<0$) and $m/n = 5/4$ ($\lambda_* >0)$.}
\label{fig:manycrossings}
\end{figure}

\begin{proof}
Using \eqref{detXY} and Proposition \ref{prop:betacurve}, we can write the set $\{\det X = 0\}$ as a countable union of smooth curves
\begin{equation}
   \bigcup_{n=1}^\infty 
   \underbrace{\left\{ \left(\frac{n\pi}{\sqrt{-\beta_1(\lambda;d)}},\lambda \right) : \lambda \leq \lambda_c \right\}}_{= \Gamma_{1,n}} \cup \underbrace{ \left\{ 
   \left(\frac{n\pi}{\sqrt{-\beta_2(\lambda;d)}},\lambda \right) : \lambda \leq \lambda_c  \right\}}_{ = \Gamma_{2,n}}.
\end{equation}
The condition \eqref{betamn} holds precisely when $\Gamma_{1,m}$ and $\Gamma_{2,n}$ intersect at the point $(x_*,\lambda_*)$. There are two cases to consider.

If $m=n$, then \eqref{beta:deriv} implies that the curve $\Gamma_{1,n} \cup \Gamma_{2,n}$ has a strict maximum at $(x_*,\lambda_*)$, so it follows from Lemma \ref{lem:mgamma} that $\Ind(W\circ\gamma_*;P_1) = 2$, as in Figure \ref{fig:mgamma}(right). On the other hand, if $m \neq n$, then \eqref{x*leave} implies $\beta_1(\lambda_*;d) \neq \beta_2(\lambda_*;d)$, and hence $\lambda_* < \lambda_c$ by Proposition \ref{prop:betacurve}(iii). If follows that the curves $\Gamma_{1,m}$ and $\Gamma_{1,n}$ are strictly increasing and strictly decreasing, respectively, so Lemma \ref{lem:mgamma} implies that $\Ind(W\circ\gamma_*;P_1) = 0$, as in Figure \ref{fig:mgamma} (left).
\end{proof}

\begin{rem}
\label{rem:curves}
For each $n$ the set $\Gamma_{1,n} \cup \Gamma_{2,n}$ is given by the graph of a single function $\lambda_n(x)$, as in the hypotheses of Lemmas \ref{lem:leave} and \ref{lem:mgamma}. In terms of this parameterization, the condition \eqref{betamn} with $m=n$ corresponds to a maximum of $\lambda_n$, and $m>n$ corresponds to a transverse intersection of $\lambda_m$ and $\lambda_n$.
\end{rem}

We are now ready to prove the main result in this section.

\begin{proof}[Proof of Theorem \ref{thm:Turing}]
From Proposition \ref{prop:betacurve}(iii) and the definition of $\lambda_\infty$ we know that $0 < \lambda_c < \lambda_\infty$, so any maxima of eigenvalue curves that occur will be contained in the rectangle $[\delta,L] \times [0,\lambda_\infty]$. The first maximum will occur at
\[
    x_{\rm max} = \frac{\pi}{\sqrt{-\beta_1(\lambda_c;d)}} = 
    \frac{\pi}{\sqrt{-\beta_2(\lambda_c;d)}}.
\]
whereas the first intersection will occur at
\[
    x_{\rm int} = \frac{2\pi}{\sqrt{-\beta_1(\lambda_*;d)}} = 
    \frac{\pi}{\sqrt{-\beta_2(\lambda_*;d)}}.
\]
Since $\beta_2(\lambda;d)$ is strictly decreasing for $0 \leq \lambda < \lambda_c$ and $\lambda_* < \lambda_c$, we have
\[
    \frac{\pi}{\sqrt{-\beta_2(\lambda_c;d)}}
    < \frac{\pi}{\sqrt{-\beta_2(\lambda_*;d)}}
\]
and hence $x_{\rm max} < x_{\rm int}$. Choosing $L_0 \in (x_{\rm max}, x_{\rm int})$ completes the proof.
\end{proof}

With the tools developed above, we can easily prove Corollary \ref{cor:Turing2}.

\begin{proof}[Proof of Corollary \ref{cor:Turing2}]
Part (i) was already shown in Corollary \ref{cor:TuringCase1}. To prove (ii) we first check the invariance conditions \eqref{xlinv}. To ensure $W(x,0) \in \cM$ for all $x>0$, it suffices, by Proposition \ref{prop:inv}, to know that the ratio
\[
    \frac{\beta_{1}(0;d)}{\beta_{2}(0;d)} = \frac{-a_{22}-d a_{11} - \sqrt{4d a_{21}a_{21} + (a_{22}-d a_{11})^{2}}}{-a_{22}-d a_{11} + \sqrt{4d a_{21}a_{21} + (a_{22}-d a_{11})^{2}}}
\]
is not the square of a rational number. Asking that this be equal to $q^{2}$ for some $q \in \mathbb{Q}$ gives 
\[
    d(q^{2}) = \frac{(1+q^4) \det A  - 2a_{21}a_{12}q^{2} + \sqrt{\det (A)(1+q^{2})^{2} \big((1+q^4) \det A + 2(a_{11}a_{22}+a_{21}a_{12})q^{2} \big)}
}{2a_{11}q^{2}},
\]
so we define the set $\Delta_* = \{d(q^2) : q \in \bbQ\}$.

On the other hand, we have that $W(L,\lambda_*) \in H_1 \cap H_2$ for some $\lambda_* \geq 0$ if and only if \eqref{x*leave} is satisfied for $x_*=L$. The fact that $\beta_1$ and $\beta_2$ are strictly monotone implies that this equality can only hold for a discrete set of $L$ values. Therefore, assuming $d \notin \Delta_*$, the index of $[\delta,L] \times [0,\lambda_\infty]$ is defined for an open, dense set of $L > \delta$. For any such $L$, Lemma \ref{lem:betaindex} says that the index $\mathfrak m$ is twice the number of local maxima of the eigenvalue curves (i.e. intersections for which $m=n$) between $x=\delta$ and $x=L$, and hence is nonzero for $L \geq L_0$.
\end{proof}

For an explicit example, consider the matrix
\begin{align}\label{eq:Amat}
    A:= \begin{pmatrix} 1 & -2 \\ 3 & -4 \end{pmatrix},
\end{align} 
for which we compute
$$
    d_{*} =4\big(\sqrt{2} + 3\big) \approx 14.92. 
$$
The matrix $B(\lambda)$ is thus 
$$ 
B(\lambda) = \begin{pmatrix} -1 + \lambda & 2 \\ -\frac{3}{d} & \frac{4+ \lambda}{d} \end{pmatrix},
$$ 
and for $d = \frac{31}{2} > d_{*}$ we find that
\begin{equation}
\label{betaexample}
\frac{\beta_{1}\big(0;\frac{31}{2}\big)}{\beta_{2}\big(0;\frac{31}{2}\big)} = \frac{281 + 23\sqrt{33}}{248}.
\end{equation}
This is clearly not the square of a rational number, and so the path $W(x,0)$ remains in the MA space for all $x>0$. That is, $d = \frac{31}2$ is not in the set $\Delta_*$ from Corollary \ref{cor:Turing2}. However, it is unavoidable that the path $W(L,\lambda)$ will leave the MA space for a discrete set of $L$ values. Some examples are shown in Figures~\ref{fig:manycrossings} and~\ref{fig:littlebigbox}.

\begin{figure}
\begin{center}
\includegraphics[scale=0.3]{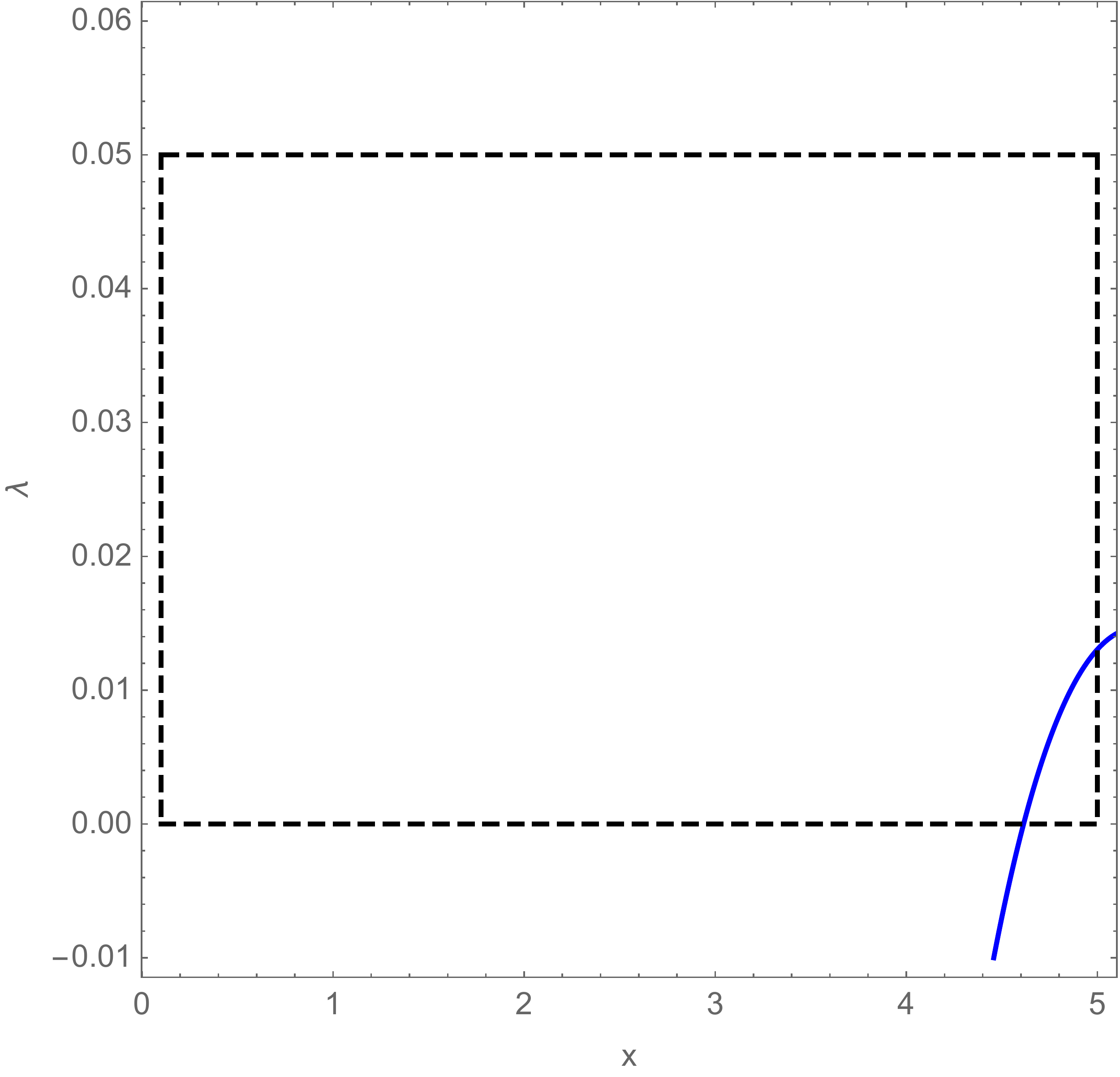}
\hfill
\includegraphics[scale=0.3]{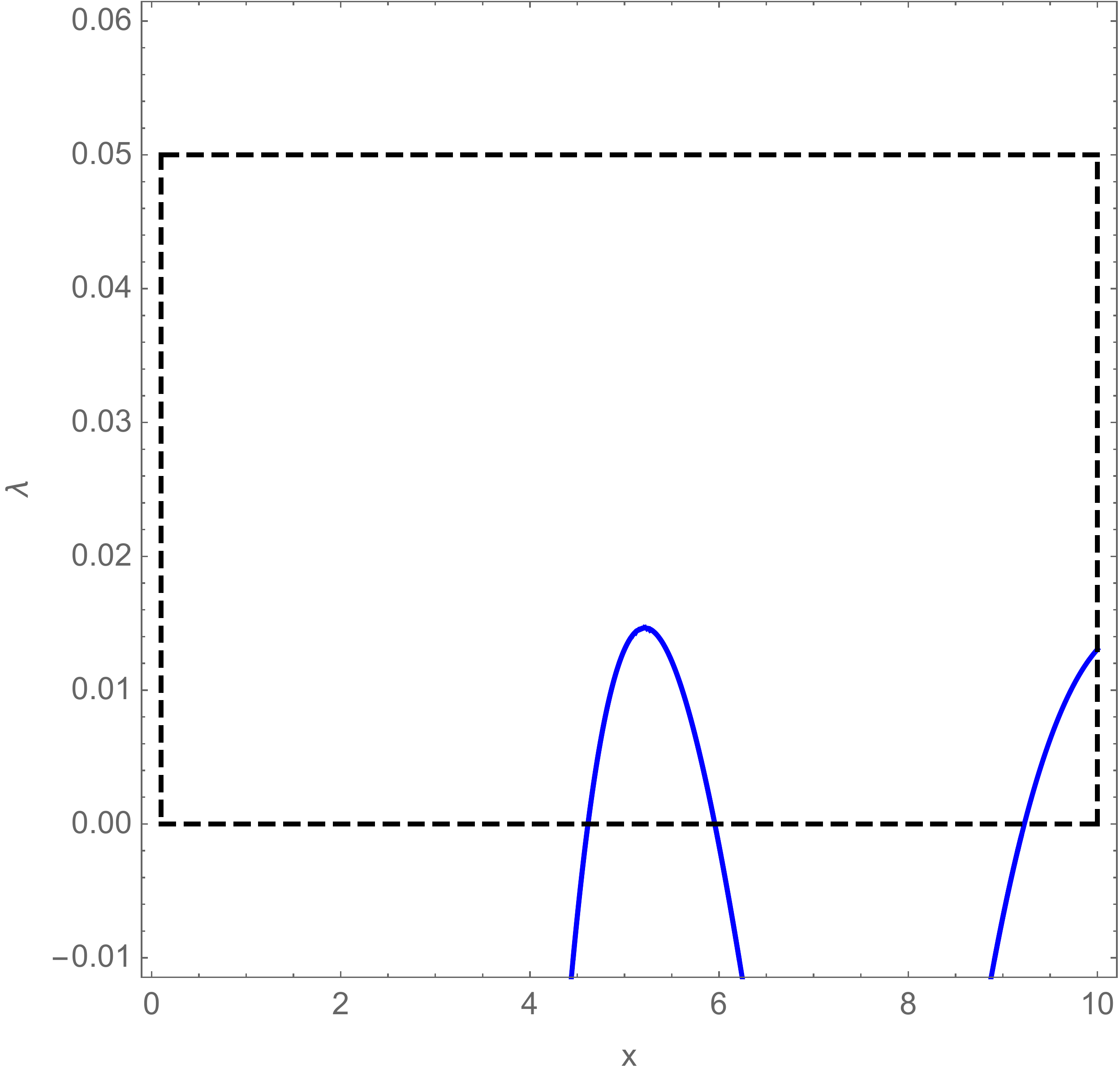}
\end{center}
\caption{A plot of $\{\det X = 0\}$ (solid curve - blue online) in the region $[0,L]\times [0,0.05]$, as well as the Maslov box (dashed), with $L=5$ on the left and $L=10$ on the right. Here $d = 15.5$, and $A$ is as in \eqref{eq:Amat}. Intersections of the solid curve with the line $\lambda =0$, are conjugate points, while intersections with the line $x = L$ represent solutions to the eigenvalue problem \eqref{LsystemTuring} with Dirichlet boundary conditions at $x=0$ and $x=L$. For $L=5$ the entire rectangle is mapped into the MA space, but for $L=10$ it leaves at the point $(x_*,\lambda_*) \approx (5.4,0.015)$, and the boundary has index $\mathfrak m = 2$ as a result.}
\label{fig:littlebigbox}
\end{figure}

\subsection{Numerical prospects}
The classical Maslov index has seen many successful numerical treatments; see for instance \cite{BM14,CDB1,CDB2}. We expect that the theory developed in this paper will be equally amenable to numerical applications, if not more so. 

To justify this, we recall from Theorem \ref{thm:cpMorse}(ii) that
\begin{equation}
	\# \big\{\text{nonnegative eigenvalues of } \cL \big\} 
	\geq \Ind\left(W(x,0)\big|_{x \in [\delta,L]} ; P_1 \right)
\end{equation}
as long as $W(x,\lambda) \in \cM = G \setminus (H_1 \cap H_2)$ for all $(x,\lambda) \in (0,L] \times [0,\lambda_\infty]$, where $H_1$ is the hyperplane corresponding to the train of the Dirichlet subspace and $H_2$ is arbitrary.

The particular choice of $H_2$ in the third part of Theorem \ref{thm:cpMorse} guaranteed monotonicity in $x$, but this is not important if the index is to be computed numerically\,---\,for any choice of $H_2$ the Maslov index computation simply becomes a winding number calculation in $\bbR P^1$. This is numerically robust, due to the homotopy invariance of the index. For instance, the curves
\begin{align*}
    \eta(t) = \begin{cases}
    e^{it}, & -\pi/2 \leq t \leq 0 \\
    e^{-it}, & 0 \leq t \leq \pi/2 \end{cases}, \qquad 
    \eta_+(t) = e^{i\epsilon} \eta(t), \qquad
    \eta_-(t) = e^{-i\epsilon} \eta(t)
\end{align*}
are $\epsilon$-close, pass through the point $1 \in S^1$ one, two and zero times, respectively, and all have zero winding number. That is, the signed count of conjugate points (i.e. the generalized Maslov index) is stable under small perturbations, while the unsigned count is not.

Therefore, a small approximation error in the calculation of the path $W(x,\lambda)$ (i.e. in the numerical solution of an initial-value problem) will not change the numerically computed winding number. The only possible complication is the presence of a conjugate point near $x=L$. If there is a conjugate point near (but not exactly at) the endpoint, it will be possible to determine so with sufficiently accurate numerics. Indeed, this can be established rigorously using validated numerics; see \cite{BBL18} for an overview of rigorous numerical methods applied to dynamical systems.

The case of a conjugate point at $x=L$ is more subtle, since it cannot be distinguished from a conjugate point that is very close (within some numerical tolerance) to $x=L$. Generically the endpoint is not a conjugate point, and when it is, this is usually a consequence of an underlying symmetry of the system. If we know a priori that $x=L$ \emph{is} a conjugate point, and understand the mechanism that causes this to happen, then we can (rigorously) find a neighbourhood around it containing no other conjugate points, and hence the discussion in the previous paragraph applies.

We finally note an equivalent invariance condition that may be easier to verify in practice.
While the condition that the image of $W$ remains in $\cM$ depends on both $\psi_1$ and $\psi_2$, which are proportional to $\det X$ and $\det(X+Y) - \det X - \det Y$, respectively, it is possible to describe the invariance just in terms of $X$. We recall that $W(x_*,\lambda_*) \in H_1$ if and only if $\psi_1(x_*,\lambda_*) = 0$.
% \begin{equation}
%     W(x_*,\lambda_*) \in H_1 \quad\Longleftrightarrow\quad
%     \psi_1(x_*,\lambda_*) = 0 
%     \quad\Longleftrightarrow\quad
%     \det X(x_*,\lambda_*) = 0.
% \end{equation}
At such a point $(x_*,\lambda_*)$, \eqref{psi1deriv} implies $\p \psi_1/\p x = \psi_2$, so we conclude that
\begin{equation}
    W(x_*,\lambda_*) \in H_1 \cap H_2 
    \quad\Longleftrightarrow\quad
    \psi_1(x_*,\lambda_*) = 
    \frac{\p \psi_1}{\p x}(x_*,\lambda_*) = 0.
\end{equation}
Since $\psi_1$ is proportional to $\det X$, this is equivalent to
\begin{equation}
\label{doubleroot}
    W(x_*,\lambda_*) \in H_1 \cap H_2 
    \quad\Longleftrightarrow\quad
    \det X(x_*,\lambda_*) = 
    \frac{\p \det X}{\p x}(x_*,\lambda_*) = 0.
\end{equation}
That is, $W$ leaves the MA space at the point $(x_*,\lambda_*)$ if and only if $\det X$ vanishes at least to first order in $x$. A numerical example, corresponding to the Turing system in \eqref{LsystemTuring}, is shown in Figure~\ref{fig:masx}. In this example, it can be seen by inspection that no double roots occur.

\begin{figure}
\begin{center}
\includegraphics[scale=0.3]{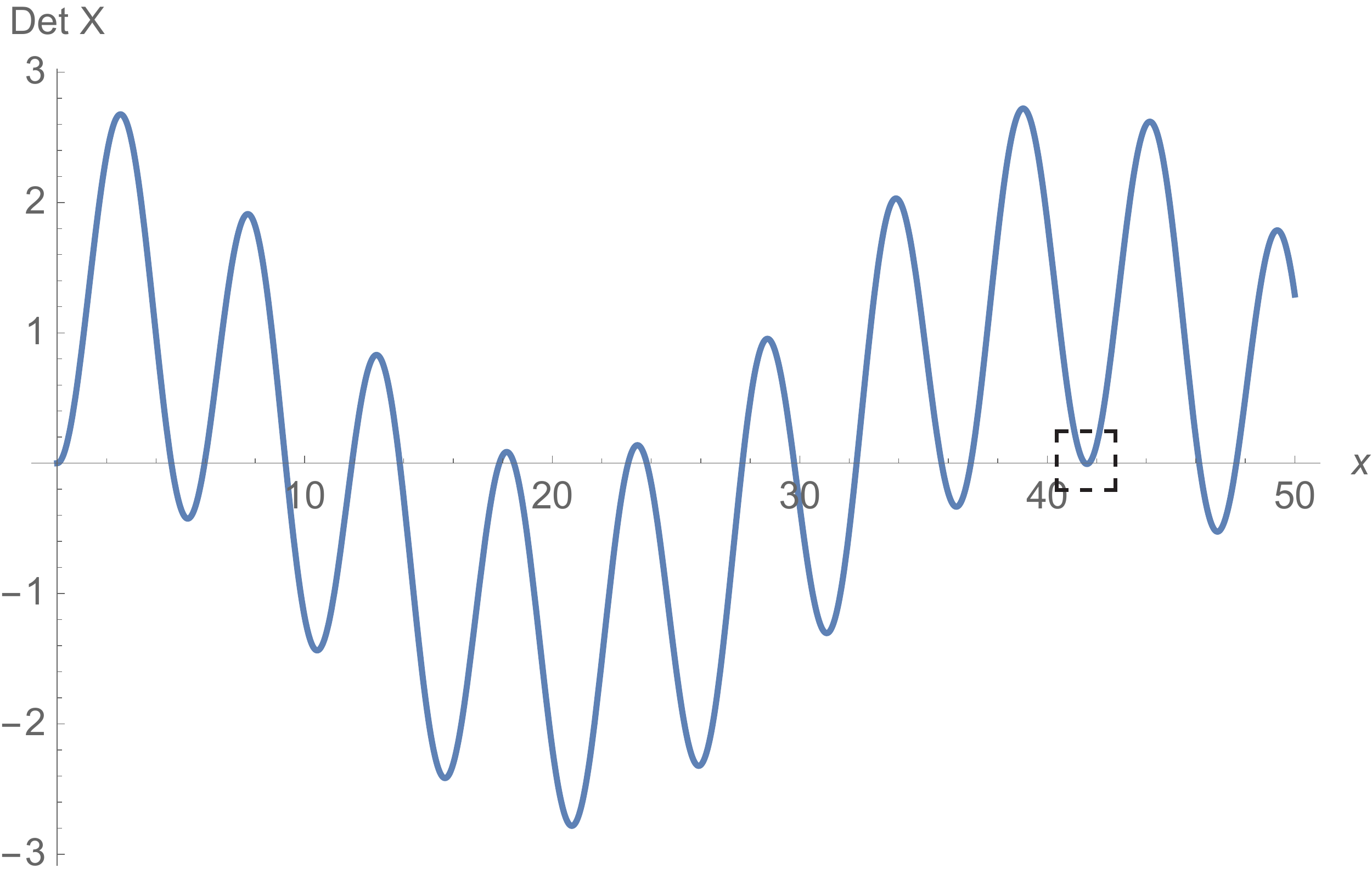}
\hfill
\includegraphics[scale=0.3]{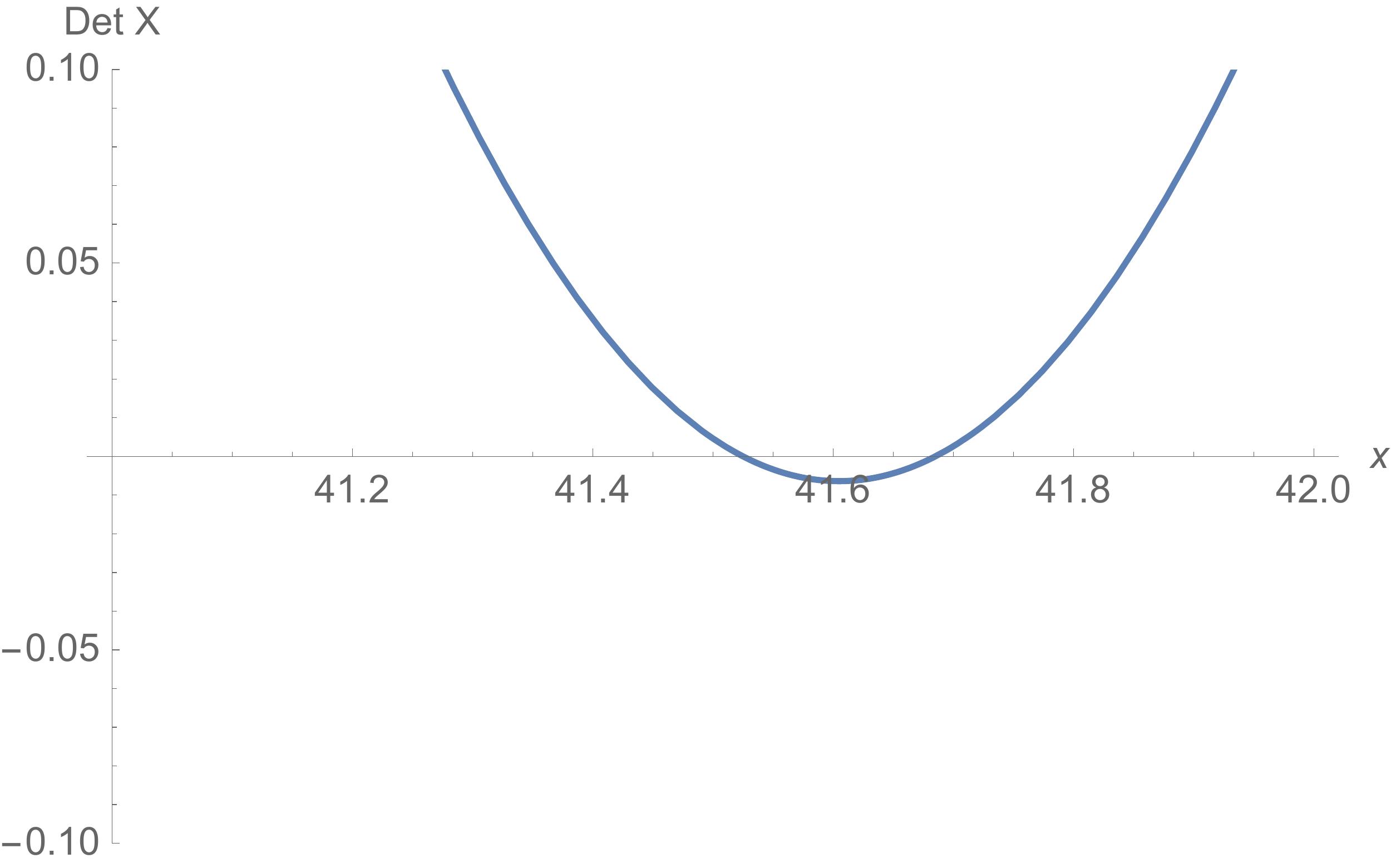}
\end{center}
\caption{A plot of $\det X(x,0)$ vs $x$ for the system \eqref{LsystemTuring} with $A$ from \eqref{eq:Amat} and $d = 15.5$. The calculation \eqref{doubleroot} says that $W(x_*,0) \in H_1 \cap H_2$ if and only if the determinant vanishes to (at least) first order at $x_*$, which never happens for this choice of $A$ and $d$, according to
\eqref{betaexample}. The right figure is a magnification of the dashed box in the left figure, clarifying that there is not a double root near $x \approx 41.6$.
}
\label{fig:masx}
\end{figure}

\section{Further examples (and non-examples) of Maslov--Arnold spaces}\label{furex}

In Section \ref{sec:eLG} we focused on the MA spaces of hyperplane type, as those have proven most useful in applications so far. We now return to the general concept of a Maslov--Arnold space, as given in Definition \ref{def:MA1}, and construct additional examples of MA spaces. We also describe some spaces that do not satisfy the definition. This sheds additional light onto the general definition, and motivates our use of the hyperplane Maslov--Arnold spaces, which do not contain all of $\Lambda(n)$. We begin with a definition.

\begin{define} 
Given a pair of equal rank Maslov--Arnold spaces, we say $(\cM_2, P_2, \alpha_2)$ extends $(\cM_1, P_1, \alpha_1)$ if $\cM_1 \subseteq \cM_2$,  $P_1 = P_2$, and $i^*(\alpha_2) =\alpha_1$, where $i \colon\cM_1 \hookrightarrow \cM_2$ is subspace inclusion. The extension is said to be proper if $\cM_1 \neq \cM_2$.
\end{define}

It is natural to look for extensions of the classical Maslov--Arnold space $(\Lambda(n),P,\alpha_0)$. A proper extension does exist when $n=2$.

\begin{theorem}\label{thm:fatLG}
There exists a rank two Maslov--Arnold space $(\cF,P,\alpha)$, with $\cF$ dense in $\Gr_2(\R^4)$, that extends the classical Maslov--Arnold space $(\Lambda(2), P, \alpha_0)$.
\end{theorem}

A generalized Maslov index is therefore defined for each loop in $\cF$; for a sufficiently generic loop it is given by the geometric intersection number with the train of $P$, and for a loop contained entirely in $\Lambda(2)$ it coincides with the classical Maslov index. This index is much more broadly defined than the classical Maslov index, since $\cF$ is dense in $\Gr_2(\bbR^{4})$, whereas $\Lambda(2)$ is a hypersurface.

However, the space $\cF$ given by Theorem \ref{thm:fatLG} is not a submanifold of the Grassmannian. It will be seen in the proof (which we give in Section \ref{app:fLG}) that it does not contain an open neighbourhood of $\Lambda(2)$. This makes it difficult to use in practice\,---\,although $\cF$ is left invariant by the flow of any Hamiltonian system with Lagrangian initial data
(because $\Lambda(2)$ is), an arbitrarily small perturbation of the system may cause its trajectories to leave $\cF$, in which case the index is no longer defined.

It turns out this undesirable behaviour is inevitable for extensions of the classical Maslov--Arnold spaces.

\begin{theorem}\label{thmcont}
There does not exist a proper extension $(\cM, P, \alpha)$ of the classical Maslov--Arnold space $(\Lambda(n),P,\alpha_0)$ for which $\cM$ is a connected, smoothly embedded submanifold of $Gr_n( \R^{2n})$.
\end{theorem}

In other words, the only smooth, connected MA space that extends $\Lambda(n)$ is $\Lambda(n)$ itself. Compare Remark \ref{fourcases}(iv), which is smooth and contains a Lagrangian Grassmanian, but is not an MA space.

The remainder of this section is devoted to the proof of these two theorems.

\subsection{The Fat Lagrangian Grassmannian}\label{app:fLG}
In this section we prove Theorem \ref{thm:fatLG} by constructing a rank two Maslov--Arnold space $(\cF,P)$ that extends the classical Maslov--Arnold space $(\Lambda(2), P)$ for any $P \in \Lambda(2)$.

As described above, $\cF$ has the desirable property of being a large MA space that contains the entire Lagrangian Grassmannian, and the undesirable property of not being a smooth manifold. The lack of smoothness follows directly from the construction given below, but also from Theorem \ref{thmcont}, which demonstrates that this problem is essential, and does not depend on the
particular details of our construction.

Let $v_1,v_2,v_3, v_4 \in V \cong \R^4$ be a basis, with dual basis $v^*_1,v^*_2,v^*_3,v^*_4 \in V^*$. Define symplectic forms 
\begin{align*}
	\omega_I &:= v^*_1 \wedge v^*_3 + v^*_2 \wedge v^*_4 & \omega_J & := v^*_1\wedge v^*_4 - v^*_2\wedge v^*_3,
\end{align*}  
with corresponding Lagrangian Grassmannians  
\begin{align*}
\Lambda_I &:= G \cap H_{\omega_I} & \Lambda_J &:= G \cap H_{\omega_J},
\end{align*}
and observe that both $Q := [v_1\wedge v_2]$ and $P:= [v_3 \wedge v_4]$ lie in the intersection $\Lambda_I \cap \Lambda_J$.

Denote Pl\"ucker coordinates by $p_{ij} = v^*_i \wedge v^*_j$, regarded as linear functions $p_{ij}\colon \bigwedge^2(V) \rightarrow \R$. The image of the Pl\"ucker embedding, $G \subseteq P(\bigwedge^2(V))$, is defined by the homogeneous quadratic equation $$p_{12}p_{34} - p_{13}p_{24} + p_{14} p_{23} = 0.$$

Consider the closed subset $S \subseteq P(\bigwedge^2(V))$ defined by the linear equation $ p_{14} -p_{23}= 0$ and the inequality $p_{12} p_{34} \geq 0$. The inequality makes sense in $P(\bigwedge^2(V))$ because given $\xi \in \Lambda^2(V)$ and $c \in \R$, we have $p_{12} p_{34}(c \xi) = c^2 p_{12} p_{34}(\xi)$, so the sign of $p_{12} p_{34}$ is well-defined.

\begin{lemma}\label{intpq}
The intersection $ \Lambda_I \cap S$ consists of the two points $P, Q \in G$.
\end{lemma}

\begin{proof}
The intersection is determined by the system of inequalities
\begin{eqnarray*}
p_{12}p_{34} - p_{13}p_{24} +p_{14} p_{23} &=& 0\\
p_{13} &=& -p_{24} \\
p_{14} &=&p_{23} \\
p_{12}p_{34} &\geq&0
\end{eqnarray*}
where the first two equations determine $\Lambda_I$ and the second two inequalities determine $S$. Substituting the first three equalities into the inequality yields $ p_{13}^2 + p_{24}^2 \leq 0$, which is only possible if $p_{13} = p_{24} =0$. We are thus reduced to the equivalent equations
\begin{equation*}
	p_{12}p_{34} = p_{13} = p_{24} = p_{14} = p_{23}  =0,
\end{equation*}
which have only two solutions:  $Q = [ v_1\wedge v_2]$  and $P= [v_3\wedge v_4]$.
\end{proof}

Define $\cU:= G \setminus S$.  This is an open, dense subset of $G$, hence it is a non-compact, orientable 4-manifold, so by Poincar\'e duality $H^1(\cU;\Z)$ is naturally isomorphic to the relative homology group $H_3(G,S;\Z)$ (alternatively, the Borel--Moore homology group $H_3^{BM}(\cU;\Z)$). The train of $P$ in $\cU$ is the intersection $\cU \cap H_{v^*_3 \wedge v^*_4}$.

\begin{lemma}\label{lem cU}
The train of $\cZ_{_P} \cap \cU$ in $\cU$ is a smooth, closed, co-orientable submanifold of $\cU$. 
\end{lemma}

\begin{proof}

The intersection $G \cap H_{v^*_3 \wedge v^*_4}$ is transverse except at $P = [ v_3 \wedge v_4]$.  By Lemma \ref{intpq} we see $P \not \in \cU$, so the intersection $\cU \cap H_{v^*_3 \wedge v^*_4}$ is transverse, hence it is a smoothly embedded codimension one  submanifold. 

The intersection $\cU \cap H_{v^*_3 \wedge v^*_4}$ is determined in Pl\"ucker coordinates by 
\begin{eqnarray*}
 \cU \cap H_{v^*_3 \wedge v^*_4} &=&  G   \cap \left( \{p_{14} -p_{23} =0\} \cap  \{ p_{12}p_{34} \geq 0\}  \right)^c   \cap \{p_{34}=0\}  \\
 & =& G  \cap  \{p_{34}=0\} \cap  \{p_{14} -p_{23} =0\} ^c  \\
 & = &  (G \setminus H_{\omega_J}) \cap H_{v^*_3 \wedge v^*_4}
 \end{eqnarray*}
where we have applied de Morgan's law and the fact that $ \{p_{34}=0\} \subseteq \{ p_{12}p_{34} \geq 0\}$. Therefore, the normal bundle of $\cU \cap H_{v^*_3 \wedge v^*_4}$ in $\cU$ is the pullback of the normal bundle of the affine space $ \left( P( \bigwedge^2(V)) \setminus H_{\omega_J} \right) \cap H_{v^*_3 \wedge v^*_4} \cong \R^4$ in the affine space $P( \bigwedge^2(V))
\setminus H_{\omega_J} \cong \R^5$. But this is clearly co-orientable, so we are done.
\end{proof}

\begin{rem}
One might expect, based on the above argument, that since the linear inclusion $\R^4 \subseteq \R^5$ has a trivial Poincar\'e dual in $H^1(\R^5) \cong \{0\}$, the same must be true of $\cU \cap H_{v^*_3 \wedge v^*_4}$ in $\cU$. However, since $\cU$ is not a subset of $P(\bigwedge^2(V) \setminus H_{\omega_J} \cong \R^5$, there is no natural map in cohomology from $H^1(\R^5)$ to
$H^1(\cU)$.
\end{rem}

\begin{cor}\label{NUP}
The open set $\cU \subseteq G$ is a Maslov--Arnold space with respect to $P$.
\end{cor}

\begin{proof}
Let $N := \Lambda_I \setminus \{P,Q\}$.  By Lemma \ref{intpq} we know $N = \cU \cap \Lambda_I$.  Since $\Lambda_I$ is a 3-manifold and $N$ is the complement of two isolated points in $\Lambda_I$, the inclusion determines an isomorphism $H^1(N;\Z) \cong H^1(\Lambda_I;\Z) \cong \Z$, which is generated in both cases by the Poincar\'e dual of the train of $P$ (with a chosen co-orientation).   

It follows from Lemma \ref{lem cU} that the train $\cZ_{_P}  \cap \cU$, equipped with a chosen co-orientation, represents a well-defined cohomology class in $H^1(\cU;\Z) \cong H_3^{BM}(\cU;\Z)$. This cohomology class must have infinite order, because it is sent to a generator of $H^1(N;\Z)$ under restriction to $N \subseteq \cU$.
\end{proof}

We now define the \emph{Fat Lagrangian Grassmannian}
\begin{align}
	\cF:= \cU \cup \Lambda(2) = \cU \cup \{P,Q\}.
\end{align}
Note that $\cF$ is not a manifold. However, it is a semialgebraic set, since $\cU$ is defined by polynomial inequalities. 

Consider the coordinate neighbourhood of $P \in G$ by $$U_P = \{ gr(A) : A \in \Hom(P,Q)\} % = \{ gr(JA) : A \in \Hom(P,P)\} \cong \Hom(P,P)
$$ consisting of all $2$-planes that intersect $Q$ trivially, and hence can be realized as graphs of linear maps from $P$ to $Q$. Denote by $J\colon P \rightarrow Q$ the complex structure with $J(v_3) = - v_2$ and $J(v_4) = v_1$. As in the proof of Theorem \ref{thmcont}, we have
$$U_P = \{ gr(JA) : A \in \Hom(P,P)\} \cong \Hom(P,P).$$ 
Using the matrix representation with respect to the basis $\{v_3,v_4\}$ of $P$ determines a coordinate chart 
$$U_P \cong \R^4  = \left\{  A = \begin{pmatrix} 
x & y \\
z & w 
\end{pmatrix}  :  x,y,z,w \in \R \right\}. $$  

Under this identification
\begin{eqnarray*}
\Lambda_J \cap U_P &=& \{ gr(JA) : A = A^T \},\\
H_{v^*_3 \wedge v^*_4} \cap U_P &=& \{ gr(JA) : \det(A) = 0 \}.
\end{eqnarray*}
Similarly, we have a coordinate neighbourhood of $Q \in G$,  $$U_Q = \{gr(A) : A \in \Hom(Q,P)\} = \{ gr(AJ^{-1}) : A \in \Hom(P,P)\} \cong \R^4.$$

\begin{lemma}\label{homeqs2}
The spaces  $U_Q \setminus S$ and $U_P \setminus S$ are both homeomorphic to $ \R \times ( \R^3 \setminus \{0\})$, and are therefore homotopy equivalent to $S^2$. 
\end{lemma}

\begin{proof}
Under the identification $U_P \cong \Hom(P,P) \cong \R^4$, the intersection $S \cap U_P$ is defined by the equations $\tr(A) = x+w = 0$ and $\det(A) = xw -yz \geq 0$. These describe a solid, closed double cone in the three-dimensional subspace $\{x+w=0\}$. The complement $U_P \setminus S$ is therefore invariant under multiplication by the positive scalar $\R_+$ and intersects the
unit sphere $S^3$ in the complement of two closed $2$-disks, which is diffeomorphic to $\R^3 \setminus \{0\}$. The case $U_Q \setminus S$ is similar.
\end{proof}

\begin{prop}\label{FLGMaslov}
The inclusion $\cU \subseteq \cF$ defines an isomorphism $H^1(\cF;\Z) \cong H^1(\cU;\Z)$.  Consequently, $\cF$ is an MA space that extends $\Lambda_I$ and is dense in $G$.
\end{prop} 

\begin{proof}
By definition $\cU = \cF \setminus \{P,Q\}$. Let $\cU'$ be the union of two small open balls around $P$ and $Q$ in $U_P$ and $U_Q$, respectively, intersected with $\cF$. From the local picture described in the proof of Lemma \ref{homeqs2}, it is clear that $\cU'$ deformation retracts onto the two point set $\{Q, P\}$ and that $\cU\cap \cU'$ deformation retracts onto $S^2 \coprod
S^2$. The isomorphism follows from the Mayer--Vietoris long exact sequence
$$  H^0(\cU) \oplus H^0(\cU')  \twoheadrightarrow H^0(\cU \cap \cU') \rightarrow H^1(\cF) \rightarrow H^1(\cU)  \oplus H^1(\cU') \rightarrow H^1(\cU \cap \cU') $$
since $H^1(\cU') \cong H^1(\cU \cap \cU') \cong \{0\}$ and $H^0(\cU) \oplus H^0(\cU')  \twoheadrightarrow H^0(\cU \cap \cU')$ is surjective.

Any sufficiently generic loop $\gamma\colon S^1 \rightarrow \cF$ is contained in $\cU$, so $\cF$ is an MA space extending $\cU$. Finally, following the proof of Corollary \ref{NUP}, subspace inclusions determine a commuting diagram of isomorphisms
$$ \xymatrix{  H^1(\cF;\Z) \ar[r]^{\cong} \ar[d]^{\cong}&  H^1(\Lambda_I, \Z) \ar[d]^{\cong} \\  H^1(\cU;\Z) \ar[r]^\cong & H^1(N;\Z)  }$$ 
so $\cF$ also extends $\Lambda_I$. 
\end{proof}

\subsection{Non-existence of smooth extensions}\label{Proof of Theorem}

We now prove Theorem~\ref{thmcont}, on the non-existence of smooth extensions of the classical Maslov--Arnold spaces. 

\begin{proof}[Proof of Theorem \ref{thmcont}]
Suppose that there exists a proper extension $\Lambda(n) \subsetneq \cM$ for which $\cM$ is a connected, smoothly embedded submanifold of $Gr_n(\R^{2n})$. This implies $\dim \cM > \dim \Lambda(n)$. Using these extra degrees of freedom, we will construct a sufficiently generic loop in $\cM$ that is contractible but has a non-zero geometric intersection number with the train (see Section \ref{sec:def}), producing a contradiction. 

Our construction takes place within a coordinate neighbourhood of $Gr_n(\R^{2n})$, wherein the classical Maslov index may be interpreted as a spectral flow, as described by Robbin and Salamon \cite{RS93}. Equip $\R^{2n}$ with the standard inner product product $\left<\cdot,\cdot \right>$ and define a complex structure $J\colon \R^{2n} \rightarrow \R^{2n}$ by $\omega(v,w) = \left<v, J w \right>$. The Lagrangian subspace $P$ has a Lagrangian complement $Q := J(P) = P^\perp$, so $\R^{2n} = Q \oplus P$. We define the coordinate neighbourhood $ U_P \subseteq Gr_n(\bbR^{2n})$ to be the set of
$n$-dimensional subspaces in $\R^{2n}$ that intersect $Q$ trivially and can therefore be represented as graphs of linear maps from $P$ to $Q$. We have an identification
\begin{equation}\label{corhood}
  U_P =  \{ gr(B) :  B \in \Hom(P,Q) \}  =  \{ gr(JA) : A \in \Hom(P,P) \} \cong  \Hom(P,P),
\end{equation}
where we have abused notation and denoted by $J\colon P \rightarrow Q$ the restriction of $J$. In this coordinate neighbourhood we have
\begin{eqnarray}\label{locallambda}
\Lambda(n) \cap U_P &\cong&  \{ A \in \Hom(P,P) : A = A^T\},\\ \cZ \cap U_P &\cong& \{ A \in \Hom(P,P) : \det A =0\}. \label{localz}
\end{eqnarray} 
The co-orientation of the train in $\Lambda(n)$ is such that under the identification (\ref{corhood}), the index of a path $\gamma\colon [0,1] \rightarrow U_P \cap \Lambda(n)$ counts the difference between the number of positive eigenvalues of the symmetric matrices $\gamma(1)$ and $\gamma(0)$. That is, the Maslov index of $\gamma$ equals the \emph{spectral flow} of the corresponding
family of symmetric matrices; see \cite[Theorem 2.3]{RS93}.

While $\Lambda(n)$ and $\cZ$ have a simple description in the coordinate chart $U_P$, the same is not true of the supposed extension $\cM$; we only know that it contains $\Lambda(n)$ and has strictly larger dimension. Therefore, we will give our construction in the tangent space $T_P\cM$, then transfer it to $\cM$ using the exponential map of a suitably chosen Riemannian metric.

Let us regard the tangent space $T_P\cM$ as a subspace of $T_0 \Hom(P,P)= \Hom(P,P)$. Since $T_P\cM$ must strictly contain $T_P\Lambda(n)$, i.e. the subspace of symmetric operators, it must also contain some non-zero $B \in \Hom(P,P)$ with $B = -B^T$. Such a $B$ is diagonalizable, and must have at least one non-zero, purely imaginary eigenvalue $\lambda_0$, with eigenvector $v_0$. 
Let $\Pi$ be the orthogonal projection onto $\spn\{v_0\}$ and let $\Pi':= I_P - \Pi$. The paths in $T_P\cM$ defined by
\begin{align*}
	A_1(t) := \Pi' + \cos(\pi t) \Pi, \qquad A_2(t) := \Pi' + \cos(\pi t) \Pi + \sin(\pi t) B
\end{align*}
for $t \in [0,1]$ have the same endpoints, $I_P$ and $I_P - 2\Pi$, which are both symmetric and have different numbers of positive eigenvalues ($n$ and $n-1$ respectively). We claim that $A_2(t)$ is non-degenerate for all $t \in [0,1]$.
Assuming $A_2(t)v = 0$ for some non-zero $v \in P$, we obtain
\[
    \cos(\pi t) \Pi v + \sin(\pi t)B(\Pi v) = 
    \Pi' v + \sin(\pi t)B(\Pi' v) = 0,
\]
using the fact that $\Pi$ and $\Pi'$ commute with $B$. If $\sin(\pi t)=0$, these equations imply $\Pi v = \Pi' v = 0$, a contradiction. If $\sin(\pi t)\neq0$, the second equation implies $\Pi' v=0$, because $B$ does not have any non-zero real eigenvalues. Then $\Pi v = v$, and hence $Bv = \lambda_0 v$, but the first equation reduces to $Bv = -\cot(\pi t) v$, which is a contradiction because $\lambda_0$ is not real.

To obtain the desired loop in $\cM$, we choose a flat, translation invariant Riemannian metric on $\Hom(P,P) \cong \R^{n^2}$, so that the geodesic through $0 \in \Hom(P,P)$ with initial velocity $A \in T_0\Hom(P,P)$ is just the line $\gamma_A(t) = tA$. Restricting this metric to $\cM$, we obtain an exponential map $\exp\colon V_P \rightarrow \cM \cap U_P$, where $V_P$ is a open ball around $0 \in T_P\cM$ and can be chosen small enough to ensure $\exp$ is a diffeomorphism from $V_P$ onto its image. Identifying $\cM \cap U_P$ with a subset of $\Hom(P,P)$, we see that any line contained in $\cM \cap U_P$ is a geodesic\footnote{More generally, any geodesic in $\Hom(P,P)$ that is contained in $\cM \cap U_P$ will also be a geodesic with respect to the induced metric on $\cM \cap U_P$. However, there may be geodesics in $\cM \cap U_P$ that are not geodesics in $\Hom(P,P)$, as there is no reason to expect that $\cM \cap U_P \subseteq \Hom(P,P)$ is a totally geodesic submanifold.}
in $\cM$, because it locally minimizes length in $U_P$, and hence also in $\cM$. Therefore, if $A \in T_P\cM$ is symmetric, the line $\gamma_A(t)= tA$ is a geodesic in $\cM \cap U_P$, which means $\exp(A) = \gamma_A(1) = A$ is symmetric, and hence corresponds to a Lagrangian subspace via the identification \eqref{locallambda}.

Choosing $\epsilon>0$ small enough that $\epsilon A_1(t)$ and $\epsilon A_2(t)$ are in $V_P$ for all $t \in [0,1]$, we define smooth paths
\begin{align*}
    \gamma_1(t) := \exp(\epsilon A_1(t)) = \epsilon A_1(t), \qquad 
    \gamma_2(t) := \exp(\epsilon A_2(t)) ,
\end{align*}
and then let $\gamma$ denote the concatenation of $\gamma_1$ with the reverse of $\gamma_2$. Since $\epsilon A_1$ is homotopic to $\epsilon A_2$ in $V_P$ and $\exp$ is a local diffeomorphism, we immediately see that $\gamma$ is contractible, so its generalized Maslov index vanishes, $\Mas(\gamma;P) = 0$.

However, recalling that we have identified $\cM \cap U_P$ with a subset of $\Hom(P,P)$, Taylor's theorem gives the uniform estimate $\exp(A) = A + \mathcal{O}(|A|^2)$ for small $A \in T_P \cM$. This implies the loop $\gamma$ is sufficiently generic for small enough $\epsilon>0$, and $\gamma_2$ does not intersect the train $\cZ$ (because $A_2$ is always non-degenerate). Definition~\ref{def:MA1} then says that $\Mas(\gamma;P)$ equals the geometric intersection number of $\gamma$ (and hence of $\gamma_1$) with $\cZ$. On the other hand, since the matrices $A_1(t)$ are symmetric, $\gamma_1(t)$ is a path in the Lagrangian Grassmannian, and its geometric intersection number with the train equals its Maslov index. By the spectral flow interpretation of the Maslov index, this is non-zero, because $\gamma_1(0) = \epsilon A_1(0)$ and $\gamma_2(0) = \epsilon A_2(0)$ have different numbers of positive eigenvalues. This contradiction finishes the proof.
\end{proof}

\subsection*{Acknowledgements}
The authors acknowledge the support of the BIRS FRG program \emph{Stability Indices for Nonlinear Waves and Patterns in Many Space Dimensions}, where some of this work was done. We also thank the anonymous referees for their helpful comments. T.B. acknowledges support of NSERC grant RGPIN-2016-05382. G.C. acknowledges the support of NSERC grant RGPIN-2017-04259. C.J. was supported by ONR grant N00014-18-1-2204. R.M. acknowledges the support of the ARC under grant DP200102130.

\bibliographystyle{amsplain}
\bibliography{nonHam}

\end{document}